\documentclass[11pt,leqno,oneside]{article}
\usepackage{commands}
\usepackage[affil-it]{authblk}

\UseEnEnv
\NOREDBOX

\title{Two-Scale Geometric Modelling for Defective Media}

\author{Mewen Crespo${}^*$}
\author{Guy Casale}
\author{Loïc Le Marrec}
\affil{IRMAR – université de Rennes 1, campus de Beaulieu, 35042 Rennes cedex, France}

\bibliography{main}

\begin{document}
\tcbstartrecording[proofs.tex]

\maketitle

\renewcommand{\phantomInput}[1]{}

\begin{abstract}
	A new geometrically exact micro-structured model is constructed using a generalisation of the notion of Riemann-Cartan manifolds and fibre bundle theory of rank $3$. This models is based around the concept of two different length scales: a macroscopic scale $-$ of dimensions $1$, $2$, or $3$ $-$ and a microscopic one $-$ of dimension $3$. As they interact with each other, they produce emergent behaviours such as dislocations (torsion) and disclinations (curvature). A first-order placement map $\Fb : \Tr\B \longto \Tr\E$ between a micro-structured body $\B$ and the micro-structured ambient space $\E$ is constructed, allowing to pull the ambient Riemann-Cartan geometry back onto the body. In order to allow for curvature to arise, $\Fb$ is, in general, not required to be a gradient. Central to this model is the new notion of pseudo-metric, providing, in addition to a macroscopic metric (the usual Cauchy-Green tensor) and a microscopic metric, a notion of coupling between the microscopic and macroscopic realms. A notion of frame indifference is formalised and invariants are computed. In the case of a micro-linear structure, it is shown that the data of these invariants is equivalent to the data of the pseudo-metric.
\end{abstract}

\section{Introduction}

\subsection{Motivations and historical context}

In continuum mechanics, one wishes to represent a physical medium though a continuous theory. That is, a theory making use of only real connected varieties and, as much as possible, smooth functions. Real matter, however, has little interest in our wishes and is, at a small scale, fundamentally discrete (e.g. atoms, molecules, biological cells, structural cells of an object, etc.). For this reason, a continuous model valid and precise for several scales is hard to reach. In order to fill the gap between the macroscopic realm $-$ which classical continuum mechanics\footnote{In this paper, the term \say{classical continuum mechanics} and \say{usual continuum mechanics} will refer to Cauchy's continuum theory where a configuration is a $\C^1$ embedding of the body into a connected sub-manifold of the $3D$ Euclidean space.} is good at describing $-$ and the microscopic\footnote{Some authors prefer different names such as \say{mesoscopic} depending on the actual size or specific nature of the smaller scale. In this paper, \say{microscopic} will be used for any scale smaller that the macroscopic one but still large enough so it can be considered as a continuum.} realm, new models extending the classical model $-$ referred to as micro-structured models $-$ are required. In the simplest case, when the micro-structure behaves ideally and emulates the (continuous) macroscopic structure, the classical models are still valid. Hence, those classical models cease to be adequate only when the behaviour of the microstructure and the macrostructure start to differ. This is the case in a wide variety of materials such as crystals, polymers, solids with micro-cracks, bones, muscles, etc. and, more generally, in the study of plasticity \parencite{le1996model}. Such a deviation from an ideal state is referred to in the literature as a defect of the material. Instead of ideal state one therefore usually prefer the term defect-free state, whose notion is theoretically as arbitrary as a reference configuration but which, in practice, is often defined as a perfectly ordered micro-structure \parencite{kroner2001benefits,kroner1980continuum}. The study of defect in media is therefore crucial to the elaboration of good micro-structured models.\\

In \citeyear{volterra1907equilibre}, \citeauthor{volterra1907equilibre} published \parencite{volterra1907equilibre} one of the first detailed description of defects in crystals, whose micro-structure is a $3\Dr$ discrete lattice. This led to the so-called Volterra process \parencite{eshelby1956continuum, sahoo1984elastic}. This process states that all crystals can be obtained by starting from an ideal defect-free crystal $-$ modelled as a classical $3$-dimensional continuum $-$ and adding a finite number of defects in the following way: choose a surface $S$, bounded by a curve $\delta S$, in the medium and an affine relative displacement $\delta \ub$ for each pair of points on either side of $S$. Then, apply the relative displacement, removing material where there would be interpenetration and adding material where there would be a gap. Since $\delta \ub$ is affine in the position, it can be decomposed into a translation $\ov b$ and a rotation $\ov\omega$ (seen as its axis vector). The vector $\ov b$ is called the Burgers vector $-$ as first introduced in the work of the Burgers brothers on crystal lattices \parencite{burgers1956dislocations} $-$ and yields a dislocation; the vector $\ov\omega$ on the other hand is referred to as the Frank vector $-$ as introduced by \citeauthor{frank1958liquid} in his work on liquid crystals \parencite{frank1958liquid} $-$ and yields what is called a disclination\footnote{\Citeauthor{frank1958liquid} originally used the term \say{disinclination} in \cite{frank1958liquid}, to which we shall prefer the term disclination, more widely used nowadays.}.\\

\VOID{\begin{figure}
	\centering
	\includegraphics{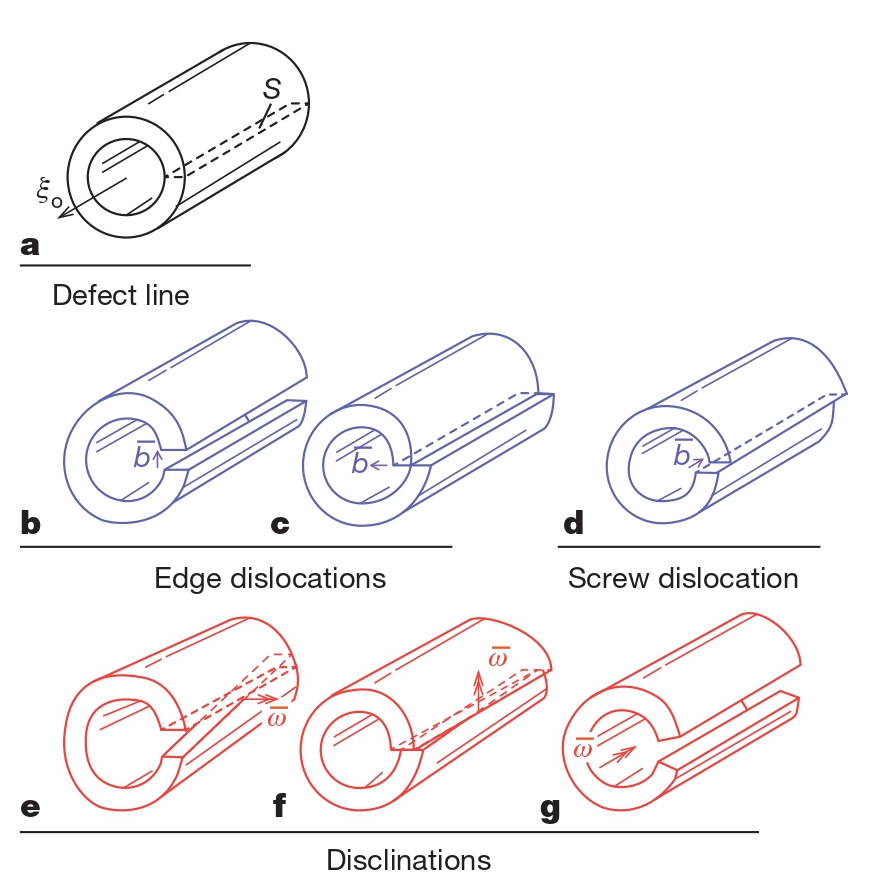}
	\caption{\textbf{Volterra's process.} $\ab$, Reference cylinder with defect line $\xi_0$ and
cut surface $S$. $\bb$, $\cb$, Edge dislocations with Burgers vector $\ov{b}$. $\db$, Screw dislocation
with Burgers vector $\ov{b}$. $\eb$, $\fb$, Twist disclinations with Frank vector $\ov\omega$. $\gb$, Wedge
disclination with Frank vector $\ov\omega$.\\
{\scriptsize Source: \citeauthor*{cordier2014disclinations}, \citeyear{cordier2014disclinations}, \citetitle{cordier2014disclinations} \cite{cordier2014disclinations}}}
	\label{fig:volterra}
\end{figure}}

Those works led to numerous results in the last century where micro-structured models are obtained by lessening classical models' regularities. This may be in the form of discontinuities $-$ as in the works of \fullciteauthor{eshelby1956continuum, sahoo1984elastic} allowing jumps at dislocation lines $-$ or in the form of non-integrability $-$ as in the now standard works of \fullciteauthors{le1996model}{le1994finite,le1996determination} multiplicatively decomposing the first-order transformation as an elastic (defect-free) and a plastic (carrying the defects) transformation. Both approaches are linked, as illustrated by \citeauthor*{reina2016derivation} in \cite{reina2016derivation}.\\

Shortly after Volterra, in a paper published in 1909, \fullciteauthor{cosserat1909theorie} took a different route by introducing one of the first generalized continua, adding three rotational degrees of freedom at each point. In addition to the standard displacement field used in a classical continuum mechanics, an independent unit vector field or, equivalently, a field of $3\Dr$ Euclidean rotations is additionally specified on the body. More than a decade later \fullciteauthor{cartan1922generalisation}, inspired by this work, constructed the so-called Riemann-Cartan manifold. This manifold has a Riemannian metric but also makes one of the first uses of an affine connection inducing torsion and curvature.\\

In this paper, E. Cartan introduces the notion of \say{trièdres rectangles} \partxt{\ie{} orthogonal frames}, adding that additional rotational degrees of freedom, and deeply talks about the important interpretation of the lack of closure of parallel transported paths along an infinitesimal loop. More precisely, when parallel transporting such a frame along a close macroscopic loop, the resulting frame may be slightly rotated and/or translated compared to the initial frame. Mathematically this leads to the notion of torsion and curvature on which Cartan focuses in \cite{cartan1922generalisation}. In the case of crystals, identifying this frames with the principal directions of the crystal lattice leads to a fundamental result: the correspondence between this torsion (resp. curvature) and the aforementioned notion of dislocation (resp. disclination) \cite[293]{kroner1980continuum}, \cite[791]{sahoo1984elastic}, \cite[12]{peshkov2019continuum}, \cite[613]{le1996model}.\\

%Epstein \cite{epstein2010geometrical, epstein2014geometric, epstein2014differential} generalised Cosserat and Cartan's continua by using jets theory and group theory, introducing the so-called G-structures.

These pioneer works led to multiple other works, where micro-structured models are created by enriching classical models with additional degrees of freedom. Among those, the models derived by \fullciteauthors{eringen1964nonlinear}{eringen1998microcontinuum}, \fullciteauthor{mindlin1964micro} and \fullciteauthors{toupin1964theories}{toupin1962elastic}, made explicit Cartan's interpretation of the generalised continuum as a macroscopic continuum where each point contains a microscopic space. This was done using two set of coordinates, one macroscopic and one microscopic. Having real coordinates, this allows for an easier physical interpretation, at the cost of not being intrinsic (i.e. covariant) out of-the-box. Those fundamental works in turn led to further developments and specialisations of the models, which are still ongoing nowadays \partxt{\textit{e.g.} \cite{neff2014unifying,grammenoudis2009micromorphic,peshkov2019continuum}}.\\

\subsection{Approach in this work}

In this article, a family of generalized continua with their strain measures is introduced, with an emphasis on explicit definition of the geometry of the system. Particular care has been given to the physical interpretation of each mathematical object and hypothesis, with the set of hypotheses being kept as small as possible.\\

In \cref{sec:geom_structure} the geometrical tools are introduced, with on emphasis on the physical interpretation not always highlighted in the standard literature. Following the usual interpretation of micro-structured continua, the space is modelled as a generic affine bundles $\A$ over a macroscopic space $\AA$ with $\dim\left (\AA\right ) \in \left \{1,2,3\right \}$. In \cref{sect:connection}, Cartan's notion of connection is generalised into the notion of Ehresmann connection \parencite{epstein2010geometrical}, seen as a lifting operator allowing to identify large (\ie{} macroscopic) classical vectors on $\AA$ with vectors on the generalised continua $\A$. In an analogous fashion, the notion of solder form is introduced in \cref{sect:solder}, allowing to identify small (\ie{} microscopic) classical vectors on $\AA$ with vectors on the generalised continua $\A$. These interpretations of the connection and solder form as macroscopic and microscopic identifications are expanded upon in \cref{sect:t_eq_h_plus_v}. In particular, a solder form encodes a notion of scale of the microstructure via its eigenvalues and is therefore analogous to the order parameter used in several theories \cite{reina2016derivation,kroner2001benefits,clayton2017finsler}. In \cref{sect:pseudo_metr}, the new notion of pseudo-metric $-$ a relaxation of the notion of metric with a possibly non-trivial kernel $-$ is introduced.\\

Following a thorough discussion on the physical interpretation of each geometrical object available, a notion of compatibility of the pseudo-metric is introduced in \cref{def:compab}. Physically, the later comes directly from the interpretation of the microscopic space $\A_\ov a$ at $\ov a \in \AA$ as the set of infinitesimally close points $\ov a + \delta \ov a$. Most importantly, an obstruction to the existence of a compatible metric (\ie{} with a trivial kernel) on the whole continuum $\A$ is proven in \cref{lemm:kergf}. In particular, macroscopic vectors $\der{\ov a}$ and microscopic vectors $\der{\delta\ov a}$ are not orthogonal. This is a key fact as this non-orthogonality contains data on the coupling of macroscopic and microscopic structures, as stated in \cref{lemm:kergf,lem:Gf_compat}. The pseudo-metric is to be compared with the notion of Sasaki metric in \cite{clayton2017finsler, nguyen2021geometric, hushmandi2012curvature}. Although a Sasaki metric, by construction, renders the microscopic and macroscopic spaces orthogonal and is therefore not compatible, its restriction on each of those yields a metric which coincides with the compatible pseudo-metric. This micro-metric is also analogous to the notion of Finslerian metric \cite[198]{amari1962theory}, \cite[121]{clayton2017finsler} as it is a $3$-dimensional metric depending on two sets of coordinates.\\

In \cref{sect:conf_pull_back}, following the classical literature \parencite{epstein2010geometrical, GON}, the generic space $\bundle{\A}{\AA}$ is replaced by a a material space $\bundle\B\BB$ and an ambient space $\bundle\E\EE$. A generalisation of the classical macroscopic placement maps $\ov\varphi : \BB \longto \EE$ and $\ov\Fb = \Tr\ov\varphi : \Tr\BB \longto \Tr\EE$ is then sought in order to pull-back the ambient geometry of $\E$ onto $\B$. In \cref{sect:punctual_map}, based on a physical interpretation of the continuum, the punctual map $\ov\varphi : \BB \longto \EE$ is generalised into an \textbf{affine} bundle morphism:
$$\varphi : \fun\B\E{\bmat{\ov X\\ Y}}{\bmat{\ov\varphi\left (\ov X\right )\\\overrightarrow{\varphi^\vr}\left (\ov X\right )\cdot Y + t^\vr\left (\ov X\right )}}$$
This form should be compared with the theory of micromorphic media \parencite{eringen1964nonlinear,eringen1998microcontinuum}, where the transformation is modelled as a \textbf{linear} morphism. Although the micro-spaces are interpreted as infinitesimal neighbourhoods, the microscopic part of $\varphi$ is not required to be the gradient of the macroscopic part. This freedom is necessary in order to allow for torsion (\ie{} dislocations) in the material \cite[8]{peshkov2019continuum}.\\

In \cref{sect:first_order_map}, further discussion on the physical interpretation of a first-order map leads to a generalisation of $\ov \Fb = \Tr \ov \varphi : \Tr\BB \longto \Tr\EE$ as a linear bundle morphism $\Fb \equiv \bmat{\Tr\ov\varphi&\bold 0\\\Fb_\hr^\vr & \overrightarrow{\varphi^\vr}}$ over $\varphi$. This is analogous to the work of \cite{eringen1964nonlinear,eringen1998microcontinuum,mindlin1964micro} with the fundamental difference that $\Fb$ does not have to be the gradient of $\varphi$, only its restrictions on the macroscopic and microscopic spaces do. In \cref{sect:pull_back}, the generalised first-order placement map $\Fb$ is used to pull-back the ambient connection $\gammab$, solder form $\varthetab$ and pseudo-metric $\gf$ onto $\B$, giving $\Gammab$, $\Thetab$ and $\Gf$ respectively. Crucially, and in contrast to several works \cite{kroner1980continuum, peshkov2019continuum, sahoo1984elastic, le1996determination}, the relaxed form of $\Fb$ allows the material connection to bear some curvature, \ie{} disclinations.\\

In \cref{frame_inv_intro}, a generalisation of the notion of frame indifference is formulated, using the notion of generalised Galilean group $\Galf$. In the following section, the group $\Galf$ is expressed as the stabiliser of a set of tensors, paving the way for the main result of this article: the computation of the tensorial invariants of the first-order configuration under the group $\Galf$ \partxt{\cref{th:inv_orb}} which directly give the strain measures on which any frame indifferent function must depend \partxt{\cref{lemm_fact_fram_inv}}. Those invariants are the material micro-metric, connection, solder form and "holonomic" connection. Crucially, in the holonomic case, where $\Fb=\Tr\varphi$, both connections are identical and those invariants correspond to the three invariants computed in \cite{eringen1964nonlinear}, \cite[15]{eringen1998microcontinuum}, \cite{grammenoudis2009micromorphic}. \Cref{sect:disc} contains discussions on the properties of those invariants, their interpretation and a conclusion. Furthermore, these invariants are also similar to those computed in \cite{le1994finite}, which are the elastic metric \parencite[613]{le1996model} and the torsion of the connection.\\

The role of the newly introduced pseudo-metric $\Gf$ is central to this model and is analogous to the role of the Cauchy-Green metric in classical models. \Cref{lemm:kergf,lem:Gf_compat} establish an equality between the kernel of $\Gf$ and the horizontal space of the no-slip connection $\Gammab-\Thetab$ \parencite[7]{nguyen2021tangent}. As a consequence, \cref{th:data_equiv} implies that most of the information is stored in this pseudo-metric. In particular, the compatible pseudo-metric prescribes the notion of lengths and angles (macroscopic, microscopic or mixed alike) and contains the sum of a torsion and a curvature \parencite[8]{nguyen2021tangent}. \Cref{sect:micro-lin} extends those results in the special case of micro-linear material. That is, materials where all micro-spaces have centres which are preserved by the placement map. In such a case, the connection $\Gammab$ is linear and the solder-form $\Thetab$ is uniform on the micro-spaces. As a consequence $\Gammab$ and $\Thetab$ can canonically be extracted from the no-slip connection $\Gammab-\Thetab$ (by taking the linear part). This means that, in this case, the torsion and curvature can be obtained from the pseudo-metric alone. Finally, and perhaps most importantly, \cref{coro:inv_orb} concludes this analysis by stating that in such micro-linear materials, any frame-invariant functional of $\Fb$ $-$ and in particular, the energy $-$ can be expressed in term of the pseudo-metric $\Gf$ alone. This is the case in the works of \fullciteauthors{toupin1964theories,mindlin1964micro,eringen1964nonlinear}{eringen1998microcontinuum,toupin1962elastic}. This last result should be compared to its classical analogue stating that every frame-invariant functional of the macroscopic placement $\ov\Fb$ is expressible as a function of the Cauchy-Green metric $\Gb$ \partxt{which is obtainable from $\Gf$} \parencite[275,283]{GON}.

\section{Geometric structure}

\label{sec:geom_structure}
\label{cartan_vision}

Adopting the vision of \fullciteauthors{cartan1922generalisation,mindlin1964micro,eringen1964nonlinear}{eringen1998microcontinuum}, points of the macroscopic space are endowed with a field of microscopic spaces (one for each macroscopic point). The total space is therefore of dimension $n+k$ $-$ where $n$ is the macroscopic dimension and $k$ the microscopic dimension $-$ and is equipped with a projection onto the macroscopic space of dimension $n$. The micro-spaces are then the set of points sharing a common macroscopic projection. Mathematically, this translates into the notion of projection structure:\\

\begin{defi}{Projection structure}{proj}
	A \underline{projection structure} is the data of:
	\begin{enumerate}
		\it a set $\A$, called the \underline{total space}
		\it a set $\AA$, called the \underline{base space}
		\it a surjective map $\pi_\A : \A \longto \AA$, called the \underline{projection map}
	\end{enumerate}
	
	One then says $\A$ is equipped with a projection structure on $\A$ over $\AA$, synthesised as $\bundle{\A}{\AA}$. Furthermore, if the projection structure on $\A$ is unambiguous then one can also use the following notations and terminologies:
	\begin{enumerate}
		\it	$\overline{\A}$ for its corresponding base space (named $\AA$ in the definition)
		\it $\forall a \in \A,\hfill\overline{a} := \pi_\A(a) \in \overline{\A}\hfill\text{called the projection of $a$}\hfill$
		\it $\forall \overline{a} \in \overline{\A},\hfill\A_\overline{a} := \pi_\A^{-1}\left (\left \{\overline{a}\right \}\right ) := \setof{a\in\A}{\pi_\A(a)=\overline{a}}\hfill\text{called the \underline{fibre at $\overline{a}$}}\hfill$\\
		\it $\forall \UU \subset \AA,\hfill \at{\A}_\UU = \pi_\A^{-1}\left (\UU\right )\hfill\text{called the \underline{restriction} of $\A$ over $\UU$}$\\
		\it any $\sigma : \AA \longto \A$ which is a right inverse of the projection $-$ that is, $\pi_\A\circ\sigma = \Id$ $-$ will be called \hbox{a \underline{section} of $\A$}.
	\end{enumerate}
	
	%More informally, when the name of the space is a single letter we will use a calligraphic font for the total space \partxt{\textit{e.g.} $\A$} and a bold font for the base space: \partxt{\textit{e.g.} $\AA:=\overline{\A}$}. Also, we will reserve the name $\pi_\A$ for the canonical projection structure, when there is one.\\
\end{defi}

The total space (of dimension $n+k$) is therefore endowed with a projection structure over the macroscopic space (of dimension $n$). The projection can then be seen as providing the macroscopic part of a point. However, the space is not solely a set but also a smooth variety where one can differentiate objects. One therefore wants the projection structure to be smooth. Furthermore, one needs to be able to have, at least locally, a coordinate system splitting the macroscopic and microscopic parts. This is linked to Cartan's notion of (microscopic) moving frames in \cite{cartan1922generalisation}. Mathematically, these two conditions lead to the notion of affine bundles.

\subsection{Affine bundle}

An affine bundle can be seen, heuristically, as a projection structure whose total space is locally diffeomorphic to a Cartesian product of a base and an affine fibre in a way that preserves the affine structure. Formally, the definition of an affine bundle is a bit more involved and goes as follows:\\

\begin{defi}{Affine bundle}{vector_bundle}
	An \underline{affine bundle} (resp. vector bundle) is the data of:
	\begin{enumerate}
		\it a projection structure $\bundle{\A}{\AA}$ where $\A$ and $\AA$ are both \textbf{smooth\footnote{In this paper, smooth will mean at least continuously differentiable,\ie{} $\C^1$.} connected orientable} real manifolds and $\pi_\A$ is a smooth map.
		\it a real affine (resp. vector) space $\F_\A$, defined up to an affine (resp. linear) automorphism, called the \underline{typical fibre}.
		\it a family of diffeomorphisms $\left (\Psib_\overline{a}\right )_{\ov{a}\in\AA}$, called the \underline{local trivialisations}, such that:
			\quant{\forall \overline{a} \in \AA,~\exists~\UU_\overline{a}~\text{neighbourhood of}~\overline{a},}{&\Psib_\overline{a} : \at{\A}_{\UU_\overline{a}}\longmapsto\UU_\overline{a}\times\F_\A\\
										   &\text{and} \hspace{2em} \pi_\A=\mu_\overline{a}\cdot\Psib_\overline{a} \hspace{1em}\text{on}\hspace{1em}\at{\A}_{\UU_\overline{a}}}
			where $\mu_\overline{a} : \fun{\UU_\overline{a}\times\F_\A}{\UU_\overline{a}}{(\overline{b}, y)}{\overline{b}}$ is the canonical left projection of $\UU_\overline{a}\times\F_\A$ onto $\UU_\overline{a}$.
		\it by construction, the transition map $\Psib_\overline{a}\circ\Psib_\overline{b}^{-1}$ \partxt{defined only when $\UU_\overline{a}\cap\UU_\overline{b}\neq\emptyset$} induces a diffeomorphism of $\F_\A$ onto itself, called a vertical change of frame. The last property is that these diffeomorphisms must be \textbf{affine} (resp. linear).
	\end{enumerate}
	
	The set of all vertical changes of frames generates a group $\G_\A$ for the composition, acting on $\F_\A$ in an affine (resp. linear) way. That is, $\G_\A$ is a sub-group of $\Aff\left (\F_\A\right )$ \partxt{resp. $\GL\left (\F_\A\right )$}. The group $\G_\A$ is called the \underline{structure group} of the affine (resp. vector) bundle $\A$.\\
	
\end{defi}
		
Notice that the third property implies that all fibres $\A_\overline{a}$ of $\A$ are isomorphic to $\F_\A$. This means that the typical fibre can be retrieved from the projection structure alone. In a similar fashion to how one would quickly forget about the specific atlas of a manifold, only the existence of such trivialisations shall be important, not their specific values. In particular, the notion of trivialising coordinates, defined here after, will often be used:

\begin{defi}{Trivialising coordinates}{triv_coord}
	A \underline{trivialising affine (resp. linear) coordinate system} on an affine (resp. vector) bundle $\bundle\A\AA$ is a smooth bijective map
\renewcommand{\tmp}{\cf}
\[
	\tmp : \A \longto \RR^{\dim(\AA)}\times\RR^{\dim\left (\F_\A\right )}
\]
such that, there exist
\begin{enumerate}
	\it a coordinate system $\ov \tmp : \AA \longto \RR^{\dim(\AA)}$ called the \underline{horizontal coordinates}
	\it for every local trivialisation ${\Psib_{\ov{a}} : \at{\A}_{\UU_\ov{a}} \longto \UU_\ov{a}\times\F_\A}$ around $\ov a \in \AA$, an affine (resp. linear) coordinate system ${\at{\tmp^\vr}_{\Psib_{\ov{a}}} : \F_\A \longto \RR^{\dim\left (\F_\A\right )}}$ on $\F_\A$ such that:
	\algn{
		\tmp &= \left (\ov\tmp_\ov a \times \at{\tmp^\vr}_{\Psib_{\ov{a}}} \right ) \circ \Psib_\ov a
	}
\end{enumerate}

The projection of $\tmp$ on the second element gives $\tmp^\vr : \A \longto \RR^{\dim\left (\F_\A\right )}$. This is a coordinate system on each fibre $\A_\ov a$ (for $\ov a \in \AA$ fixed) called the \underline{vertical coordinates}.\\

A \underline{trivialising affine (resp. linear) frame} on $\bundle{\Tr\A}{\A}$ is a frame obtained by differentiation:
\[
	\Tr\tmp \equiv \left (\tmp, \dr{\tmp}\right ) : \Tr\A \longto \RR^{\dim(\AA)}\times\RR^{\dim\left (\F_\A\right )}\times\RR^{\dim(\AA)}\times\RR^{\dim\left (\F_\A\right )}
\]
of a trivialising affine (resp. linear) coordinate system $\tmp$ on $\bundle\A\AA$. By duality, the basis vectors $\der{\tmp} \subset \RR^{\dim\left (\F_\A\right )}$  can be used interchangeably with the the coordinate system $\dr\tmp$. For readability reason, the former shall be preferred.\\

The projection on the left elements gives $\tmp : \Tr\A \longto \RR^{\dim(\AA)}\times\RR^{\dim\left (\F_\A\right )}$, called the \underline{punctual coordinates} associated to the frame. The projection on the other elements gives $\der{\tmp} : \Tr\A \longto \RR^{\dim(\AA)}\times\RR^{\dim\left (\F_\A\right )}$ \partxt{or equivalently $\dr\tmp$} which is a coordinate system on every $\Tr_a\A$ (for $a \in \A$ fixed) called the \underline{vectorial coordinates}.\\
\end{defi}

\subsubsection*{Tangent bundle}

	Particular attention should be given to the tangent bundle $\bundle{\Tr\A}{\A}$\footnote{The convention used is the mathematical one, where $\Tr\A$ is the set of fixed vectors. That is, the set of couples $(a, \ub)$ where $a \in \A$ and $\ub$ is a vector tangent to $\A$ at $a$. One therefore has $\dim(\Tr\A) = 2\dim(\A)$.} in the case where $\bundle{\A}{\AA}$ is an affine bundle. $\Tr\A$ can then be endowed with two different projection structures\\

Assume trivialising affine coordinate systems are given on $\A$ and $\Tr\AA$. First, the projection structures on $\A$ and $\Tr\AA$ are unambiguous:

\byalgn{
	\text{\underline{the punctual projection}}&\\
	\pi_{\Tr\AA} &: \fun{\Tr\AA}{\AA}{\bmat{\ov x\\ \delta \ov x}}{\ov x}
}{
	\text{\underline{the macroscopic projection}}&\\
	\pi_\A &: \fun{\A}{\AA}{\bmat{\ov x\\ y}}{\ov x}
}

where, underlined, are the name they are given. Secondly, those projections can both be generalised to $\Tr\A$:
\begin{enumerate}
	\it the first one by substituting $\AA \to \A$, leading to a vector bundle over $\A$.
		\algn{
			\text{\underline{the punctual projection}}&
			&\pi_{\Tr\A} &: \fun{\Tr\A}{\A}{\bmat{\bmat{\ov x\\ y}\\\bmat{\delta\ov x\\ \delta y}}}{\bmat{\ov x\\ y}}
		}
	\it the second one, by differentiating, leading to a projection structure\footnote{The projection $\Tr\pi_\A$ fails to induce an affine bundle structure since the fibres are not affine spaces. However, it induces a more general notion $-$ out of the scope of this article $-$ of a fibre bundle structure over $\Tr\AA$ with typical fibre $\Tr\F_\A$ \parencite[215]{epstein2010geometrical}.}.
		\algn{
			\text{\underline{the macroscopic projection}}&
			&\Tr\pi_\A &: \fun{\Tr\A}{\Tr\AA}{\bmat{\bmat{\ov x\\ y}\\\bmat{\delta\ov x\\ \delta y}}}{\bmat{\ov x\\\delta \ov x}}
		}
\end{enumerate}

All of this is summarised in the following diagram:

\renewcommand{\tmp}[2]{$\mat{#1\\\raisebox{.5em}{\scalebox{0.75}{$#2$}}}$}

\begin{center}
\begin{tikzpicture}[->,>=stealth',shorten >=1pt,auto,node distance=3.5cm,
                    semithick]
  \tikzstyle{every state}=[fill=none,draw=none,text=black]

  \node[state]		(TA) 							{\tmp{\Tr\A}{(a,\ub)}};
  \node[state]		(A)		[below left of = TA]	{\tmp{\A}{a}};
  \node[state]		(TAA) 	[below right of = TA]	{\tmp{\Tr\AA}{\left (\overline{a},\overline{\ub}\right )}};
  \node[state]		(AA) 	[below right of = A]	{\tmp{\AA}{\overline{a}}};

  \draw [-{To}{To}]%
  		(TA) 	edge					node [above left] 	{$\pi_{\Tr\A}$} (A)
  		(TA) 	edge					node [above right] 	{$\Tr\pi_{\A}$} (TAA)
  		(A) 	edge					node [below left] 	{$\pi_{\A}$} (AA)
  		(TAA) 	edge					node [below right] 	{$\pi_{\Tr\AA}$} (AA);
\end{tikzpicture}
\end{center}

\label{overline_notation}

In this diagram and the remaining of this article, the notation $\overline{\ub}$ for the \underline{macroscopic part} $\overline{\ub} = \Tr_a\pi_\A\left(\ub\right ) \in \Tr_\overline{a}\AA$ of $\ub \in \Tr_a\A$ was chosen \partxt{not to be confused with $a = \pi_{\Tr\A}(\ub) \in \A_\ov a$}.\\

Since, for $\ov a \in \AA$, $\Tr_\ov a\AA$ is a vector space it has a zero $0_{\ov a} \in \Tr_{\ov a}\AA$. This allows us to define a special fibre of the $\bundle[\Tr\pi_\A]{\Tr\A}{\Tr\AA}$ vector bundle structure as follows:

\[
	\Vr_a\A = \left (\Tr_a\pi_\A\right )^{-1}\left (\left \{0_{\ov a}\right \}\right ) \subset \Tr_a\A
\]

This defines a vector bundle $\bundle[\pi_{\Tr\A}]{\Vr\A}{\A}$ over $\A$ (and not $\Tr\AA$) called the \underline{vertical bundle} which can also be defined in a more concise way as $\Vr\A = \ker\left (\Tr\pi_\A\right )$.

\begin{rema}
	\label{rem:iso_vert_Tfib}
	Let $\bundle\A\AA$ be an affine bundle, then one has the following canonical isomorphisms:
		\quant{\forall a \in \A,}{\Vr_a\A &\simeq \Tr_a \A_\ov a}
	Indeed, let $\left (\ov \xb, \yb ,\der{\ov\xb}, \der\yb\right )$ be a generic trivialising affine frame on $\Tr\A$. Then it induces a coordinate system $\left (\ov\xb(\ov a), \yb\right )$ on $\A_\ov a$, where $\ov\xb(\ov a)$ are the (fixed) coordinates of $\ov a \in \AA$. Differentiating it, one obtains a coordinate system $\left (\ov\xb(\ov a), \yb, \bold0, \der\yb\right )$ on $\Tr\A_\ov a$. Specifying it at $a \in \A_\ov a$, one obtains the coordinate system $\left (\ov\xb(\ov a), \yb(a), \bold0, \der\yb\right )$ on $\Tr_a\A_\ov a$ in which only $\der\yb$ is free.\\
	
	Going the other-way around, restricting $\left (\ov \xb, \yb ,\der{\ov\xb}, \der\yb\right )$ on $\Vr\A$ gives the coordinate system $\left (\ov \xb, \yb ,\bold 0, \der\yb\right )$ on $\Vr\A$. This is because $\Tr\pi_\A$ is simply $\left (\ov\xb, \der{\ov\xb}\right )$ in those coordinates (since they are trivialising). Specifying the coordinates at $a \in \A$, one obtains the coordinate system $\left (\ov\xb(\ov a), \yb(a), \bold0, \der\yb\right )$ on $\Vr_a\A$ in which only $\der{\yb}$ is free. This provides an isomorphism between $\Tr_a\A_\ov a$ and $\Vr_a\A$ which does not depend on the choice of frame.
\end{rema}

\subsubsection*{Fibre product along a common base}

A useful construct is the fibre product of two affine (resp. vector) bundles along a common base, defined as follows:

\begin{defi}{Fibre product of bundles}{dir_sum}
	Let $\bundle[\pi^1]{\A^1}\AA$ and $\bundle[\pi^2]{\A^2}\AA$ be two affine (resp. vector) bundles over the same base $\AA$. The (fibre) product of ${\A^1}$ and ${\A^2}$ along $\AA$ is the affine (resp. vector) bundle
	\[
		\bundle[~~\pi_\times~~]{\A^1\times_\AA\A^2}{\AA}
	\]
	with
	\begin{enumerate}
		\it total space $\A^1\times_\AA\A^2 := \setof{\left (a^1, a^2\right ) \in \A^1\times\A^2}{\pi^1\left (a^1\right ) = \pi^2\left (a^2\right )}$
		\it projection $\pi_\times : \fun{\A^1\times_\AA\A^2}{\AA}{\left (a^1, a^2\right )}{\pi^1\left (a^1\right ) \hspace{2em} \left [= \pi^2\left (a^2\right )\right ]}$
	\end{enumerate}
	and where the transition maps are the Cartesian products of the transitions maps of $\A^1$ and $\A^2$.
\end{defi}

As a direct consequence of the construction, one has that the fibre at $\overline{a} \in \AA$ is the Cartesian product of the fibres at $\overline{a}$:
	\[
		\left (\A^1\times_\AA\A^2\right )_\overline{a} = \A^1_\overline{a}\times \A^2_\overline{a}
	\]
which directly implies that the same holds for the typical fibre:
	\[
		\F_{\A^1\times_\AA\A^2} = \F_{\A^1}\times \F_{\A^2}
	\]
Since the transition maps of the fibre product are the Cartesian products of the transitions maps, one has that the structure group of the product is a subgroup of the product of the structure groups (or, equivalently, the direct sum):
	\[
		\G_{\A^1\times_\AA\A^2} \subset \G_{\A^1}\times \G_{\A^2} \subsetneq \Aff\left (\F_{\A^1}\times \F_{\A^2}\right )
	\]

\subsection{Moving frames and connection}
\label{sect:connection}

In \cite{cartan1922generalisation} Cartan describes a notion of frames, depicted as a \say{trièdre trirectangle} at each $\xr \in \RR^3$ of the affine space $\RR^3$. Such frames allow to \say{identify infinitely close points}, that is, the frame at $\xr$ is a coordinate system of the infinitesimal space $\Tr_\xr\RR^3$. More precisely, \cite{cartan1922generalisation} describes it as an affine basis composed of a \say{rotation} and a \say{translation}. Cartan's space is therefore a principal bundle\footnote{The projection structure is not an affine as the fibres are not affine spaces but groups. More technically, this structure corresponds to the notion of principal bundle \parencite{epstein2010geometrical}, which this paper is not interested in.} $\operatorname{AFrame}(\RR^3)\longto\RR^3$ whose typical fibre is $\Aff(\RR^3)$ $-$ the set of affine transformations of $\RR^3$ seen as affine frames of $\RR^3$ $-$ which also corresponds to its structure group.\\

Furthermore, frames at close points can be compared one to another\footnote{\say{repérer, par rapport au trièdre d'origine $\Ar$, tout trièdre de référence ayant son origine $\Ar'$ voisine de $\Ar$} $-$ \cite[593]{cartan1922generalisation}} meaning a frame at a point can be transported into a frame at another close point through a translation and a rotation. This process is called a parallel transport and leads to the notion of principal Ehresmann connection.\\

Since transforming affine frames is the same as transforming points in an affine manner, there exists a one to one correspondence between parallel transport of affine frames of $\operatorname{AFrame}(\RR^3)$ and affine parallel transport of points of $\bundle{\Tr\RR^3}{\RR^3}$. Therefore, while Cartan used the principal bundle $\operatorname{AFrame}(\RR^3)$, this paper shall use the affine bundle $\Tr\RR^3$ directly. The notion of principal Ehresmann connection then translates to the analogous notion of affine Ehresmann connection, defined as follows:\\

\begin{defi}[breakable]{Ehresmann connection}{connection}
	Let $\bundle\A\AA$ be an affine bundle and $\bundle[\Tr\pi_\A]{\Tr\A}{\Tr\AA}$ its associated tangent bundle. An affine\footnote{For the remaining of this paper, the terms of "connection" and "Ehresmann connection" will, unless explicitly mentioned, refer to the notion of affine Ehresmann connection.} Ehresmann connection $\gammab$ on $\A$ is a smooth map ${\gamma : \A \times_\AA \Tr\AA \longto \Tr\A}$ satisfying the following properties:
	\begin{enumerate}
		\it for all $a \in \A$, the image of the \underline{horizontal lift operator} $\gammab_a$ at $a$ is a subset of $\Tr_a\A$:
			\[
				\gammab_a : \fun{\Tr_\ov{a}\AA}{\Tr_a\A}{\ov\ub}{\gammab\left (a, \ov\ub\right )}
			\]
		\it for all $a \in \A$, ~$\gammab_a$ is \textbf{linear} and is a \textbf{section} of $\Tr\A$: $~\hfill\Tr\pi_\A\cdot\gammab_a = \Id\hfill~\hfill$
		\it for all $\ov\ub \in \Tr_\ov a\AA$ and every trivialising affine frame on $\Tr\A$ (see \cref{def:triv_coord}),\\
		the vectorial coordinate of $\gammab(a, \ov\ub) \in \Tr_a\A$ is \textbf{affine} in the vertical coordinate of $a$.
		\it the field $a \longmapsto \gammab_a$ is \textbf{smooth}, as defined by Epstein in \cite[240,246]{epstein2010geometrical}.
	\end{enumerate}
	The value $\gammab_a\cdot\overline{\ub} = \gammab\left (a, \ov{\ub}\right )$ is called the \underline{horizontal lift of $\overline{\ub}$} at $a$. The set of all connections on $\A$ will be denoted
\algn{
	\opp{Con}\A &\subset \smooth{\A \times_\AA \Tr\AA}{\Tr\A}
}
\end{defi}

\renewcommand{\tmp}{\rho}
The \underline{parallel transport} $\gammab\left (\tmp\right )_s^t : \A_{\tmp(s)} \simeq \A_{\tmp(t)}$ along a $\C^1$ path $\tmp$ in $\AA$ is then defined as the integration of the flow obtained by lifting $\Tr\tmp$ by $\gammab$. That is:

\[
	\gammab\left (\tmp\right )_s^t : \fun{\A_{\tmp(s)}}{\A_{\tmp(t)}}{a}{a + \int_s^t \gammab\cdot\Tr_u\tmp~ \dr u}
\]

Parallel transport is illustrated in \cref{fig_con} as the dotted blue paths. Another tangent notion is the notion of \underline{covariant derivative}, defined for a smooth section $\sigmab : \AA \longto\A$ on $\A$ and a vector $\overline\ub \in \Tr_\overline{a}\AA$ of the base as
\[
	\nabla^\gammab_{\overline\ub}~\sigmab = \left (\Id-\gammab_{\sigmab\left (\overline{a}\right )}\cdot\Tr\pi\right )\cdot\Tr_\overline{a}~\!\sigmab\cdot\overline{\ub}
\]
This is, when $\gammab$ is linear, Koszul's definition of a connection\footnote{See \cite[p.~3-5]{nguyen2021geometric} for a definition of a connection via its covariant derivate or \cite[255]{epstein2010geometrical} for a derivation of the formula in the $\A = \Tr\AA$ case.}. This is illustrated in \cref{fig_con} as a red vertical vector \partxt{equal to the difference between the orange and the blue one}. One then has the following result:\\

\begin{lemm}
	\label{lemm:lemm_equiv_con}
	Let $\gammab$ be a \textbf{linear} connection on $\bundle\A\AA$. Then, the following data are equivalent (\ie~any one can be retrieve from any other):
	\begin{enumerate}
		\itn The horizontal lifting operators $\setof{\gammab_a : \Tr_\overline{a}\AA \longto \Tr_a\A}{a \in \A}$
		\itn The \underline{horizontal spaces} $\setof{\Hr^\gammab_a = \gammab_a\left (\Tr_\overline{a}\AA\right ) \subset \Tr_a\A}{a \in \A}$
		\itn The parallel transport maps $\setof{\gammab\left (\tmp\right )}{\tmp\in \C^1\left (\intcc01, \AA\right )}$
		\itn The covariant derivatives $\setof{\nabla^\gammab_\overline{\ub}}{\overline{\ub}\in \Tr\AA}$ 
	\end{enumerate}
\end{lemm}

Of all those equivalent definitions, only the forms $(1)$ and $(2)$ shall be used in this paper, as they are the most convenient when dealing with linear/affine operators. The forms $(3)$ and $(4)$ are perhaps the most commonly used in the literature, in particular in the field of general relativity. It should be duly noted that, in the affine case, while $(1)$, $(2)$ and $(3)$ are well-defined, $(4)$ loses its linearity and is therefore not a derivative any-more. In an even broader setting, not in the scope of this paper, where connections are not even affine in the punctual vertical coordinates, only $(1)$ and $(2)$ subsist as $(3)$ may diverge.

\begin{figure}
    \centering
    \def\svgwidth{0.75\textwidth}
    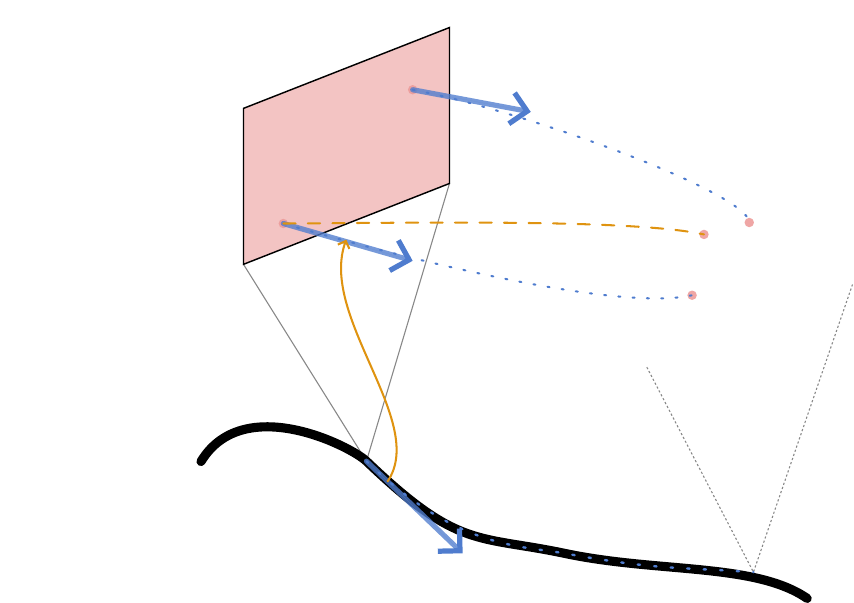
    \caption{Illustration of \cref{lemm:lemm_equiv_con} with $a \in \A$, $\ov\ub\in\Tr_\ov a\AA$, $\sigma : \AA \longto \A$ a section of $\A$, $\rho : \intcc01 \longto \AA$ a path in $\AA$ and the notation $\hr_\gammab = \gammab \cdot \Tr\pi$. {\color{macro} Blue} indicates macroscopic/horizontal quantities, {\color{micro} red} indicates microscopic/vertical quantities and {\color{gold} gold} indicates mixed quantities.}
	\label{fig_con}
\end{figure}

\subsection{Solder form}
\label{sect:solder}

Since the model seeks to be general, the space will not to be restricted $\Tr\RR^3$. Nevertheless, the fibres must be interpretable as infinitesimal neighbourhoods of the base points. Mathematically, this translates into the notion of solder form, defined as follows:

\begin{defi}[unbreakable]{Solder form}{solder}
	Let $\bundle\A\AA$ be an affine bundle and $\bundle[\Tr\pi_\A]{\Tr\A}{\Tr\AA}$ its associated tangent bundle. A solder form $\varthetab$ on $\A$ is a \textbf{smooth injective} map $\varthetab : \A \times_\AA \Tr\AA\longto\Tr\A$ satisfying the following properties:
	\begin{enumerate}
		\it for all $a \in \A$, the \underline{vertical lift} $\varthetab_a$ at $a$ is \textbf{linear} and its image is a subset of $\Vr_a\A$:
			\[
				\varthetab_a : \fun{\Tr_\ov{a}\AA}{\Vr_a\A}{\ov \ub}{\varthetab\left (a, \ov \ub\right )}
			\]
			in particular: $~\hfill\Tr\pi_\A\cdot\varthetab_a = \mathbf{0}\hfill~\hfill$
		\it for all $\ov\ub \in \Tr_\ov a\AA$ and every trivialising affine frame on $\Tr\A$ (see \cref{def:triv_coord}),\\
		the vectorial coordinate of $\varthetab(a, \ov\ub) \in \Tr_a\A$ is \textbf{affine} in the vertical coordinate of $a$.
	\end{enumerate}
	
	where $\Vr_a\A = \ker \Tr_a\pi_\A \subset \Tr_a\A$ is the vertical space at $a \in \A$, isomorphic to $\Tr_a\A_\overline{a}$ \partxt{see \cref{rem:iso_vert_Tfib}}. The set of all solder forms on $\A$ will be denoted
\algn{
	\opp{Sold}\A &\subset \smooth{\A \times_\AA \Tr\AA}{\Tr\A}
}
	Since $\vartheta$ is injective, the following notation is introduced:
	\algn{
		\varthetab^{-1} &: \fun{\Vr\A\supset\Imr\left (\varthetab\right )}{\Tr\AA}{\ub}{\left [\varthetab_{\pi_{\Tr\A}(\ub)}\right ]^{-1}\left (\ub\right )}
	}
\end{defi}

\begin{figure}
    \centering
    \duo{
	    \def\svgwidth{\textwidth}
    	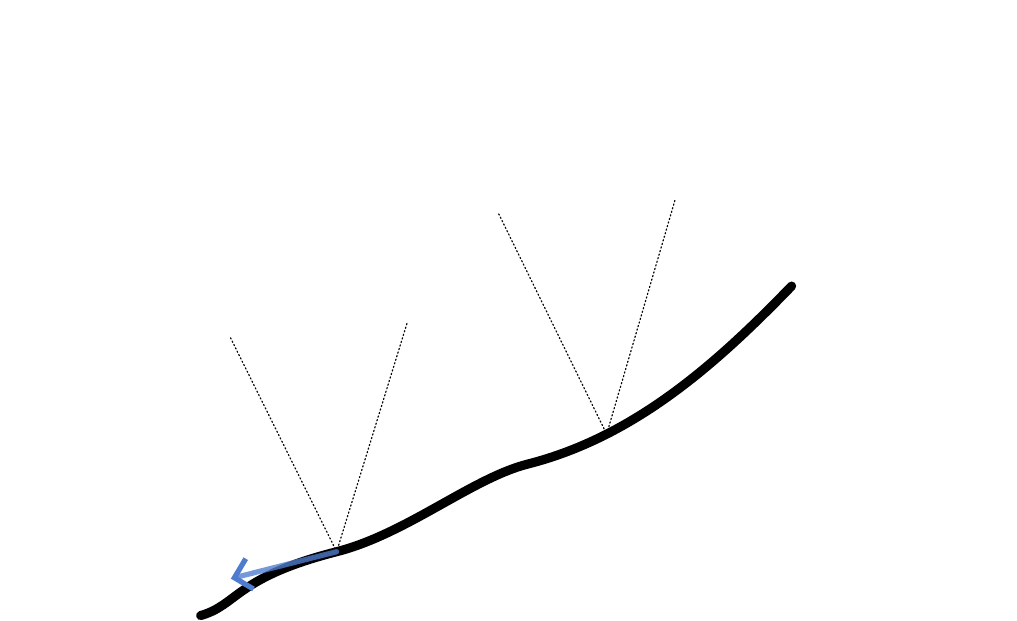
    }{
	    \def\svgwidth{\textwidth}
	    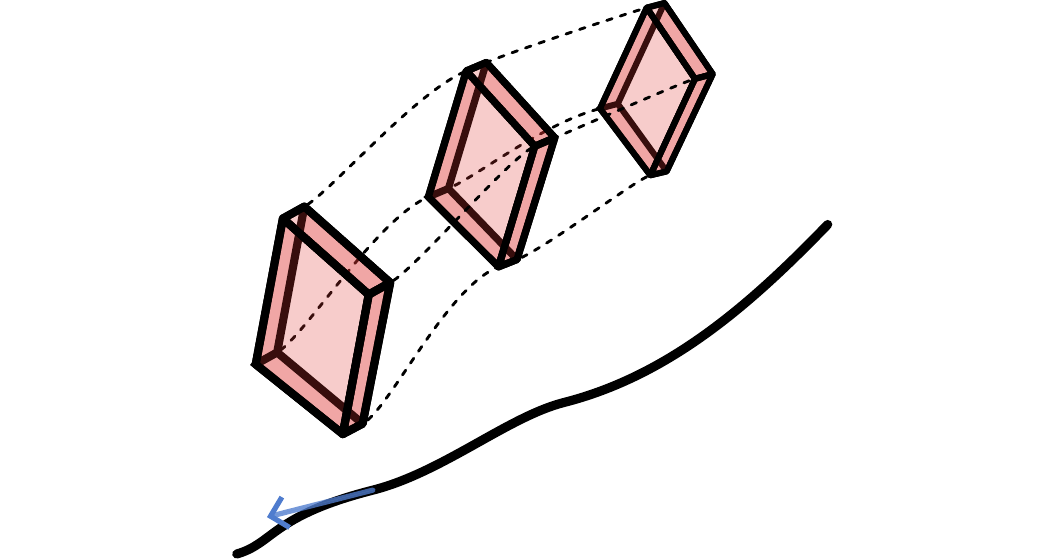
    }
    \caption{Illustrations of \cref{def:solder} for a Cosserat media (left side) and a Timoshenko beam (right side). {\color{macro} Blue} indicates macroscopic/horizontal quantities, {\color{micro} red} indicates microscopic/vertical quantities and {\color{gold} gold} indicates mixed quantities.}
\end{figure}

Being injective, the existence of a solder form on $\A$ implies $\dim\left (\F_\A\right ) \ge \dim(\AA)$. The lack of equality in general allows to include spaces where the fibres are infinitesimal neighborhoods of the base points but in a larger space. For example the $3\Dr$ neighborhood of a curve in $\RR^3$ where the base would be the $1\Dr$ curve.\\

In the case where $\dim\left (\F_\A\right ) = \dim(\AA)$, for example when $\A \simeq \Tr\AA$, any solder form will be bijective at any given point. This allows one to project microscopic (\ie{} vertical) vectors onto the macroscopic world using the microscopic projection $\varthetab^{-1}$:

\[
	\varthetab^{-1}: \fun{\Vr\A}{\Tr\AA}{\ub}{\varthetab_{\pi_{\Tr\A}(\ub)}^{-1}\cdot\ub}
\]

\subsection{\texorpdfstring{$\Tr\A=\Hr\A\oplus_\A\Vr\A$}{TA=HA+VA} and the block-wise decomposition}
\label{sect:t_eq_h_plus_v}

\subsubsection*{On the $\Tr\A=\Hr\A\oplus_\A\Vr\A$ decomposition}

Let $\gammab$ be a connection on an affine bundle $\bundle\A\AA$. As quickly mentioned in \cref{lemm:lemm_equiv_con}, its image is called the \underline{horizontal space} $\Hr^\gammab\A = \Imr(\gammab) \subset \Tr\A$ which is a vector bundle over $\A$ with projection $\pi_{\Tr\A}$. By definition, one has:

\vbox{
\algn{
	\Hr_a^\gammab\A \cap \Vr_a\A &= \setof{\gammab_a\left (\overline{\ub}\right )}{\overline{\ub} \in \Tr\AA~\text{and}~\Tr\pi_\A\left (\gammab_a\left (\overline{\ub}\right )\right ) = 0}\\
			&= \setof{\gammab_a\left (\overline{\ub}\right )}{\overline{\ub} \in \Tr\AA~\text{and}~\overline{\ub} = 0}\\
			&= \left \{\gammab_a\left (\bold{0}\right )\right \}\\
			&= \left \{0\right \} \subset \Tr_a\A
}
}

meaning that:

\quant{\forall a \in \A,}{\Tr_a\A = \Hr_a\A \oplus \Vr_a\A}

Which decomposes the total tangent space $\Tr_a\A$ at $a\in \A$ into a direct sum of the horizontal space $\Hr_a\A\simeq\Tr_a\AA$ and the \underline{vertical space} $\Vr_a\A\simeq\Tr_a\A_\overline{a}$.  A connection can therefore be seen as an horizontal lift while a solder form would be a vertical one. This interpretation is similar to the notion of \say{internal observer} described by \fullciteauthor{kroner1980continuum}, allowing to \say{[jump] from one atom to the next} $-$ this is the role of the connection through parallel transport $-$ and \say{distinguish crystallographic directions, \ie{} [always know] what is straight-on, to the left, upwards, \textit{etc}.} $-$ this is the role of the solder form through the lifting of a frame.\\

On one hand, the horizontal space $\Hr_a\A$ $-$ isomorphic to $\Tr_a\AA$ through $\Tr\pi_\A$ $-$ will therefore be physically interpreted as the set of macroscopic vectors at $a$. On the other hand, the vertical space $\Vr_a\A$ $-$ isomorphic to $\Tr_a\A_\overline{a}$ (see \cref{rem:iso_vert_Tfib}) $-$ will be physically interpreted as the set of microscopic vectors at $a$. Furthermore, projections onto the horizontal and vertical spaces are also provided by the connection. The horizontal projection associated to $\gammab$ is:
\[
	\hr_\gammab : \fun{\Tr\A}{\Hr^\gammab\A}{\ub}{\gammab_{\pi_{\Tr\A}(\ub)}\cdot\Tr\pi_\A\cdot\ub}
\]
while the vertical projection is its complementary:
\[
	\vr_\gammab = \Id - \hr_\gammab: \fun{\Tr\A}{\Vr\A}{\ub}{\ub-\gammab_{\pi_{\Tr\A}(\ub)}\cdot\Tr\pi_\A\cdot\ub}
\]\\

\label{scale_discuss}
Unit-wise, one has that, while the connection is dimensionless (it maps macro. to macro.), the solder form has mixed unit $\left [\frac{\mathrm{micro.}}{\mathrm{macro.}}\right ]$. In this paper, $\left [\mathrm{micro.}\right ]$ and $\left [\mathrm{macro.}\right ]$ have the dimension of a length. Let $\ell$ and $L$ be respectively the microscopic and macroscopic characteristic scales. Then, informally, the solder form scales \partxt{potentially anisotropically\footnote{Scaling here should be interpreted in the sense of $\det\left (\vartheta\right )=\zeta^3$ (which is ill-defined) and not $\vartheta = \zeta\cdot\Id$.}} by a factor $\zeta = \frac{\ell}{L}$, where the real $\zeta \in \RR^*_+$ is dimensionless. Such a ratio appears in \cite[87]{nguyen2021geometric}, \cite[14]{nguyen2021tangent} and \cite[235]{reina2016derivation} and is an example of an order parameter of a model \cite{kroner2001benefits}. The horizontal projection can then be seen heuristically as a rounding to the closest "multiple" of $\ell$ (or $\zeta$ if $L$ is normalised to $1$) and the vertical part as its remainder.

\subsubsection*{Reference connection and block-wise decomposition}

A connection will most of the time be endowed with a physical interpretation corresponding to the one depicted in \cref{sect:connection}. Nevertheless, as explained in the previous subsection a connection is, mathematically, not much more than a smooth (linear) decomposition of the tangent space as a direct sum of an horizontal and vertical space. Such a decomposition can be very useful for computation, even if arbitrary. This is analogous to how specifying a frame or basis helps the computation, even if no canonical one exists. When such a non-necessarily physical connection exists, it shall be called a reference connection. Formally, one defines the following:\\

\begin{defi}{Reference connection}{ref_con}
	Let $\bundle\A\AA$ be an affine bundle. The \underline{reference connection} of $\A$ is a connection ${\RefCon\A : \A \times_\AA \Tr\AA \longto \Tr\A}$ on $\A$ which is either
	\begin{enumerate}
		\it canonical: its value is prescribed when $\A$ is first defined and never changed afterward
		\it free: its value is not set (\ie{} it is a free variable)
	\end{enumerate}
	
	In the second case, expressions depending on $\RefCon\A$ can therefore be seen as functions of a connection and equalities or other statements must remain true whatever the value of that connection argument.\\
	
	Regardless of its freedom or canonicity, the reference connection of a space solely depends on the space and is stable by restriction. That is, $\RefConName$ is a well-defined functor verifying:
	\quant{\forall \UU \subset \AA,}{\RefCon{\at{\A}_\UU} &=& \at{\RefCon\A}_{\UU} : \at{\A}_\UU \times_\UU \Tr\UU \longto \Tr\at{\A}_\UU}
\end{defi}

When one has a connection $\gammab$ on an affine bundle $\A$, one also has a decomposition $\Tr\A = \Hr^\gammab\A \oplus_\A \Vr\A$ along with projections $\hr_\gammab : \Tr\A \longto \Hr^\gammab\A$ and $\vr_\gammab : \Tr\A \longto \Vr\A$. If one has such a decomposition on the argument and image space of a linear application, this provides a way to decompose the linear application in a block-wise manner:

\begin{nota}{Block-wise decompositions}{blck_dec}
	Let $\bundle{\A_1}{\AA_1}$ and $\bundle{\A_2}{\AA_2}$ be two affine bundles. Let $\left (a_1, a_2\right ) \in \A_1\times\A_2$ and ${\Lb : \Tr_{a_1}\A_1 \longto \Tr_{a_2}\A_2}$ be a linear application.\\
	
\hrulefill\\
	
	Let $\gammab_1$ be a connection on $\A_1$. The following notations are introduced:
	
	\byalgn[.]{
		\at{\Lb_\hr}_{\gammab_1} &= \Lb \cdot \hr_{\gammab_1}
	}{
		\at{\Lb_\vr}_{\gammab_1} &= \Lb \cdot \vr_{\gammab_1}
	}
	
	Additionally, if one has $\Ab : \Hr\strut^{\gammab_1}_{a_1}\A_1 \longto \Tr_{a_2}\A_2$ and $\Bb : \Vr_{a_1}\A_1 \longto \Tr_{a_2}\A_2$, a new operator $\equiv$ is defined as follows:
	\algn{
		\Lb &\equiv \at{\bmat{\Ab&\Bb}}_{\gammab_1} &\iff \Lb = \Ab\cdot\hr_{\gammab_1} + \Bb \cdot \vr_{\gammab_1}
	}
	
	This means that the notation can be summarized as:
	
	\algn{
			\Lb &\equiv \at{\bmat{\at{\Lb_\hr}_{\gammab_1}&\at{\Lb_\vr}_{\gammab_1}}}_{\gammab_1}
	}
	
\hrulefill\\

	Let $\gammab_2$ be a connection on $\A_2$. The following notations are introduced:
	
	\byalgn[.]{
		\at{\Lb^\hr}_{\gammab_2} &= \hr_{\gammab_2} \cdot \Lb
	}{
		\at{\Lb^\vr}_{\gammab_2} &= \vr_{\gammab_2} \cdot \Lb
	}
	
	Additionally, if one has $\Ab : \Tr_{a_1}\A_1 \longto \Tr_{a_2}\A_2$ and $\Bb : \Tr_{a_1}\A_1 \longto \Tr_{a_2}\A_2$, a new operator $\equiv$ is defined as follows:
	\algn{
		\Lb &\equiv \at{\bmat{\Ab\\\Bb}}_{\gammab_2} &\iff \Lb = \hr_{\gammab_2}\cdot\Ab + \vr_{\gammab_2}\cdot\Bb
	}
	
	This means that the notation can be summarized as:
	
	\algn{
			\Lb &\equiv \at{\bmat{\at{\Lb^\hr}_{\gammab_2}\\\at{\Lb^\vr}_{\gammab_2}}}_{\gammab_2}
	}

\hrulefill

	Given a connection $\gammab_1$ on $\A_1$ and $\gammab_2$ on $\A_2$, the two previous notations are combined in a similar fashion. This leads to the following "summary":
	\algn{
			\Lb &\equiv \at{\bmat{\at{\Lb^\hr_\hr}_{\gammab_1,\gammab_2} & \at{\Lb^\vr_\hr}_{\gammab_1,\gammab_2}\\
								  \at{\Lb^\hr_\vr}_{\gammab_1,\gammab_2} & \at{\Lb^\vr_\vr}_{\gammab_1,\gammab_2}}}_{\gammab_1,\gammab_2}
	}
	where $\at{\Lb_\hr^\vr}_{\gammab_1,\gammab_2} = \vr_{\gammab_2}\cdot\Lb\cdot\hr_{\gammab_1}$, etc. The notations are also extended to the case where $\Lb$ is valued in $\Tr^\star\A_2$ by setting $\at{\Lb_\hr^\vr}_{\gammab_1,\gammab_2} = \vr^\star_{\gammab_2}\cdot\Lb\cdot\hr_{\gammab_1}$, etc. Similarly, the case where $\Lb$ takes its arguments in $\Tr^\star\A_1$ is also included.\\
	
	Furthermore it is chosen, by convention, that \underline{when no connection is specified, the reference connection is used}. Meaning that one has:
	\algn{
			\Lb &\equiv \bmat{\Lb_\hr^\hr&\Lb_\vr^\hr\\\Lb_\hr^\vr&\Lb_\vr^\vr}
	}
	where $\Lb_\hr^\vr = \vr_{\RefCon{\A_2}}\cdot\Lb\cdot\hr_{\RefCon{\A_1}}$, etc.
\end{nota}

By construction one then has the decomposition
\algn{
	\Lb &= \Lb_\hr^\hr + \Lb_\vr^\hr + \Lb_\hr^\vr + \Lb_\vr^\vr\\
		&= \Lb_\hr + \Lb_\vr\\
		&= \Lb^\hr + \Lb^\vr
}

Where exponents imply projections of the images and indexes projections of the arguments. The usefulness of the matrix notation resides in the fact that, because of the orthogonality $\hr_\gammab\cdot\vr_\gammab = \vr_\gammab\cdot\hr_\gammab = \bold{0}$, sum and compositions of linear applications become sum and product of matrices.\\

Beware that, even in the case where $\Lb$ and all its blocks are invertible, $\left [\Lb_\hr^\hr\right ]^{-1}$ may differ from $\left [\Lb^{-1}\right ]_\hr^\hr$ on $\Imr\left (\Lb_\hr^\hr\right )$ \partxt{on which $\left [\Lb_\hr^\hr\right ]^{-1}$ is defined}. The same subtlety arises with the other blocks. This is analogous to the case of real matrices.

\subsection{Metric and pseudo-metric}
\label{sect:pseudo_metr}

One goal of this paper is to describe the geometry of a material under deformation. Consequently, being able to measure lengths and angles is crucial. As is the case in standard mechanics, the existence of a metric $\gb$ on the macroscopic space $\AA$ is therefore postulated. That is, a linear, symmetrical and positive-definite application\footnote{%\label{foot_cotang}%
	$\Tr^\star\D\longto\D$ is the cotangent of $\D$, \ie{} the algebraic dual space of $\Tr\D$:
\quant{\forall x \in \D,}{\Tr^\star_x\D &=& \setof{\Lb : \Tr_x\D \longto \RR}{\Lb~\text{is linear and continuous}}}%
} $\ifun{\Tr_\ov{a}\AA}{\Tr_\ov{a}^\star\AA}{\ub}{{\bra\ub}_\gb}$ at every $\ov{a} \in \AA$. This induces a scalar product ${\braket\ub\wb}_\gb := {\bra\ub}_\gb\cdot \wb$ and a norm $\left \|\ub\right \|_\gb=\sqrt{{\braket\ub\ub}_\gb}$ on every tangent space $\Tr_\overline{a}\AA$.\\

The aim of this section is to define an analogous linear application on the total tangent space $\bundle{\Tr\A}{\A}$ of the total space $\bundle\A\AA$. Algebraically, it is required to be symmetric and non-negative but, for further convenience, not necessarily non-degenerate. Such an object is called a pseudo-metric:\\

\begin{defi}{Pseudo-metric and semi-norm}{pseudo_met}

Let $\A$ be a smooth manifold and $\bundle{\Tr\A}\A$ its tangent bundle.\\
\rule{15em}{.1px}\\
A pseudo-metric $\gf$ on $\A$ is a smooth application%\footref{foot_cotang}
\byalgn[.\text{\scriptsize or, equivalently}]{
	\fun{\Tr_a\A}{\Tr^\star_a\A}{\ub}{{\bra\ub}_\gf}~\text{at every}~a\in\A
}{
	\fun{\Tr\A\times_\A\Tr\A}{\RR}{\left (\ub, \wb\right )}{{\braket\ub\wb}_\gf := {\bra\ub}_\gf\cdot \wb}
}
such that:
\begin{enumerate}
	\it $\gf$ is \textbf{bi-linear}: $\forall a \in \A,\hfill\Tr_a\A\times\Tr_a\A\ni\left (\ub, \wb\right )\longmapsto{\braket\ub\wb}_\gf~\text{is bi-linear}~\hspace{7.1em}~$
	\it $\gf$ is \textbf{symmetric}: $\forall a \in \A,~\forall \left (\ub, \wb\right ) \in \Tr_a\A\times\Tr_a\A,\hfill{\braket\ub\wb}_\gf = {\braket\wb\ub}_\gf\hspace{7.7em}~$
	\it $\gf$ is \textbf{positive semi-definite}: $\forall \ub \in \Tr\A,\hfill{\braket\ub\ub}_\gf \ge 0\hspace{10.85em}~$
\end{enumerate}
\rule{15em}{.1px}\\

A semi-norm $\left \|\cdot\right \|$ on $\A$ is a smooth application
\[
	\left \|\cdot\right \| : \Tr\A \longto \RR_+
\]
such that:
\begin{enumerate}
	\it $\left \|\cdot\right \|$ is absolutely homogeneous: \hspace{1em} $\forall \ub \in \Tr\A, \forall \lambda \in \RR_+,\hfill \left \|\lambda\cdot\ub\right \| = |\lambda|\cdot\left \|\ub\right \|\hfill\hspace{3em}~$
	\it $\left \|\cdot\right \|$ is sub-additive: \hspace{1em} $\forall a \in \A,~\forall \left (\ub, \wb\right ) \in \Tr_a\A\times\Tr_a\A,\hfill \left \|\ub+\wb\right \| \le \left \|\ub\right \| + \left \|\wb\right \|\hspace{5.5em}~$
\end{enumerate}
\end{defi}

The non-negativity alone is sufficient to make Cauchy-Schwartz's lemma proof still work, leading to the following lemma:\\

\begin{lemm}
	If $\gf$ is a pseudo-metric on a manifold $\A$ then $\left \|\cdot\right \|_\gf: \fun{\Tr\A}{\RR_+}{\ub}{\sqrt{{\braket\ub\ub}_\gf}}$ is a semi-norm.
\end{lemm}

Having a metric $\gb$ on $\AA$ and a pseudo-metric $\gf$ on $\A$, the next step is to investigate the different relations between the two. As stated in \cref{cartan_vision} and after, the affine bundle $\A$ is seen as $\AA$ endowed with a microscopic space at each point. This vision implies that the space is endowed with a strong physical interpretation, in particular regarding the physical size of its elements. Leveraging this interpretation, the desired relations shall be obtained.\\

First, let $a \in \A$ be a micro-structured point and $\overline{\wb} \in \Tr_\overline{a}\AA$ be a vector. Assuming its magnitude $\left \|\ov\wb\right \|_\gb$ is large enough, one would consider it to physically represent a macroscopic vector. The connection defined in \cref{sect:connection} then implements that vision by allowing us to identify this vector with its horizontal lift $\gammab_a\cdot\overline{\wb}\in\Hr^\gammab_a\A$. However, as a direct consequence of the choice of implementing $\AA$ as a continuum, while the magnitude of a vector can be arbitrarily large, it can also be arbitrarily small. Let therefore $\overline{\ub}\in\Tr_\overline{a}\AA$ be another vector. Assuming $\left \|\ov\ub\right \|_\gb$ is small enough, one would consider $\ov\ub$ to physically represent a microscopic vector. The solder form defined in \cref{sect:solder} then implements this other vision by allowing us to identify it with its microscopic lift $\vartheta_a\cdot\overline{\ub} \in \Vr_a\A$. Notice that the later can canonically be seen as a vector of $\Tr_a\A_\ov a$ \partxt{see \cref{rem:iso_vert_Tfib}}.\\

Since those identifications relate objects with the same physical interpretation, one requires them to preserve lengths and angles. This means in particular that, focusing on length preservation, ${\left \|\gammab_a\cdot\overline{\wb}\right \|_\gf = \left \|\overline{\wb}\right \|_\gb}$ and ${\left \|\varthetab_a\cdot\overline{\ub}\right \|_\gf = \left \|\overline{\ub}\right \|_\gb}$. Furthermore, one can go a step further by combining the two identifications. The vector $\overline{\ub} +\overline{\wb}\in \Tr_\overline{a}\AA$ $-$ considered to have a microscopic part $\overline{\ub}$ and a macroscopic part $\overline{\wb}$ $-$ can then be identified with the vector $\varthetab_a\cdot\overline{\ub}+\gammab_a\cdot\overline{\wb} \in \Tr_a\A$. Using the same argument, one requires this identification to preserve lengths and angles. This means that, focusing again on length preservation, ${\left \|\varthetab_a\cdot\overline{\ub}+\gammab_a\cdot\overline{\wb}\right \|_\gf = \left \|\overline{\ub}+\overline{\wb}\right \|_\gb}$.\\

Although the norm and semi-norm were used for more readability, this reasoning implies similar equalities for the metric ${\braket\cdot\cdot}_\gb$ and the pseudo-metric ${\braket\cdot\cdot}_\gf$, of the form ${\braket{\overline{\ub} + \overline{\wb}}{{\overline{\ub}}' + {\overline{\wb}}'}}_\gb = 
		{\braket{\varthetab_a\cdot\overline{\ub} + \gammab_a\cdot\overline{\wb}}{\varthetab_a\cdot{\overline{\ub}}' + \gammab_a\cdot{\overline{\wb}}'}}_\gf$. However, since $\gb$ and $\gf$ are bi-linear, this equality holds for some size of the inputs if and only if it holds for any size. Assumptions on the magnitude of the vector are therefore unnecessary. Notice that, if one allows vectors to be zero then one can obtain the previous relations from the later. These observations lead to the following compatibility condition:\\

\begin{defi}{Compatibility condition}{compab}
	Let $\bundle\A\AA$ be a vector bundle. A pseudo-metric $\gf$ on $\A$ is said to be \underline{compatible} with a metric $\gb$ on $\AA$, a connection $\gammab$ on $\A$ and a solder form $\varthetab$ on $\A$ if and only if the following property holds:\\
	~\\
	$\forall a \in \A, \forall \left (\overline{\ub}, \overline{\wb}, \overline{\ub}', \overline{\wb}'\right ) \in \left (\Tr_\overline{a}\AA\right )^4,$
	\algn{
		{\braket{\varthetab_a\cdot\overline{\ub} + \gammab_a\cdot\overline{\wb}}{\varthetab_a\cdot{\overline{\ub}}' + \gammab_a\cdot{\overline{\wb}}'}}_\gf
		&= {\braket{\overline{\ub} + \overline{\wb}}{{\overline{\ub}}' + {\overline{\wb}}'}}_\gb
	}
\end{defi}

In the special case where $\dim\left (\F_\A\right ) = \dim\left (\AA\right )$, $\varthetab$ is bijective just like $\gammab$. This means that, in this case, the set of values taken by $\varthetab_a\cdot\overline{\ub} + \gammab_a\cdot\overline{\wb}$ covers the whole space $\Tr_a\A$. As a direct consequence, one has the following uniqueness theorem:\\

\begin{theo}{Conditioned uniqueness of a compatible pseudo-metric}{unis_gf}
	Let $\bundle\A\AA$ be an affine bundle and assume a metric $\gb$ on $\AA$, a connection $\gammab$ on $\A$ and a solder form $\varthetab$ on $\A$ are provided.\\
	
	Then if $\dim\left (\F_\A\right ) = \dim\left (\AA\right )$ there exist a unique pseudo-metric $\gf$ on $\A$ compatible with $\gb$, $\gammab$ and $\varthetab$ given by:
	\[
		\gf : \fun{\Tr\A\times_\A\Tr\A}{\RR}{\left (\ub, \wb\right )}{{\braket{\interproj(\ub)}{\interproj(\wb)}}_\gb}
	\]
	where\footnote{Beware that $\vr_\gammab = \Id - \gammab\cdot\Tr\pi_\E$ is the vertical part of $\gammab$. That is, a linear operator not a vector.} $\interproj = \Tr\pi_\A + \varthetab^{-1}\cdot\vr_\gammab : \Tr\A \longto \Tr\AA$ is called the \underline{projection of interpretation}. In other words, the only compatible pseudo-metric is the pull-back\footnote{Beware that the pull-back $\interproj^*\gb$ (with an asterisk) corresponds to the composition $\interproj^\star\cdot\gb\cdot\interproj : \Tr\A \longto \Tr^\star\A$ where $\interproj^\star : \Tr^\star\AA\longto \Tr^\star\A$ (with a star) is the transpose of $\interproj$. In particular $\interproj^*\gb \neq \interproj^\star\cdot\gb$} $\gf = \interproj^*\gb$ of the metric $\gb$ by the projection of interpretation $\interproj$.\\
	
	Furthermore, when $\dim\left (\F_\A\right ) \neq \dim\left (\AA\right )$, a pseudo-metric is compatible if and only if its restriction on $\opp{Dom}{\interproj} := \opp{Im}{\varthetab} \oplus \Hr^\gammab$ is compatible. That is, if and only if,
	\algn{
		\at{\gf}_{\opp{Dom}{\interproj}} &: \opp{Dom}{\interproj} \longto \opp{Dom}{\interproj}^\star\\
										 &= \interproj^*\gb
	}
\end{theo}

\begin{proof}
	First, for a given $a \in \A$, one has $\Vr_a\A \simeq \Tr_a\A_\overline{a}$. Hence, $\dim\left (\F_\A\right ) = \dim\left (\AA\right )$ implies:
\algn{	
	\dim\left (\Vr_a\A\right ) &= \dim\left (\Tr_a\A_\overline{a}\right )\\
							   &= \dim\left (\A_\overline{a}\right )\\
							   &= \dim\left (\F_\A\right )\\
							   &= \dim\left (\AA\right )\\
							   &= \dim\left (\Tr_\overline{a}\AA\right )
}
Therefore, $\varthetab_a : \Tr_\overline{a}\AA \longto \Vr_a\A$ is a linear injective map between two finite dimensional vector spaces with the same dimension and is hence bijective. Since $\Imr\left (\varthetab\right ) = \Vr\A = \Imr\left (\vr_\gammab\right )$, $\interproj[a]$ is well defined.\\

Secondly, one has
\[
	\Imr\left (\vartheta_a\right ) \oplus \Imr\left (\gammab_a\right ) = \Vr_a\A \oplus \Hr_a^\gammab\A = \Tr_a\A
\]
So, if $\gf$ is a compatible pseudo-metric and $\left (\ub, \ub'\right ) \in \Tr_a\A\times\Tr_a\A$ then there exists a (unique) quadruplet ${\left (\overline{\vb}, \overline{\wb}, \overline{\vb}', \overline{\wb}'\right ) \in \left (\Tr_\overline{a}\AA\right )^4}$ such that ${\ub = \vartheta_a\cdot\vb + \gammab_a\cdot\wb}$ and ${\ub' = \vartheta_a\cdot\vb' + \gammab_a\cdot\wb'}$. The compatibility condition then implies ${\braket\ub{\ub'}}_\gf = {\braket{\vb+\wb}{\vb'+\wb'}}_\gb$. This proves the uniqueness of $\gf$.\\

Finally, $\Tr_a\pi_\A\cdot\gammab_a = \Id = \vartheta_a^{-1}\cdot\vr_\gammab\cdot\varthetab_a$ and $\Tr_a\pi_\A\cdot\varthetab_a = \bold{0} =\vr_\gammab\cdot\gammab_a$. Consequently, $\interproj[a]\cdot\gammab_a=\interproj[a]\cdot\varthetab_a=\Id$. Taking $\gf$ as stated by the theorem \partxt{\ie~the pull-back of $\gb$ by $\interproj$} and taking ${\left (\overline{\vb}, \overline{\wb}, \overline{\vb}', \overline{\wb}'\right ) \in \left (\Tr_\overline{a}\AA\right )^4}$ one then has:
\algn{
		{\braket{\varthetab_a\cdot\overline{\vb} + \gammab_a\cdot\overline{\wb}}{\varthetab_a\cdot{\overline{\vb}}' + \gammab_a\cdot{\overline{\wb}}'}}_\gf &= {\braket{\interproj\cdot\varthetab_a\cdot\overline{\vb} + \interproj\cdot\gammab_a\cdot\overline{\wb}}{\interproj\cdot\varthetab_a\cdot{\overline{\vb}}' + \interproj\cdot\gammab_a\cdot{\overline{\wb}}'}}_\gb\\
			&= {\braket{\overline{\vb} + \overline{\wb}}{\overline{\vb}' +\overline{\wb}'}}_\gb\\
}
proving that $\gf$ is compatible. Since the set of elements of the form $\varthetab_a\cdot\overline{\vb} + \gammab_a\cdot\overline{\wb}$ is exactly $\opp{Dom}{\interproj}$, this last series of equalities proves the last statement.
\end{proof}

\begin{rema}
	The projection of interpretation $\interproj$ comes from the interpretation of the microscopic spaces as infinitesimal neighbourhoods $\E_\ov x \simeq \left \{\ov x + \delta\ov x\right \}$ \parencite[4]{eringen1998microcontinuum}, \parencite[10]{grammenoudis2009micromorphic}. Applying this interpretation on vectors, while only keeping the macroscopic part of the point, one gets the stated projection. Indeed, upon fixing $\ov \ub \in \Tr\EE$ and choosing some trivialising coordinates, the projection locally becomes $\at{\interproj}_\ov \ub : \fun{\Tr\E_\ov x}{\Tr\EE}{\delta\ov\ub}{\ov \ub + \delta\ov\ub}$ which corresponds to the change of frame between a frame centred at $\ov \ub$ and one centred at $0$.
\end{rema}

Going back to the algebraic assumptions, one may wonder if they may require the pseudo-metric to be a metric, at least in some cases. This leads to the notion of kernel, defined as follows:

\begin{defi}{Kernel of a pseudo-metric}{kern} 
	If $\gf$ is a pseudo-metric on a manifold $\A$ then its kernel at $a \in \A$ is the following sub-vector-space of $\Tr_a\A$:
	\[
		\ker_a\gf = \setof{\ub \in \Tr_a\A}{\forall \wb \in \Tr_a\A, {\braket\ub\wb}_\gf = 0} = \setof{\ub \in \Tr_a\A}{\left \|\ub\right \|_\gf=0}
	\]
\end{defi}

The projection of $\Tr\A$ turns it into a projection structure $\bundle[\pi_{\Tr\A}]{\ker\gf}\A$ but, in general, it fails to be a vector bundle as, among other things, the scalar field $\dim\left (\ker_a\gf\right )$ on $\A$ may not be uniform. A pseudo-metric is then a metric if and only if $\ker_a\gf = \{\bold{0}_a\}$ for all $a \in \A$. In the special case where $\gf$ is compatible one sees that $\left \|\gammab\cdot\ub-\varthetab\cdot\ub\right \|_\gf = \left \|\ub-\ub\right \|_\gb = 0$ meaning that the kernel is non-trivial. More precisely, one has the following:\\

\begin{lemm}
	\label{lemm:kergf}
	Let $\bundle\A\AA$ be a vector bundle and $\gf$ a pseudo-metric on $\A$ compatible with a metric $\gb$ on $\AA$, a connection $\gammab$ on $\A$ and a solder form $\varthetab$ on $\A$. Then one has the following inclusion:
		\quant{\forall a \in \A}{\Hr^{\gammab-\varthetab}\A \subset \ker_a\gf}
leading to the following inequalities:
		\quant{\forall a \in \A}{\dim\left (\AA\right ) \leq \dim\left (\ker_a\gf\right ) \leq \dim\left (\F_A\right )}
In particular, when $\dim\left (\AA\right ) = \dim\left (\F_A\right )$ one has
		\[
			\ker\gf = \Hr^{\gammab-\varthetab}\A
		\]
\end{lemm}

\begin{proof}	
	First, one has:
	\[
		\Tr\pi_\A\cdot\left (\gammab-\varthetab\right ) = \Tr\pi_\A\cdot\gammab-\Tr\pi_\A\cdot\varthetab = \Id - \bold0
	\]
	meaning that $\gammab-\varthetab$ is an affine connection. Then, since $\gf$ is compatible, one has for any $\ub \in \Tr_\overline{a}\AA$:
	\algn{
		\braket{\gammab_a\cdot\ub-\varthetab_a\cdot\ub}{\gammab_a\cdot\ub-\varthetab_a\cdot\ub}_\gf &= \braket{\ub-\ub}{\ub-\ub}_\gb = 0
	}
	showing that $\Hr^{\gammab-\varthetab}\A \subset \ker_a\gf$. Since ${\gammab-\varthetab : \Tr_\overline{a}\AA\longto \Hr^{\gammab-\varthetab}}\A$ is injective, one then has:
	\algn{
			\dim\left (\AA\right ) &= \dim\left (\Tr_\overline{a}\AA\right )\\
								   &= \dim\left (\Hr^{\gammab-\varthetab}\A\right )\\
								   &\leq \dim\left (\ker_a\gf\right )
		}
	For the second inequality, one uses the following form for $\gf_a$ \partxt{with $a \in \A$}:
	\[
		\gf_a:\fun{\Tr_a\A}{\Tr^\star_a\A}{\ub}{{\bra\ub}_\gf = \ub \mapsto {\braket\ub\ub}_{\gf}}
	\]
	This application is a linear application between vector spaces whose kernel is $\ker\gf_a=\ker_a\gf$. Since $\gf$ is compatible, $\gf_a$ reduces on $\Hr_a^\gammab\A$ into $\at{\gf_a}_{\Hr_a^\gammab\A}: \ub \longmapsto \left [\wb \mapsto {\braket{\Tr_a\pi_\A\cdot\ub}{\Tr_a\pi_\A\cdot\wb}}_{\gb_\ov{a}}\right ]$. Since $\at{\Tr_a\pi_\A}_{\Hr_a^\gammab\A} : \Hr_a^\gammab\A \simeq \Tr_\ov{a}\AA$ is bijective \partxt{of inverse $\gammab_a$} one has:
	\[
		\rank\left (\gf_a\right ) \ge \rank\left (\at{\gf_a}_{\Hr_a^\gammab\A}\right ) = \rank\left (\gb_\ov{a}\right ) = \dim\left (\AA\right )
	\]
	This means that:
		\algn{
			\dim\left (\ker_a\gf\right ) &= \dim\left (\Tr^\star_a\A\right ) - \rank\left (\gf_a\right )\\
										 &= \dim\left (\A\right ) - \rank\left (\gf_a\right )\\
										 &= \dim\left (\F_\A\right ) + \dim\left (\AA\right ) - \rank\left (\gf_a\right )\\
										 &\leq \dim\left (\F_\A\right ) + \dim\left (\AA\right ) - \dim\left (\AA\right )\\
										 &\leq \dim\left (\F_\A\right )
		}
	If $\dim\left (\AA\right ) = \dim\left (\F_A\right )$ then $\dim\left (\ker_a\gf\right ) = \dim\left (\AA\right ) = \dim\left (\Hr_a^{\gammab-\varthetab}\A\right )$. One being a subspace of the other they have to be equal.
\end{proof}

\section{Material placement and induced geometry}
\label{sect:conf_pull_back}

In classical mechanics, one has a macroscopic material space $\BB$ placed into the macroscopic Euclidean ambient space $\EE$ using a macroscopic punctual placement\footnote{The somewhat standard term "placement" shall be used, avoiding the ambiguity of the term "configuration", used in the literature to refer to the placement map $\ov\varphi:\BB\longto\EE$ \cite[4]{eringen1998microcontinuum}\cite[25]{Marsden:1994}\cite[5]{rakotomanana2012geometric} but also its image space $\Imr\left (\ov\varphi\right ) \subset \EE$ \cite[50]{CHAD}\cite[112]{GON}\cite[28]{rakotomanana2012geometric}.} map $\ov\varphi:\BB\longto\EE$. Elements of $\BB$ are interpreted as macroscopic particles while elements of $\EE$ are seen as macroscopic positions. From the punctual placement map, one then induces a first order macroscopic placement map $\ov{\Fb} = \Tr\ov\varphi : \Tr\BB\longto\Tr\EE$. The later is used to pull the ambient geometry $-$ which only consist of the Euclidean metric $\gb$ $-$ back onto the material space via $\Gb = \Fb\strut^\star\cdot\gb\cdot\Fb$. Where $\strut^\star$ denotes transposition.\\

This process is crucial as, using this material geometry, one can measure deformations and postulate energies. The aim of this section is to extend this process to the generalised continua covered in this paper.

\subsection{Material space, ambient space and the standard interpretation}

Based on the discussion in \cref{cartan_vision}, the material and ambient spaces are modelled as affine bundles where the projections provide the macroscopic part of, respectively, a material particle or an ambient position. Those macroscopic parts live in the (macroscopic) spaces of classical mechanics. In fact, one can go even further and require that every object of the micro-structured model has a macroscopic part living in the classical model. This thinking leads to the following axiom, which will be called the \underline{standard interpretation}:\\

\begin{axiom}{Standard interpretation of micro-structured continuum mechanics}{std_int}
	Each kinematic space of the continuum mechanics of micro-structured media ($\upmu$CM) has a projection structure whose basis is a kinematic space of the continuum mechanics of classical media (cCM).\\
		
	Each kinematic object of the $\upmu$CM belongs to such a projection structure. Its projection onto the associated base is referred to as its \underline{macroscopic part}.\\
	
	A mathematical quantity is \underline{physically observable} if and only if it belongs to a kinematic space of the cCM.
\end{axiom}

The first direct consequence of this axiom is that the material and ambient spaces must both have a projection structure over the usual macroscopic body $\BB$ and macroscopic $3$-dimensional Euclidean space $\EE$. These spaces are denoted $\B$ and $\E$ respectively:

\byalgn[.]{
	\text{The material bundle:}~\bundle\B\BB
}{
	\text{The ambient bundle:}~\bundle\E\EE
}

\renewcommand{\tmp}{In a trivialisation $\E \overset{\mathrm{loc.}}\simeq \EE \times \F_\E$ one can interpret the Cartesian product as a $\mathrm{macro}\times\mathrm{micro}$ decomposition. Since trivialisations are in general local, one would not be able to compare distant objects. When well-defined, having the same microscopic part would not be trivialisation-dependent \partxt{although its value in $\F_\E$ would}.}

In accordance with \cref{cartan_vision}, these spaces are required to be affine bundles. An element $x \in \E$ \partxt{resp. $X \in \B$} therefore has a macroscopic part $\overline{x} = \pi_\E(x) \in \EE$ \partxt{resp. $\overline{X} = \pi_\B(X) \in \BB$}. The notion of microscopic part, on the other-hand, requires a trivialisation to be defined\footnote{\tmp} and would be a fundamentally local concept. Following interpretation of the microscopic parts of \fullciteauthors{cartan1922generalisation,eringen1964nonlinear,kroner1980continuum,mindlin1964micro}{eringen1998microcontinuum,amari1962theory,grammenoudis2009micromorphic}, the dimension of the microscopic spaces $-$ that is, the rank of $\B$ and $\E$ $-$ are required to be $3$:

\algn{
	&\rank\left (\B\right )& &:=& &\dim\left (\F_\B\right )& &=& &3&\\
	&\rank\left (\E\right )& &:=& &\dim\left (\F_\E\right )& &=& &3&	
}
This corresponds to setting $k=3$ in \cref{cartan_vision}.

\subsection{Ambient geometry}

In order to be able to pull-back the ambient geometry on the material, the geometry of $\E$ needs to be prescribed first. Since $\E$ is the set of positions, it lacks any defaults caused by fibres miss-alignment or shape irregularity (since those are in the material $\B$). Heuristically, this means that all microscopic fibres will be of the same size and shape \partxt{$\gf$ is uniform}, perfectly aligned with the macroscopic space \partxt{$\varthetab$ is "trivial"} and perfectly aligned with each others \partxt{$\gammab$ is "trivial"}. In order to implement this vision, the ambient geometry is implemented as the following generalized Euclidean geometry:

\renewcommand{\tmp}{%
The term holonomic comes from the grec terms ὅλος (holos) meaning "whole" and νόμος (nomos) meaning "law". Several uses of this word exist in the literature. In this paper, holonomic will be used in the sense of to characterise an integrable system \cite[8]{peshkov2019continuum}, that is, a system expressible as a gradient.%
}

\begin{enumerate}
	\it The ambient space is $\E \simeq \Tr\EE$ with $\EE$ the $3\Dr$ Euclidean space.\\
	Let $\eb^{\ov\xb}=\left (\eb^{\ov\xb}_1, \eb^{\ov\xb}_2, \eb^{\ov\xb}_3\right )$ be the canonical coordinates on $\EE$. Then $\left (\eb^{\ov\xb}, \eb^\yb\right ) := \left (\eb^{\ov\xb}, \der{\eb^{\ov\xb}}\right )$ is a frame on $\EE$ \partxt{\ie{} a field of bases of $\Tr\EE$}. This means that $\E \simeq \Tr\EE \simeq \EE\times\EE$ is trivial with $\F_\E \simeq \EE$. Differentiating once more, one obtains the coordinates $\left (\left (\eb^{\ov\xb}, \eb^\yb\right ),\left (\der{\eb^{\ov\xb}},\der{\eb^\yb}\right )\right )$ of $\Tr\E$ or, equivalently, the trivialization $\Tr\E \simeq \Tr\EE\times\Tr\EE \simeq \EE^2\times\EE^2$. Since this last trivialization is Euclidean-valued, linear applications $\Tr_{x_1}\E \longto \Tr_{x_2}\E$ can be seen as linear applications $\EE^2\longto\EE^2$, that is, block-matrices.\\
	
	These trivializations are called \underline{holonomic coordinates}\footnote{\tmp}. In those coordinates, the macroscopic projections become:
		\[
			\pi_\E : \fun{\E\simeq\EE\times\EE}{\EE}{\bmat{\ov{x}\\y}}{\ov x} \hspace{5em} \Tr\pi_\E \equiv \bmat{\Id&\bold0}: \fun{\Tr\E\simeq\Tr\EE\times\Tr\EE}{\Tr\EE}{\bmat{\bmat{\ov x\\y}\\\bmat{\delta \ov x\\\delta y}}}{\bmat{\ov x\\\delta\ov x}}
		\]
	
	\it The connection $\gammab$ is the Levi-Civita connection which is linear and, in this trivialization, becomes:
\algn{
	\gammab \equiv \bmat{\Id\\\bold0} : \fun{\E\times_\EE\Tr\EE}{\Tr\E}{\bmat{\bmat{\ov x\\y}\\\bmat{\ov x\\\delta \ov x}}}{\bmat{\bmat{\ov x\\ y}\\\bmat{\delta \ov x\\\bold 0}}}
}
	
	\it The solder form $\varthetab$ is the canonical solder form which, in this trivialization, becomes:
\algn{
	\vartheta \equiv \bmat{\bold0\\\Id} : \fun{\E\times_\EE\Tr\EE}{\Vr\E \subset \Tr\E}{\bmat{\bmat{\ov x\\y}\\\bmat{\ov x\\\delta \ov x}}}{\bmat{\bmat{\ov x\\ y}\\\bmat{\bold 0\\\delta \ov x}}}
}
\end{enumerate}

In this geometry and this trivialization, the (compatible) pseudo-metric $\gf : \Tr\E \longmapsto \Tr^\star\E$ takes the following form:

\byalgn[\}]{
		\interproj &= \Tr\pi_\E + \varthetab^{-1}\cdot\vr_\gammab\\
			&= \Tr\pi_\E + \varthetab^{-1}\cdot\left (\Id - \gammab\cdot\Tr\pi_\E\right )\\
			&\equiv \bmat{\Id&\bold0} + \bmat{\bold0&\Id}\cdot\left (\bmat{\Id&\bold0\\\bold0&\Id}-\bmat{\Id\\\bold0}\cdot\bmat{\Id&\bold0}\right )\\
			&\equiv \bmat{\Id&\Id}\\
			\\
		\gf &= \left (\Tr\pi_\E + \varthetab^{-1}\cdot\vr_\gammab\right )^\star\cdot\gb\cdot\left (\Tr\pi_\E + \varthetab^{-1}\cdot\vr_\gammab\right )\\
			&\equiv \bmat{\Id\\\Id}\cdot\bmat{\Id}\cdot\bmat{\Id&\Id}\\
			&\equiv \bmat{\Id&\Id\\\Id&\Id}\\
}{
		\gf &: \fun{\Tr\E}{\Tr^\star\E}{\bmat{\bmat{\ov x\\y}\\\bmat{\delta \ov x\\\delta y}}}{\bmat{\bmat{\ov x\\y}\\\bmat{\delta \ov x+\delta y\\\delta \ov x+\delta y}}}
}

Even if $\E$ and $\Tr\EE$ are then technically isomorphic, their elements will be interpreted differently (as affine points and linear vectors respectively). This isomorphism shall therefore be forgotten for the remainder of this article. One could think that the solder form $\vartheta$ defined this way states that the macroscopic and microscopic worlds are of the same size since it does not seem to scale the vectors \partxt{see \cref{scale_discuss}}. However, this is not necessarily true: the coordinate system is simply chosen such that the scaling is $1$.\\

\subsubsection*{On the choice of reference connection}

In accordance with \cref{def:ref_con}, the ambient reference connection is set to be the Levi-Civita connection:
\algn{
	\RefCon\E &:= \gammab
}
and the material ambient reference connection is left free. Nevertheless, a more compact notation is introduced:
\algn{
	\Gammabref &:= \RefCon\B
}

\subsection{Physical acceptability \hyph punctual maps}
\label{sect:punctual_map}

Let $\varphi : \B \longto \E$ be a smooth map of the micro-structured system. Not all such maps are physically meaningful. This section aims at specifying some properties which make such a map physically acceptable.\\

Assume a local trivialisation on $\B$ is given and let $X\equiv\bmat{\ov X\\ Y}\in\B$ and $x = \varphi(X) \equiv \bmat{\ov x\\ y}$. Then, the last statement of \cref{ax:std_int} implies that there is a projection structure on the set of physically acceptable applications such that $\ov \varphi:\BB\longto\EE$ is the macroscopic part of $\varphi$. However, $\varphi$ maps $X$ $-$ whose macroscopic part is $\ov X$ $-$ to $x$ $-$ whose macroscopic part is $\ov x$. One therefore requires that the macroscopic part of the transformation by $\varphi$ is the transformation by $\ov \varphi$ of the macroscopic part. That is:
\[
	\ov x := \ov{\varphi(X)} = \ov{\varphi}\left (\ov X\right )
\]
or, equivalently, $\pi_\E\circ \varphi = \ov \varphi \circ \pi_\B$. This means that the macroscopic part of the image solely depends on the macroscopic part of its pre-image or, put differently, that $\varphi$ does not break apart microscopic fibres. Physically, this implies that if one were to zoom out and only see the macroscopic worlds $\BB$ and $\EE$, then $\varphi$ would still be expressible and not break apart macroscopic points (or "grains") in a indescribable way.\\

One therefore has $\varphi : \ifun{\B}{\E}{\bmat{\ov X\\Y}}{\bmat{\ov \varphi\left (\ov X\right )\\\varphi^\vr\left (\ov X, Y\right )} =: \bmat{\ov x\\y}}$. As mentioned in earlier sections, the coordinate $\ov X$ of $\BB$ is physically interpreted as a macroscopic quantity and the coordinate $Y$ of $\B_\ov X$ as a microscopic quantity. This paper will treat only materials in which the microstructure is several order of magnitudes smaller than the macrostructure, itself of order close to $1$. In other words, one has
\algn{
	|Y| &\ll \abs{\ov X} \simeq 1
}

A direct consequence is that $Y^2$ is several orders of magnitude smaller that $Y$. This means that, in $\varphi^\vr$'s Taylor expansion, higher orders will not have any physically significant contributions. The vertical coordinate $\varphi^\vr$ is therefore required to be affine. Since for two local trivialisations the change of vertical coordinates is affine in $Y$, this requirement is independent of the choice of coordinates.\\

All of these requirements can be summarised in the following mathematical concept, which formalises a notion of structure preservation:\\
\renewcommand{\tmp}{\D}
\newcommand{\tmpp}{\DD}

\begin{defi}{Morphisms}{morph}
	\renewcommand{\tmp}{\D}
	\renewcommand{\tmpp}{\DD}
	
	Let $\bundle\A\AA$ and $\bundle\tmp\tmpp$ be two projection structures. An application $f : \A \longto \tmp$ is called \underline{fibre-preserving} if and only if there exists an application $\overline{f} : \AA\longto\tmpp$ such that the following diagram is commutative:
\simplemorphism{\pi_\A}{\pi_\tmp}\A\AA\tmp\tmpp{f}{\overline{f}}
	which means that for any $(a_1, a_2) \in \A\times\A$:
		\[
			\pi_{\A}(a_1) = \pi_{\A}(a_2) \implies \pi_{\tmp}(f(a_1)) = \pi_{\tmp}(f(a_2)) =: \overline{f}\left (\pi_{\A}(a_2)\right )
		\]	
	The map $\overline{f}$ will be called the \underline{shadow} of $f$. One can also say that $f$ is over $\overline{f}$.\\
		
	For $\overline{a} \in \AA$, the following notation for the \underline{restriction} of $f$ over $\overline{a}$ is introduced:
		\[
			f_{\overline{a}} = \at{f}_{{\A}_{\overline{a}}} : {\A}_{\overline{a}} \longto {\tmp}_{\overline{f}\left (\overline{a}\right )}
		\]

\hrulefill\\

	Let $\bundle\A\AA$ and $\bundle\tmp\tmpp$ be two affine bundles. A smooth map $f : \A \longto \tmp$ is called an \underline{affine bundle morphism} if and only if $f$ is fibre-preserving, its shadow $\ov f$ is smooth and $f$ is \underline{affine}. Where the later means that $f_\overline{a}$ is affine on $\A_\overline{a}$ for all $\overline{a} \in \AA$.\\
	
	Respectively, if $\bundle\A\AA$ and $\bundle\tmp\tmpp$ are vector bundles, $f$ is called a \underline{vector bundle morphism} if and only if it is an affine bundle morphism and $f_\overline{a}$ is \underline{linear} on $\A_\overline{a}$ for all $\overline{a} \in \AA$.

\hrulefill\\

	The set of fibre-preserving applications from $\A$ to $\tmp$ is denoted $\fibpres{\A}{\tmp}$. It is equipped with a projection structure over the set $\tmpp^\AA$ of applications from $\AA$ to $\tmpp$, whose fibres are denoted $\fibpres[\overline{f}]{\A}{\tmp}$ \partxt{for $\overline{f}\in\tmpp^\AA$}:
	\[
		\pi_{\fibpres\A\tmp}:\fun{\fibpres\A\tmp}{\tmpp^\AA}{f}{\overline{f}}
	\]	
	The sets of affine (resp. vector) bundle morphisms from $\A$ to $\tmp$ are denoted $\aff\A\tmp\subset\fibpres{\A}{\tmp}$ \partxt{resp. $\l\A\tmp\subset\aff{\A}{\tmp}$}. They are equipped with the induced projection structures, whose fibres are denoted $\aff[\overline{f}]\A\tmp$ \partxt{resp. $\l[\overline{f}]\A\tmp$}, for $\overline{f}\in\smooth\AA\tmpp$:
	\byalgn{
		\pi_{\aff\A\tmp}:\fun{\aff\A\tmp}{\smooth\AA\tmpp}{f}{\overline{f}}
	}{
		\pi_{\l\A\tmp}:\fun{\l\A\tmp}{\smooth\AA\tmpp}{f}{\overline{f}}
	}
	
	In the case where the structure on $\A$ and/or $\tmp$ is ambiguous one can use the notations $\fibpres{\bundle\A\AA}{\bundle\tmp\tmpp}$, $\aff{\bundle\A\AA}{\bundle\tmp\tmpp}$ and $\l{\bundle\A\AA}{\bundle\tmp\tmpp}$.
\end{defi}

By construction, one then has that $\varphi:\B \longto \E$ is physically acceptable if and only if it is an affine bundle morphism $\varphi \in \aff\B\E$. Its macroscopic part is then its shadow $\ov \varphi \in \ov{\aff\B\E}=\smooth\BB\EE$.

\VOID{\begin{rema}
	In general, a bundle morphism $f : \A \longto \tmp$ is not required to be linear nor affine but to preserve the structure groups, meaning:
	\quant{\forall \ov a \in \A,~\forall \Ab \in \G_\A,~\exists \Db \in \G_\tmp\simeq\G_\A,}{f_\ov a \circ \Ab = \Db \circ {f_\ov a}}
	wish can be interpreted as stating that $f$ transforms the fibres as an element of the group. In this paper, and without any loss of generality, we shall only consider vector (resp. affine) bundles, whose structure group is $\GL\left (\F_\A\right )$ \partxt{resp. $\Aff\left (\F_\A\right )$}. Consequently, this notion and the previous definition are equivalent. In general, however, the structure group is fundamental as it prescribed the geometry of the material. For example, the structure group of Coserrat media is $\SO(3)$, the sphere \cite[8]{boyer2017poincare}, meaning the micro-structure can only be rotated.
\end{rema}}

\subsection{Physical acceptability \hyph first-order maps}
\label{sect:first_order_map}

Let $\Fb : \Tr\B \longto \Tr\E$ be a smooth map of the micro-structured system. Physically, the spaces $\Tr\B$ and $\Tr\E$ are interpreted as the set of first-order variations of particles and positions respectively. As for punctual maps, not all such maps are physically relevant. This section aims at specifying some properties which make such a map physically acceptable. In classical mechanics, physically acceptable maps are the first-order transformations associated to a punctual transformation, that is, gradient maps $\Tr\ov\varphi : \Tr\BB \longto \Tr\EE$. One therefore needs to generalise this property to microstructured media.\\

First of all, $\Fb$ must be over a punctual transformation. That is, if two vectors of $\Tr\B$ are over the same point of $\B$, then their images must be over a common point of $\E$. Secondly, such a mapping of vectors must be linear, which is well defined since the pre-images and images are over common points and can therefore be summed. These two properties are exactly equivalent to saying that $\Fb$ is a morphism $\Fb \in \l{\bundle{\Tr\B}\B}{\bundle{\Tr\E}\E}$ for the punctual projections.\\

$\Fb$ is therefore over a certain shadow $\varphi : \B \longto \E$. Since $\Fb$ needs to be physically acceptable, its shadow $\varphi$ has to be physically acceptable. That is, $\varphi \in \aff\B\E$. In trivialising affine coordinates, $\varphi$ and $\Fb$ take the forms
\byalgn{
	\varphi :& \fun{\B}{\E}{\bmat{\ov X\\Y}}{\bmat{\ov \varphi\left (\ov X\right )\\\varphi^\vr\left (\ov X\right )\cdot Y + {\overrightarrow{t}}^\vr\left (\ov X\right )}}
}{
	\Fb_{\bmat{\ov X\\Y}} \equiv& \bmat{\Fb_\hr^\hr\left (\ov X, Y\right ) & \Fb_\vr^\hr\left (\ov X, Y\right ) \\ %
										\Fb_\hr^\vr\left (\ov X, Y\right ) & \Fb_\vr^\vr\left (\ov X, Y\right )}
}

Here, \cref{ax:std_int} comes into play. Indeed, the standard interpretation stipulates that $\Fb$ must be part of a projection structure such that its macroscopic part is of the form $\ov{\Fb} : \Tr\BB \longto \Tr\EE$. Using the same reasoning as for punctual transformations, the macroscopic part of the transformation by $\Fb$ is required to be the transformation by $\ov\Fb$ of the macroscopic part. That is:
\quant{\forall \Ub \in \Tr\B,}{\ov \ub := \ov{\Fb\cdot\Ub} = \ov{\Fb}\cdot\ov \Ub}
or, equivalently, $\Tr\pi_\E\circ \Fb = \ov \Fb \circ \Tr\pi_\B$. This requirement is equivalent to saying that $\Fb$ is fibre-preserving for the macroscopic projections:
\[
	\Fb \in \fibpres{\bundle[\Tr\pi_\B]{\Tr\B}{\Tr\BB}}{\bundle[\Tr\pi_\E]{\Tr\E}{\Tr\EE}}
\]

This forces $\Fb_\vr^\hr$ to be zero, as stated in the following lemma:\\

\renewcommand{\tmp}{\D}
\renewcommand{\tmpp}{\DD}
\begin{lemm}
	\renewcommand{\tmp}{\D}
	\renewcommand{\tmpp}{\DD}
	\label{lemm:fvh_eq_zero}
	Let $\bundle\A\AA$ and $\bundle\tmp\tmpp$ be two affine bundles and $\Lb \in \l{\Tr\A}{\Tr\tmp}$. Then:
	\algn{
		\boxed{\Lb_\vr^\hr = 0} &\iff \Lb \in \fibpres{\bundle[\Tr\pi_{\A}]{\Tr{\A}}{\Tr{\AA}}}{\bundle[\Tr\pi_{\tmp}]{\Tr{\tmp}}{\Tr{\tmpp}}}
	}
	in which case its shadow as an element of $\fibpres{\bundle[\Tr\pi_{\A}]{\Tr{\A}}{\Tr{\AA}}}{\bundle[\Tr\pi_{\tmp}]{\Tr{\tmp}}{\Tr{\tmpp}}}$ is
	\algn{
		\ov{\Lb} &= \Tr\pi_\tmp \cdot \Lb \cdot \RefCon{\A}
	}
\end{lemm}

\begin{proof}
	One has the following sequences of implications, where the left side proves the direct sense while the right side shows that $\ov\Lb$ verifies $\ov\Lb\cdot\Tr\pi_\A = \Tr\pi_{\tmp}\cdot\Lb$, which implies the indirect sense:
	\byalgn{
		\Tr\pi_\A &= \Tr\pi_\A\cdot\hr_{\RefCon{A}}\\
		\ov{\Lb}\cdot\Tr\pi_\A &= \ov{\Lb}\cdot\Tr\pi_\A\cdot\hr_{\RefCon{A}}\\
		\Tr\pi_\tmp\cdot\Lb &= \Tr\pi_\tmp\cdot\Lb\cdot\hr_{\RefCon{A}}\\
		\RefCon{\tmp}\cdot\Tr\pi_\tmp\cdot\Lb &= \RefCon{\tmp}\cdot\Tr\pi_\tmp\cdot\Lb_\hr\\
		\hr_\RefCon{\tmp}\cdot\Lb &= \hr_\RefCon{\tmp}\cdot\Lb_\hr\\
		\Lb^\hr &= \Lb_\hr^\hr\\
		\Lb_\vr^\hr + \Lb_\hr^\hr &= \Lb_\hr^\hr\\
		  \Lb_\vr^\hr &= \bold{0}
}{
		\ov\Lb\cdot\Tr\pi_\A
			&= \Tr\pi_\tmp \cdot \Lb \cdot \RefCon{\A}\cdot\Tr\pi_\A\\
			&= \Tr\pi_\tmp \cdot \Lb \cdot \hr_{\RefCon{\A}}\\
			&= \Tr\pi_\tmp \cdot \Lb_\hr\\
			&= \Tr\pi_\tmp \cdot \hr_{\RefCon{\A}}\cdot\Lb_\hr\\
			&= \Tr\pi_\tmp \cdot \Lb_\hr^\hr\\
			&= \Tr\pi_\tmp \cdot \Lb^\hr \hspace{5em}\text{since $\Lb_\vr^\hr = \bold0$}\\
			&= \Tr\pi_\tmp \cdot \hr_{\RefCon{\A}} \cdot \Lb\\
			&= \Tr\pi_\tmp \cdot \Lb
	}
\end{proof}

Furthermore, since $\ov{\Fb} : \Tr\BB \longto \Tr\EE$ is the macroscopic part of $\Fb$, it is required to be a map of the classical macroscopic model. As stated above, this means that $\ov{\Fb}$ is a gradient. Since $\ov{\Fb}$ is over $\ov \varphi$, this means
\begin{equation}
	\label{eq:ovF=Tovphi}
	\ov{\Fb} = \Tr\ov \varphi
\end{equation}
This determines $\Fb^\hr_\hr$. Physically this means that, if one were to zoom out and pause time, macroscopic vectors would be mapped via a gradient map, as is the case in classical continuum mechanics.\\

On the opposite side lie the microscopic spaces. In this paper, all materials considered will have the property that their microscopic spaces are a scaled-down version of their macroscopic spaces. This last sentence must be interpreted as stating that the classical theory of continuum mechanics is also valid on the microscopic spaces.\\

Mathematically, since $\Fb$ is fibre-preserving, microscopic vectors are mapped to microscopic vectors. The restriction of $\Fb$ on the microscopic spaces is therefore well-defined and corresponds to $\Fb_\vr^\vr$. The aforementioned physical requirement implies that $\Fb_\vr^\vr$ must also be a gradient. One therefore requires that for all $X \equiv \bmat{\ov X\\ Y} \in \B$:
\begin{equation}
	\label{eq:Fvv=Tphivv}
	\at{\Fb}_{\Vr\B} = \at{\left (\Tr\varphi\right )}_{\Vr\B}
\end{equation}
where $\at{\left (\Tr\varphi\right )}_{\Vr_X\B} \simeq \Tr_X\varphi_\ov X \simeq \varphi^\vr\left (\ov X\right )$ through $\Vr_X\B \simeq \Tr_X\B_\ov X$. This means that neither the purely macroscopic nor the purely microscopic terms break the holonomy. If the later is to be broken, it must therefore be by the coupling term $\Fb_\hr^\vr$. The later has been, until now, a generic term depending on $\ov X$ and $Y$. Since $Y$ is small, this dependency is required to be affine in $Y$.\\

All those properties are summarised in what are called physically acceptable (tangential) maps:\\

\begin{defi}{Physically acceptable tangential map}{lpa}
	Let $\bundle{\A}{\AA}$ and $\bundle{\D}{\DD}$ be two affine bundles. A map $\Lb : \Tr\A \longto \Tr\D$ is said to be a \underline{physically acceptable (tangential) map} from $\Tr\A$ to $\Tr\D$ if and only if:
	\begin{enumerate}
		\it $\Lb$ is a morphism for the punctual projection with shadow $\ell \in \aff\A\D$.
		\[
			\Lb \in \l[\aff\A\D]{\bundle{\Tr{\A}}{\A}}{\bundle{\Tr{\D}}{\D}}
		\]
		\it $\Lb$ is fibre-preserving for the macroscopic projection, with shadow $\ov{\Lb} := \Tr\ov\ell \in \l[\ov\ell]{\Tr\AA}{\Tr\DD}$.
		\[
			\Lb \in \fibpres[\Tr\ov\ell]{\bundle[\Tr\pi_{\A}]{\Tr{\A}}{\Tr{\AA}}}{\bundle[\Tr\pi_{\D}]{\Tr{\D}}{\Tr{\DD}}}
		\]
		\it $\Lb$ is a gradient on the micro-fibres.
	%\footnote{Here, $\op{iso}_{\Vr_a\A} : \Vr_a\A \simeq \Tr_a\A_{\ov a}$ and $\op{iso}_{\Vr_{\ell(a)}\D}^{-1} : \Tr_{\ell(a)}\D_{\ov{\ell}(\ov a)} \simeq \Vr_{\ell(a)}\D$ are the canonical isomorphisms.}
	%	\algn{
	%		&\forall a \in \A,& &\at{\Lb}_{\Vr_X\A}& &:& &\Vr_a\A \xrightarrow{\hspace{3em}} \Vr_{\ell(a)}\D&\\
	%		&&&&				 		 &=& &\op{iso}_{\Vr_{\ell(a)}\D}^{-1} ~~\circ~~ \Tr_a{\ell_{\ov a}}~~\circ~~ \op{iso}_{\Vr_a\A}&
	%	}
		\algn{
			\at{\Lb}_{\Vr\A} &= \at{\left (\Tr\ell\right )}_{\Vr\A}
		}
		\it for all $\ub \in \Tr_a\A$ and every trivialising affine frame on $\Tr\A$ and $\Tr\D$ (see \cref{def:triv_coord}),\\
		the vectorial coordinate of $\Lb_a\cdot\ub \in \Tr_{\ell(a)}\D$ is \textbf{affine} in the vertical coordinate of $a\in\A$.
	\end{enumerate}
The \underline{set of physically acceptable tangential maps} from $\Tr\A$ to $\Tr\D$ is denoted
\algn{
	\lpa{\Tr\A}{\Tr\D} &\subset \l[\aff\A\D]{\bundle{\Tr{\A}}{\A}}{\bundle{\Tr{\D}}{\D}}\\
					 &\qquad \cap \fibpres{\bundle[\Tr\pi_{\A}]{\Tr{\A}}{\Tr{\AA}}}{\bundle[\Tr\pi_{\D}]{\Tr{\D}}{\Tr{\DD}}}
}
It inherits from both the projection structures of $\l[\l\A\D]{\bundle{\Tr{\A}}{\A}}{\bundle{\Tr{\D}}{\D}}$ and $\fibpres[\Tr\ov\ell]{\bundle[\Tr\pi_{\A}]{\Tr{\A}}{\Tr{\AA}}}{\bundle[\Tr\pi_{\D}]{\Tr{\D}}{\Tr{\DD}}}$. In accordance with previous notations, the over-line notation $\overline{\Lb} : \Tr\AA \longto \Tr\DD$ for $\Lb \in \lpa{\Tr\A}{\Tr\D}$ shall be used to refer to its shadow as an element of the later \partxt{\ie{} its "macroscopic" shadow}.
\end{defi}

Intuitively, one has that $\Fb : \Tr\B \longto \Tr\E$ is physically acceptable if and only if it is over a physically acceptable map $\varphi \in \l\B\E$, $\left (\Fb_\hr^\hr, \Fb_\vr^\vr\right )$ are identifiable with $\left (\Tr\ov \varphi, \left (\Tr \varphi\right )_\vr\right )$ and $\Fb_\hr^\vr$ is affine in the microscopic punctual coordinate. More formally, if one specifies everything in a local trivialisation, one gets the following lemma:\\

\begin{lemm}
	Let $\Lb : \Tr\A \longto \Tr\D$ be a smooth map between vector bundles. Then $\Lb \in \lpa{\Tr\A}{\Tr\D}$ if and only if, in every affine trivialising coordinates on $\A$ and $\D$, $\Lb$ can be expressed in the induced frames in the following form:
	\algn{
			&\Lb& &:& &\fun{\Tr\A \overset{\mathrm{loc.}}\equiv \Tr\left (\AA\times\F_\A\right )}{\Tr\D\overset{\mathrm{loc.}}\equiv \Tr\left (\DD\times\F_\D\right )}{%
			\bmat{%
				\bmat{\ov X\\ Y}\\%
				\bmat{\delta \ov X \\\delta Y}}%
			}{%
			\bmat{%
				\bmat{%
					\ov\ell\left (\ov X\right )%
					\\~\\%
					\ell^\vr\left (\ov X\right )\cdot Y%
						&+&	t^\vr\left (\ov X\right )
				}\\~\\%
				\bmat{%
					\ider{\ov X}{\ov \ell} \cdot \delta \ov X%
					\\~\\%
					\ell^\vr\left (\ov X\right )\cdot\delta Y%
						+	\left [ {\Lb}^\mathrm{coupl}\left (\ov X\right )\cdot Y \right .
						+ \left . {\Tb}^\mathrm{coupl}\left (\ov X\right )\right ] \cdot \delta \ov X%
					}%
				}%
			}
}
with $\ov \ell$, $\ell^\vr$, $t^\vr$, $\Lb^\mathrm{coupl}$ and ${\Tb}^\mathrm{coupl}$ all smooth and trivialisation-dependent and where $A\cdot b$ \partxt{resp. $A \cdot b \cdot c$} denotes linear \partxt{resp. bilinear} dependency \partxt{everything is trivialised, hence in $\RR^n$, $n \in \NN$}.\\

In such a case, $\displaystyle\left [\Lb^\mathrm{coupl}\left (\ov X\right )\cdot Y + {\Tb}^\mathrm{coupl}\left (\ov X\right )\right ]$ is the coordinate representation of $\Lb_\hr^\vr$. When $\Lb = \Tr\ell$, this is $\left [\ider{\ov X}{\ell^\vr}\cdot Y + \ider{\ov X}{t^\vr}\right ]$. In general, however, $\Lb^\mathrm{coupl}$ and $\Tb^\mathrm{coupl}$ are entirely free.
\end{lemm}

\subsection{Placement map}

The previous discussion specifies how the desired generalised placement maps must be physically acceptable. However, this condition does not suffice. In particular, \cref{ax:std_int} states that the macroscopic part are classical placement maps, which are required to be differential embeddings. This relates to the non-interpenetrability of matter. As this is not a purely macroscopic phenomena, the whole generalised placement map is required to be a (physically acceptable) differential embedding. This leads to the following definitions:\\

\begin{defi}{Generalised placement maps}{conf}

An application $\varphi : \B \longto \E$ is a \underline{generalized punctual placement map} if and only if:
	\begin{enumerate}
		\it $\varphi$ is a $\C^1$-differential embedding. That is, it is a differentiable injective bi-continuous map onto its image with an everywhere injective differential.
		\it $\varphi \in \aff\B\E$ is an affine bundle morphism \partxt{\ie{} is physically acceptable}
	\end{enumerate}

\hrulefill\\

An application $\Fb : \Tr\B \longto \Tr\E$ is a \underline{generalized first-order placement map} if and only if:
	\begin{enumerate}
		\it $\Fb$ is a differential embedding
		\it $\Fb$ is physically acceptable:
		\[
			\Fb \in \lpa{\Tr\B}{\Tr\E}
		\]
		\it The punctual shadow of $\Fb$ \partxt{for the punctual projections $\pi_{\Tr\B}$ and $\pi_{\Tr\E}$} is a punctual configuration $\varphi : \B \longhookrightarrow \E$.
	\end{enumerate}
	
The \underline{set of generalized first-order placement map} from $\B$ to $\E$ will be denoted
\algn{
	\conf{\Tr\B}{\Tr\E} \subset \lpa{\Tr\B}{\Tr\E}
}
\end{defi}

In this definition three things are to be noticed:
\begin{enumerate}
	\it In accordance with \cref{def:lpa} and \cref{ax:std_int}, the overline notation is used for the macroscopic shadow $\overline{\Fb} : \Tr\BB \longhookrightarrow \Tr\EE$ of $\Fb$, not to be confused with its punctual shadow $\varphi : \B \longto \E$.
	\it Although the micro-spaces are interpreted as infinitesimal neighbourhoods, the microscopic part of $\varphi$ is not required to be the "gradient" of the macroscopic part\footnote{This notion is hill-defined in general but makes sense when $\B\simeq\Tr\BB$ $-$ which is the case for crystals $-$ or more generally when $\B\simeq\Tr_\BB\MM$ for $\BB \subset \MM$ and $\dim\left (\MM\right )=3$ $-$ which is the case for shells \partxt{$\dim\left (\BB\right )=2$} and beams \partxt{$\dim\left (\BB\right )=1$}.}. This freedom is necessary in order to allow for torsion (\ie{} dislocations) in the material \cite[8]{peshkov2019continuum}.
	\it Similarly, nothing states that $\Fb$ corresponds to $\Tr\varphi$. In fact, previous works \partxt{\cite[16]{nguyen2021tangent}, \cite{nguyen2021geometric,KATA}} showed that this cannot be the case if the material has some curvature. Requiring the curvature to be zero is pretty common in the literature and corresponds to the requirement that the material should only exhibit dislocation defects and no disclination \cite{kroner1980continuum, peshkov2019continuum, sahoo1984elastic, le1996determination}.
\end{enumerate}

Most of the properties enumerated in \cref{def:conf} can be summarized in the circular commutative diagram of \cref{fig:comm_diag}.

\begin{figure}[!htb]
\begin{center}
\begin{tikzpicture}[->,>=stealth',shorten >=1pt,auto,node distance=2.8cm,semithick]
  \tikzstyle{every state}=[fill=none,draw=none,text=black]

  \node[state]		(TM) 							{$\Tr\B$};
  \node[state]		(M)		[below left of = TM]	{$\B$};
  \node[state]		(TBB) 	[below right of = TM]	{$\Tr\BB$};
  \node[state]		(BB) 	[below right of = M]	{$\BB$};%

  \node[state]		(TE) 	[above of = TM]			{$\Tr\E$};
  \node[state]		(E)		[left of = M]			{$\E$};
  \node[state]		(TEE) 	[right of = TBB]		{$\Tr\EE$};
  \node[state]		(EE) 	[below of = BB]			{$\EE$};

  \draw [-{To}{To}] (TM) 	edge [bend right]		node [above left] 	{$\pi_{\Tr\B}$} (M)
   		(TM) 	edge [bend left]		node [above right] 	{$\Tr\pi_\B$} (TBB)
  		(M) 	edge [bend right]		node [below left] 	{$\pi_\B$} (BB)
 		(TBB) 	edge [bend left]		node [below right] 	{$\pi_{\Tr\BB}$} (BB)%
 		
 		(TE) 	edge [bend right=35]	node [above left] 	{$\pi_{\Tr\E}$} (E)
  		(TE) 	edge [bend left=35]		node [above right] 	{$\Tr\pi_\E$} (TEE)
  		(E) 	edge [bend right=35]	node [below left] 	{$\pi_\E$} (EE)
 		(TEE) 	edge [bend left=35]		node [below right] 	{$\pi_{\Tr\EE}$} (EE);
 		
  \draw [{Hooks[right, width=10, length=4]}->]
  		(TM)	edge					node				{$\Fb$} (TE)
 		(M)		edge					node				{$\varphi$} (E)
 		(TBB)	edge					node				{$\Tr\overline\varphi$} (TEE)
 		(BB)	edge					node				{$\overline\varphi$} (EE);
\end{tikzpicture}
\end{center}
\caption{The circular commutative diagram associated to a generalized first-order placement map $\Fb : \Tr\B \longto \Tr\E$. All arrows represent smooth maps, $\twoheadrightarrow$ means the map is surjective and $\hookrightarrow$ means it is injective. The top-left quadrant comes from $\Fb \in \l{\bundle{\Tr\B}\B}{\bundle{\Tr\E}\E}$, the top-right from $\Fb \in \fibpres{\bundle[\Tr\pi_\B]{\Tr\B}{\Tr\BB}}{\bundle[\Tr\pi_\E]{\Tr\E}{\Tr\EE}}$, the bottom-left from $\varphi \in \aff\B\E$ and the bottom-right from $\Tr\ov{\varphi}\in\l{\Tr\BB}{\Tr\EE}$.}
\label{fig:comm_diag}
\end{figure}

\subsubsection*{Inverse map}

Since $\Fb$ is injective $\Fb^{-1} : \Imr\left (\Fb\right ) \longto \B$ is a well-defined application. One may wonder whether $\Fb^{-1}$ is fibre-preserving. This question is answered by the following result:

\begin{lemm}
	Let $\bundle{\A}{\AA}$ and $\bundle{\D}{\DD}$ be two projection structures and $f \in \fibpres{\A}{\D}$ be injective with an injective shadow. Then:
	\[
		f^{-1} \in \fibpres{\bundle[\pi_{\A}]{\Imr\left (f\right )}{\Imr\left (\ov{f}\right )}}{\A}~~~\text{and}~~~\ov{f^{-1}}=\ov{f}^{-1}
	\]
\end{lemm}

\begin{proof}
	By contradiction, if $f^{-1}$ is not fibre-preserving then there exist $\ov{y} \in \AA^2$ and $(y, y') \in \A_\ov{y}^2$ such that $\ov{f^{-1}(y)} \neq \ov{f^{-1}(y')}$. But since $\ov{f}$ is injective :
	\algn{
					\ov{f^{-1}(y)} \neq \ov{f^{-1}(y')}
		\implies&	&\ov{f}\left (\ov{f^{-1}(y)}\right ) &\neq \ov{f}\left (\ov{f^{-1}(y'}\right )\\
		\implies&	&\ov{f\left (f^{-1}(y)\right )} &\neq \ov{f\left (f^{-1}(y')\right )}\\
		\implies&	&\ov{y} &\neq \ov{y'}
	}
	For the shadow one has, for $(a, d) \in \A\times\D$:
	\byalgn{
		\ov{f}\left (\ov{f^{-1}}(\ov{d})\right )
		&=& \ov{f}\left (\ov{f^{-1}(d)}\right )\\
		&=& \ov{f\left (f^{-1}(d)\right )}\\
		&=& \ov{d}
	}{
		\ov{f^{-1}}\left (\ov{f}(\ov{a})\right )
		&=& \ov{f^{-1}}\left (\ov{f(a)}\right )\\
		&=& \ov{f^{-1}\left (f(a)\right )}\\
		&=& \ov{a}
	}
\end{proof}

This result implies the following crucial lemma:

\begin{lemm}
Let $\bundle{\A}{\AA}$ and $\bundle{\D}{\DD}$ be two affine bundles and $\Lb \in \conf{\Tr\A}{\Tr\D}$. Then:
$$\boxed{\Lb^{-1} \in \conf{\Imr\left (\Lb\right )}{\Tr\D}}$$
with the following structures on $\Imr\left (\Lb\right )$:
$$\bundle[\pi_{\Tr\D}]{\Imr\left (\Lb\right )}{\Imr\left (\ov{\Lb}\right )}~~~~~\text{and}~~~~~\bundle[\Tr\pi_\D]{\Imr\left (\Lb\right )}{\Imr\left (\ell\right )}$$
\end{lemm}

\subsubsection*{Block-wise decomposition}

Recall \cref{lemm:fvh_eq_zero} stating that $\Fb \in \lpa{\Tr\A}{\Tr\D}$ takes a lower-triangular form. Consequentially, the inverse also takes a special form:

\begin{lemm}
	Let $\bundle{\A}{\AA}$ and $\bundle{\D}{\DD}$ be two affine bundles and $\Lb \in \lpa{\Tr\A}{\Tr\D}$ be an injective map with both shadows injective. Then:
	\algn{
		\Lb_\hr^\hr \cdot \left (\Lb^{-1}\right )_\hr^\hr &= \hr_{\RefCon{\D}}
					&\hspace{5em}&  \Lb_\vr^\vr \cdot \left (\Lb^{-1}\right )_\vr^\vr = \vr_{\RefCon{\D}}\\
		\left (\Lb^{-1}\right )_\hr^\hr \cdot \Lb_\hr^\hr &= \hr_{\RefCon{\A}}
					&&				 \left (\Lb^{-1}\right )_\vr^\vr \cdot \Lb_\vr^\vr = \vr_{\RefCon{\A}}
	}
	and
	\algn{
		\left (\Lb^{-1}\right )_\hr^\vr &= -\left (\Lb^{-1}\right )_\vr^\vr\cdot\Lb_\hr^\vr\cdot\left (\Lb^{-1}\right )_\hr^\hr
	}
\end{lemm}

\begin{proof}
	\algn{
		\Lb_\hr^\hr\cdot\left (\Lb^{-1}\right )_\hr^\hr &= \Lb_\hr^\hr\cdot\left (\Lb^{-1}\right )_\hr &\hspace{5em}& \text{since $\Lb_\hr^\hr\cdot\left (\Lb^{-1}\right )_\hr^\vr = \bold{0}$}\\
				&= \Lb^\hr\cdot\left (\Lb^{-1}\right )_\hr && \text{since $\Lb_\vr^\hr = 0$}\\
				&=  \hr_{\RefCon{\D}}\cdot\Lb\cdot\Lb^{-1}\cdot\hr_{\RefCon{\D}}\\
				&=  \hr_{\RefCon{\D}}
	}
	The same reasoning applies for the other three equalities.
\end{proof}

The relations for the diagonal parts mean that, on $\Hr^\Gammabref\B$, $\left (\Fb^{-1}\right )_\hr^\hr$ is the inverse of $\Fb_\hr^\hr$. Similarly, $\Fb_\hr^\hr$ is the inverse of $\left (\Fb^{-1}\right )_\hr^\hr$ on $\Hr^\gammab\E$ and $\Fb_\vr^\vr$ and $\left (\Fb^{-1}\right )_\vr^\vr$ are inverses of one another on $\Vr\B$ and $\Vr\E$. These observations lead to the following notations:

\begin{nota}{Partial inverses}{part_inv}
	Let $\bundle{\A}{\AA}$ and $\bundle{\D}{\DD}$ be two affine bundles and $\Lb \in \lpa{\Tr\A}{\Tr\D}$ be an injective map with both shadows injective. The following notations are introduced:
	\algn{
		{\Lb_\hr^\hr}^{-1} &:= \left (\Lb^{-1}\right )_\hr^\hr
			&\hspace{5em}& {\Lb_\vr^\vr}^{-1} &:= \left (\Lb^{-1}\right )_\vr^\vr
	}
	called the \underline{pseudo-inverses} of $\Lb_\hr^\hr$ and $\Lb_\vr^\vr$ respectively.
\end{nota}

With these notations, $\Fb$ and $\Fb^{-1}$ take the following block-wise forms:
\algn{
	\Fb &\equiv \bmat{\Fb_\hr^\hr&\bold{0}\\\Fb_\hr^\vr&\Fb_\vr^\vr}
		&\hspace{5em}& \Fb^{-1} &\equiv \bmat{{\Fb_\hr^\hr}^{-1}&\bold{0}\\-{\Fb_\hr^\hr}^{-1}\cdot\Fb_\hr^\vr\cdot{\Fb_\vr^\vr}^{-1}&{\Fb_\vr^\vr}^{-1}}
}

The term pseudo-inverse comes from the fact that the products do not yield identities but projectors: $\Fb_\hr^\hr\cdot{\Fb_\hr^\hr}^{-1} = \hr_\gammab$, ${\Fb_\hr^\hr}^{-1}\cdot\Fb_\hr^\hr = \hr_{\Gammabref}$, etc. Therefore, these are actual inverses only on $\Hr^\gammab$, $\Hr^\Gammabref$, etc.

\subsection{Material geometry}
\label{sect:pull_back}

Having defined everything, the next step is to use the first-order placement map $\Fb$ to pull the ambient geometry back onto $\B$. Let $X \in \B$ and $\overline{\Ub} \in \Tr_\overline{X}\BB$. Then the placement map places ${X}$ at $x = \varphi\left (X\right ) \in \EE$ and $\overline{\Ub}$ at $\overline{\ub} = \overline{\Fb}\cdot\overline{\Ub} \in \Tr_\overline{x}\EE$. Using $\gammab$, one can lift this vector into a vector $\ub = \gammab_x\cdot\overline{\ub} \in \Hr^\gammab_x\E$. Using the injectivity of $\Fb_X$ one has that, if $\ub$ is in its image, there exist a unique $\Ub \in \Tr_X\B$ placed at $\ub$. This process defines a connection $\Gammab$ on $\B$, as stated in the following theorem:

\begin{theodef}{Material connection}{Gammab}
	Let $\Fb : \Tr\B \longto \Tr\E$ be a first-order placement map, $\varphi:\B\longto\E$ its associated punctual placement and $\gammab$ be a connection on $\E$. Then the following application, called the \underline{pull-back of $\gammab$ by $\Fb$}, is a well-defined affine connection on $\B$:
		$$\Gammab_X = \Fb^{-1}\cdot\gammab_{\varphi(X)}\cdot\overline{\Fb} \quad\quad\forall X \in \B$$
	Furthermore, its horizontal space $\Hr^\Gammab\B$ is the preimage of $\Hr^\gammab\E$:
		$$\Hr^\Gammab\B = \Fb^{-1}\left (\Hr^\gammab\E\right )$$
	and its horizontal projection is:
		$$\hr_\Gammab = \hr_\Gammabref - {\Fb_\vr^\vr}^{-1}\cdot\Fb_\hr^\vr$$
	or, equivalently:
		$$\Gammab = \left [\Id - {\Fb_\vr^\vr}^{-1}\cdot\Fb_\hr^\vr\right ]\cdot\Gammabref$$
One can use the notation $\Fb^*\gammab := \Gammab$ if the dependency on $\Fb$ needs to be made explicit\footnote{Once again $\ast$ denotes pull-backs, not to be confused with $\star$ for transposition.}.
\end{theodef}

\begin{proof}

	First, let $X \in \B$ and $\overline{\Ub} \in \Tr_\overline{X}\BB$. Since $\Fb \in \fibpres{\bundle[\Tr\pi_\B]{\Tr\B}{\Tr\BB}}{\bundle[\Tr\pi_\E]{\Tr\E}{\Tr\EE}}$, one has\footnote{%
		Recall that, for $\bundle\A\AA$ an affine bundle, $x \in \A$ and $\overline{\ub} \in \Tr_\overline{x}\AA$, the following notations are used
			$$\at{\Tr\A}_{\overline{\ub}} := \left [\Tr\pi_\A\right ]^{-1}\left (\left \{\overline{\ub}\right \}\right )\hspace{5em}\at{\Tr_x\A}_{\overline{\ub}} := \at{\Tr\A}_{\overline{\ub}} \cap \Tr_x\A$$%
		}%
		$$\Fb_X \cdot \at{\Tr_X\B}_{\overline{\Ub}} \subset  \at{\Tr_{\varphi(X)}\E}_{\overline{\Fb}\cdot\overline{\Ub}}$$
		and, since those are both vector spaces of dimension $3$ \partxt{thanks to the injectivity of $\Fb_X$}, the inclusion is in fact an equality. Since $\gammab$ is a connection on $\E$, one further has:
		$$\gammab_{\varphi(X)}\cdot\overline{\Fb}\cdot\overline{\Ub} \in \at{\Tr_{\varphi(X)}\E}_{\overline{\Fb}\cdot\overline{\Ub} }$$
		These relations imply that $\Gammab_X\cdot\overline{\Ub} := \Fb^{-1}\cdot\gammab_{\varphi(X)}\cdot\overline{\Fb}\cdot\overline{\Ub}$ is well defined and verifies
		$$\Gammab_X\cdot\overline{\Ub} \in \at{\Tr_X\B}_{\overline{\Ub} }$$
		meaning in particular that $\Tr\pi_\E \cdot \Gammab = \Id_{\Tr\BB}$.\\
		
		Secondly, the horizontal projection is given by:
		\algn{
			\hr_\Gammab :=& \Gammab\cdot\Tr\pi_\B\\
						=& \Fb^{-1}\cdot\gammab\cdot\overline{\Fb}\cdot\Tr\pi_\B\\
						=& \Fb^{-1}\cdot\gammab\cdot\Tr\pi_\E\cdot\Fb\\
						=& \Fb^{-1}\cdot\hr_\gammab\cdot\Fb\\
						=& \left (\Fb^{-1}\right )_\hr\cdot\Fb_\hr^\hr\\
						=& \left [{\Fb_\hr^\hr}^{-1}-{\Fb_\vr^\vr}^{-1}\cdot\Fb_\hr^\vr\cdot{\Fb_\hr^\hr}^{-1}\right ]\cdot\Fb_\hr^\hr\\
						=& \hr_\Gammabref-{\Fb_\vr^\vr}^{-1}\cdot\Fb_\hr^\vr
		}
		
		Furthermore, for $\Ub \in \Tr\B$, one has:
		\algn{
			\exists \ub \in \Tr\E,~\Fb\cdot\Ub = \hr_\gammab\cdot\ub%
						&\impliedby \exists \Ub' \in \Tr\B,~ &\Fb\cdot\Ub &= \hr_\gammab \cdot \Fb \cdot \Ub'\\
			\exists \ub \in \Tr\E,~\Fb\cdot\Ub = \hr_\gammab\cdot\ub%
						&\implies \exists \ub \in \Tr\E,~&\hr_\gammab \cdot \Fb \cdot \Ub &= \hr_\gammab\cdot\ub = \Fb\cdot \Ub\\
			\exists \ub \in \Tr\E,~\Fb\cdot\Ub = \hr_\gammab\cdot\ub%
						&\implies \exists \Ub' \in \Tr\B,~&\Fb\cdot\Ub &= \hr_\gammab \cdot \Fb \cdot \Ub'
		}		
		
		Meaning one has:
\algn{
	\Fb^{-1}\left (\setof{\hr_{\gammab}\cdot\ub}{\ub\in\Tr\E}\right ) &= \Fb^{-1}\left (\setof{\hr_{\gammab}\cdot\Fb\cdot\Ub'}{\Ub'\in\Tr\B}\right )
}		

		The horizontal space is then given by:
		\algn{
				\Hr^\Gammab\B
					&=& \setof{\hr_{\Gammab}\cdot\Ub}{\Ub\in\Tr\B}\\
					&=& \setof{\Fb^{-1}\cdot\hr_{\gammab}\cdot\Fb\cdot\Ub}{\Ub\in\Tr\B}\\
					&=& \Fb^{-1}\left (\setof{\hr_{\gammab}\cdot\Fb\cdot\Ub}{\Ub\in\Tr\B}\right )\\
					&=& \Fb^{-1}\left (\setof{\hr_{\gammab}\cdot\ub}{\ub\in\Tr\E}\right )\\
					&=& \Fb^{-1}\Hr^\gammab\E
		}
				
		Lastly, since $\Fb \in \lpa{\Tr\B}{\Tr\E}$, $\Fb^{-1} \in \lpa{\Tr\E}{\Tr\A}$ is affine in the punctual vertical coordinate, so is $\varphi$. Since $\gammab$ is affine, $\Gammab$ will be too.
\end{proof}

This process also works for the solder form and the pseudo-metric as stated by the following theorems:

\begin{theodef}{Material solder form}{Thetab}
	Let $\Fb : \Tr\B \longto \Tr\E$ be a first-order placement map, $\varphi : \B \longto \E$ its associated punctual placement map and $\varthetab$ a solder form on $\E$. Then the following application, called the \underline{pull-back of $\varthetab$ by $\Fb$}, is a well-defined solder form on $\B$:
		$$\Thetab_X = \Fb^{-1}\cdot\varthetab_{\varphi(X)}\cdot\overline{\Fb}\quad\quad\forall X \in \B$$
One can use the notation $\Fb^*\varthetab := \Thetab$ if the dependency on $\Fb$ needs to be made explicit\footnote{Once again $\ast$ denotes pull-backs, not to be confused with $\star$ for transposition.}.
\end{theodef}

\begin{proof}
	Using the same argument as in the proof of \cref{def:Gammab} one has
		\algn{
			\Fb\cdot\Vr\B &= \Fb\cdot\at{\Tr\B}_\bold{0}\\
						  &= \at{\Tr\E}_{\overline{\Fb}\cdot\bold{0}}\\
						  &= \Vr\E
			}
	where restrictions are for the macroscopic projections \partxt{$\Tr\pi_\B$ and $\Tr\pi_\E$} and $\bold0 \subset \Tr\BB$ is the zero section of $\Tr\BB$. This means that $\Thetab_X = \Fb^{-1}\cdot\varthetab_{\varphi(X)}\cdot\overline{\Fb}_X$ is well defined. Then one has, for $X \in \B$:
	\algn{
		\Tr\pi_{\B}\cdot\Thetab_X &= \Tr\pi_{\B}\cdot\Fb^{-1}\cdot\varthetab_{\varphi(X)}\cdot\overline{\Fb}_X\\
								  &= \ov{\Fb}^{-1}\cdot\Tr\pi_{\E}\cdot\varthetab_{\varphi(X)}\cdot\overline{\Fb}_X\\
								  &= \bold0
		}
	Meaning that $\Thetab$ is a solder-form. Lastly, $\Thetab$ is affine for the same reason $\Gammab$ is in \cref{def:Gammab}.
\end{proof}

\begin{rema}
	\label{rem:theta_holo}
	The third condition on $\Fb$ stated in \cref{def:lpa} \partxt{or \cref{eq:Fvv=Tphivv}}, namely that it must be a gradient on the micro-fibres, is equivalent to requiring $\Fb^*\varthetab = \left (\Tr\varphi\right )^*\varthetab$. Roughly speaking, this is because this equality looks like ${\Fb_\vr^\vr}^{-1}\cdot\varthetab\cdot\ov\Fb = {\left (\Tr\varphi\right )_\vr^\vr}^{-1}\cdot\varthetab\cdot\Tr\ov\varphi$ and the second condition of \cref{def:lpa} \partxt{or \cref{eq:ovF=Tovphi}} already states that $\ov\Fb = \Tr\ov\varphi$.
\end{rema}

\begin{theodef}{Material pseudo-metric}{Gf}
	Let $\Fb : \Tr\B \longto \Tr\E$ be a first-order placement map and $\gf : \Tr\E \longto \Tr^\star\E$ be a pseudo-metric on $\E$. Then the following application, called the \underline{pull-back of $\gf$ by $\Fb$}, is a well-defined pseudo-metric on $\B$:
		$$\Gf = \Fb^\star\cdot\gf\cdot\Fb$$
	Similarly, if $\gb : \Tr\EE \longto \Tr^\star\EE$ is a metric on $\EE$, the following pull-back is a well-defined metric on $\BB$:
		$$\Gb = \overline{\Fb}^\star\cdot\gb\cdot\overline{\Fb}$$
	and corresponds to the classical notion of the right Cauchy-Green tensor.\\
	
One can use the notations $\Fb^*\gf := \Gf$ and ${\ov\Fb}^*\gb := \Gb$ if the dependencies on $\Fb$ and $\ov\Fb$ need to be made explicit\footnote{Once again $\ast$ denotes pull-backs, not to be confused with $\star$ for transposition.}.
\end{theodef}

\begin{proof}
	Let $X \in \B$ and $\left (\Ub, \Wb\right ) \in \Tr_X\B\times\Tr_X\B$. Recall that one has:
		${\braket\Ub\Wb}_\Gf := {\braket{\Fb\cdot\Ub}{\Fb\cdot\Wb}}_\gf$
	Symmetry and non-negativity are then automatically verified since $\gf$ verifies it. The bi-linearity on the other hand comes from the linearity of $\Fb$. A similar statement holds for $\Gb$. The definiteness of $\Gb$ then comes from the injectivity of $\ov\Fb$ and the definiteness of $\gb$ as $\Gb\cdot\ov\Ub = 0 \implies \gb\cdot\ov\Fb\cdot\ov\Ub = 0 \implies \ov\Fb\cdot\ov\Ub=0 \implies \ov\Ub = 0$.
\end{proof}

Two key facts regarding the metrics and the connection are two be highlighted:
\begin{enumerate}
	\it The metrics $\Gb : \Tr\BB \longto \Tr^\star\BB$ and $\displaystyle\at{\Gf_\vr^\vr}_\ov X : \Vr\B_\ov X \longto \Vr\strut^\star\B_\ov X$, $\ov X \in \BB$ will always be curvature-free.
	\it The connection $\Gammab$ will, in general, not be the Levi-Civita connection associated to $\Gb$. In particular, it may possess some torsion and/or curvature.
\end{enumerate}

An important consequence is that, contrary to several models where the metric carries the curvature, here the curvature is carried by the (independent) connection $\Gammab$. More generally, $\Gb$ will be the Cauchy-Green tensor of classical mechanics of continuum media while $\Gammab$ and $\Thetab$ will be the one carrying the micro-structure data.\\

Regarding the value of $\Gf$, a thorough derivation of the computation leads to the following formulae:

\begin{lemm}
\label{lemm:Gf}
Let $\gf : \Tr\E \longto \Tr^\star\E$ be a pseudo-metric and $\Gf : \Tr\B \longto \Tr^\star\B$ its pull-back by $\Fb \in \lpa{\Tr\B}{\Tr\E}$. Then, one has:

\algn{
	\Gf_\hr^\hr &= \left (\Fb_\hr\right )^\star\cdot\gf\cdot\Fb_\hr &\hspace{2em}%
							 \Gf_\vr^\hr &= \left (\Fb_\hr\right )^\star\cdot\gf\cdot\Fb_\vr\\
				&= {\Fb_\hr^\hr}^\star\gf_\hr^\hr\Fb_\hr^\hr%
					+ {\Fb_\hr^\hr}^\star\gf_\vr^\hr\Fb_\hr^\vr%
					+ {\Fb_\hr^\vr}^\star\gf_\hr^\vr\Fb_\hr^\hr%
					+ {\Fb_\hr^\vr}^\star\gf_\vr^\vr\Fb_\hr^\vr&%
										 &= {\Gf_\hr^\vr}^\star\\
	\\
	\Gf_\hr^\vr &= \left (\Fb_\vr\right )^\star\cdot\gf\cdot\Fb_\hr &\hspace{2em}
							 \Gf_\vr^\vr &= \left (\Fb_\vr\right )^\star\cdot\gf\cdot\Fb_\vr\\
				&= {\Fb_\vr^\vr}^\star\gf_\hr^\vr\Fb_\hr^\hr%
					+ {\Fb_\vr^\vr}^\star\gf_\vr^\vr\Fb_\hr^\vr&%
										&= {\Fb_\vr^\vr}^\star\gf_\vr^\vr\Fb_\vr^\vr
}
\end{lemm}

\subsubsection*{On the compatibility of the pull-backs}

By use of these theorems one now has a way of pulling the geometrical tools from the ambient space to the material space using a given first order placement map. Nevertheless, these objects had certain relations. In particular, the pseudo-metric was compatible. These properties are preserved by the pull-back operation:

\begin{lemm}
	\label{lem:Gf_compat}
	Let $\Fb : \Tr\B \longto \Tr\E$ be a first-order placement map, $\gammab$ a connection on $\E$, $\varthetab$ a solder form on $\E$, $\gb$ a metric on $\EE$ and $\gf$ a compatible pseudo-metric on $\E$. Let $\Gammab$, $\Thetab$, $\Gb$ and $\Gf$ be their respective pull-backs by $\Fb$. Then:
	\begin{enumerate}
		\it $\Gf$ is compatible with $\Gammab$, $\Thetab$ and $\Gb$
		\it $\ker\Gf = \Hr^{\Gammab-\Thetab}\B$
	\end{enumerate}
\end{lemm}

\begin{proof}
	Let $X \in \B$ and let $\ov\Ub$, $\ov\Wb$, $\ov{\Ub}'$, $\ov{\Wb}'$ be vectors in $\Tr_\ov{X}\BB$. Then one has:
	
	{
		\algn{
			{\braket{\Thetab_X\cdot\ov{\Ub} + \Gammab_X\cdot\ov{\Wb}}{\Thetab_X\cdot{\ov{\Ub}}' + \Gammab_X\cdot{\ov{\Wb}}'}}_\Gf
				\hspace{-7em}&\\
				&= {\braket{\Fb\cdot\Thetab_X\cdot\ov{\Ub} + \Fb\cdot\Gammab_X\cdot\ov{\Wb}}{\Fb\cdot\Thetab_X\cdot{\ov{\Ub}}' + \Fb\cdot\Gammab_X\cdot{\ov{\Wb}}'}}_\gf\\
				&= {\braket{\varthetab_{\varphi(X)}\cdot\ov\Fb\cdot\ov{\Ub} + \gammab_{\varphi(X)}\cdot\ov\Fb\cdot\ov{\Wb}}{\varthetab_{\varphi(X)}\cdot\ov\Fb\cdot{\ov{\Ub}}' + \gammab_{\varphi(X)}\cdot\ov\Fb\cdot{\ov{\Wb}}'}}_\gf\\
				&= {\braket{\ov\Fb\cdot\ov{\Ub} + \ov\Fb\cdot\ov{\Wb}}{\ov\Fb\cdot{\ov{\Ub}}' + \ov\Fb\cdot{\ov{\Wb}}'}}_\gb\\
				&= {\braket{\ov{\Ub} + \ov{\Wb}}{{\ov{\Ub}}' + {\ov{\Wb}}'}}_\Gb
		}
	}
	showing the compatibility condition. For the kernel, one has using \cref{lemm:kergf}, $\ker\gf = \Hr^{\gammab-\varthetab}\E$. Furthermore, one has :
	\algn{
		{\braket{\Ub}{\Ub}}_\Gf = 0 &\iff {\braket{\Fb\cdot\Ub}{\Fb\cdot\Ub}}_\gf = 0\\
									&\iff \Fb\cdot\Ub \in \Hr^{\gammab-\varthetab}\E\\
									&\iff \Ub \in \Fb^{-1}\left (\Hr^{\gammab-\varthetab}\E\right )\\
									&\iff \Ub \in \Hr^{\Gammab-\Thetab}\B
	}
	where the equality $\Fb^{-1}\left (\Hr^{\gammab-\varthetab}\E\right ) = \Hr^{\Gammab-\Thetab}\B$ comes from \cref{def:Gammab} and the fact that $\Gammab-\Thetab$ is the pull-back of $\gammab-\varthetab$ by $\Fb$ (by linearity).
\end{proof}

It is important to notice that $\dim\left (\F_\B\right )$ may differ from $\dim\left (\BB\right )$, in which case \cref{th:unis_gf} does not apply and $\Gf$ is not the unique pseudo-metric compatible with $\Gb$, $\Thetab$ and $\Gammab$.

\section{Frame invariance}

\label{frame_inv_intro}

Having the material geometry fully defined, the next step is to be able to define some properties of the material, with the later goal of formulating some energies. In order to do so, a key concept in classical mechanics is the notion of frame invariance. A property of a material is then a frame invariant function of the current state of the material. In this paper the state will be the generalised first-order placement map $\Fb$. This section aims at generalising this notion of frame invariance from $\EE$ to $\E$.

\subsection{Generalised Galilean group}

In classical mechanics, frame invariance refers to the invariance of a function in the event of a change of Galilean frame. These changes of frame are given by the set of rigid transformations, i.e. bijective affine transformations that preserve the metric $\gb$. Those transformations, composed of a rotation and a translation part, form a group $\Gal$ called the Galilean group defined as:

\byalgn{
	\Gal :=& \setof{\bmat{\Rb\\\ov\tb} : \fun\EE\EE{\ov\xb}{\Rb\cdot\ov\xb + \ov\tb}}{\Rb \in \O(3), \ov\tb \in \Tr_\ov\xb\EE}
}{
	\Tr\Gal :=& \setof{\Tr\Of}{\Of \in \Gal}
}

In practice, in the classical case, the state of the system is often entirely depicted by its first order placement map $\ov \Fb = \Tr\ov\varphi$. In this case, a frame invariant function is a function of $\ov \Fb$ invariant under the action of $\Tr\Gal$ on $\ov \Fb$. The later action being the gradient of the left action of $\Gal$ on the punctual placement map $\ov\varphi$. One therefore seeks to generalise $\Tr\Gal$ into a group acting on the whole space $\Tr\E$, and therefore on $\Fb$.\\

Assume a generalised first order Galilean transformation $\Ab : \Tr\E \longto \Tr\E$ is given. Not knowing much about what such a generalisation may look like, one can still require it to be, at least, physically acceptable. That is, $\Ab \in \lpa{\Tr\E}{\Tr\E}$. Furthermore, one wants such a transformation not to alter the state of the system, adding neither defects nor misalignments. In particular, this means that $\Ab$ must transform vectors and points through a similar process. Namely, $\Ab = \Tr\ab$ for some $\ab : \E\longto \E$. Under this condition, the requirement $\Ab \in \lpa{\Tr\E}{\Tr\E}$ becomes equivalent to $\ab \in \aff\E\E$. \\

Recall the projection of interpretation defined in \cref{th:unis_gf}:
\algn{
	\interproj ~~:=&~~ \Tr\pi_\E + \varthetab^{-1}\cdot\vr_\gammab
				:& \fun{\Tr\E}{\Tr\EE}{\bmat{\bmat{\ov x\\y}\\\bmat{\delta \ov x\\\delta y}}}{\bmat{\ov x\\\delta\ov x + \delta y}}
}

The macroscopic projection $\Tr\pi_\E$ maps a vector at $x \in \E$ to the vector at $\ov x$ one would be able to measure, keeping only the macroscopic part. The projection of interpretation $\interproj$ on the other hand, maps a vector at $x \in \E$ to the vector at $\ov x$ it physically represents, its so-called interpretation. That is, it also includes the microscopic part of the original vector, scaled down by $\varthetab^{-1}$.\\

Although one would only see the macroscopic part of a vector, reality also cares about its interpretation. In fact, frame invariance should also apply to it. A generalised Galilean transformation $\Ab : \Tr\E\longto\Tr\E$ should physically represent a Galilean transformation of the total space. Therefore, it must transform interpretations in a Galilean manner. That is:

\quant{\exists \bmat{\Rb\\\ov\tb}\in\Gal,}{\interproj\circ\Ab &=& \Tr\bmat{\Rb\\\ov\tb}\circ\interproj\\
															  &=& \Tr\Rb\circ\interproj + \Tr\ov\tb}

the interpretation of the image is the interpretation of the pre-image, up to a (first-order) Galilean transformation. By definition, this is equivalent to saying:
\algn{
	\Ab \in& \fibpres[\Tr\Gal]{\bundle[\interproj]{\Tr\E}{\Tr\EE}}{\bundle[\interproj]{\Tr\E}{\Tr\EE}}
}

Notice the $\Tr\Gal$ as an index, stating that the shadow for the projection structure $\interproj$ is Galilean. Based on this analysis, the generalised first-order Galilean group is defined as follows:\\

\begin{defi}{Generalised Galilean group}{galf}
	The \underline{generalised (first-order) Galilean group} is the group, noted $\Tr\Galf$, defined as follows:
\algn{
	\Tr\Galf =& \Tr\aff\E\E \cap \fibpres[\Tr\Gal]{\bundle[\interproj]{\Tr\E}{\Tr\EE}}{\bundle[\interproj]{\Tr\E}{\Tr\EE}}\\
			 =& \setof{\Tr\ab}{\ab \in \aff\E\E,~\exists\Of\in\Tr\Gal,~\interproj\circ\Tr\ab = \Of\circ\interproj}\\
			 \subset& \lpa{\Tr\E}{\Tr\E}
}
\end{defi}

In the holonomic coordinates of $\E$, elements of $\Tr\Galf$ take a particular form, as stated in the following lemma:\\

\begin{lemm}
	\label{lemm:form_TGalf}
	Let $\Ab : \Tr\E \longto \Tr\E$. Then $\Ab \in \Tr\Galf$ if and only if there exist $\bmat{\Rb\\\ov\tb} \in \Gal$ and $\bmat{\Rb\\\tb^\vr}\in\Gal$ with the same rotational part such that, in the holonomic coordinates of $\Tr\E$, $\Ab$ takes the following form:
	\algn{
		\Ab :& \fun{\Tr\E}{\Tr\E}%
			{\bmat{	\bmat{\ov x\\ y}\\	\bmat{\delta\ov x\\\delta y}}}%
			{\bmat{	\bmat{\Rb\cdot\ov x + \ov \tb\\ \Rb\cdot y+\tb^\vr}\\	\bmat{\Rb\cdot\delta\ov x\\\Rb\cdot\delta y}}}
	}
\end{lemm}

\begin{proof}
	Assume $\Ab \in \Tr\Galf$, then $\Ab = \Tr\ab$ with $\ab \in \aff\E\E$ and $\interproj\circ\Ab=\Of\circ\interproj$ with $\Of \in \Tr\Gal$. In the holonomic coordinates of $\E$ and $\Tr\E$ one therefore has:\\
	
	\byalgn{
		\ab :& \fun{\E}{\E}%
			{\bmat{\ov x\\ y}}%
			{\bmat{\ov\ab(\ov x)\\\ab^\vr(\ov x)\cdot y+\tb^\vr(\ov x)}}\\
		\interproj :& \fun{\Tr\E}{\Tr\EE}%
			{\bmat{	\bmat{\ov x\\y}\\		\bmat{\delta \ov x\\\delta y}}}%
			{\bmat{\ov x\\\delta \ov x + \delta y}}\\
	}{
		\Of :& \fun{\Tr\EE}{\Tr\EE}%
			{\bmat{\ov x\\\delta \ov x}}%
			{\bmat{\Rb\cdot\ov x + \ov \tb\\\Rb\cdot\delta \ov x}}\\
		\Of \circ \interproj :& \fun{\Tr\E}{\Tr\EE}%
			{\bmat{	\bmat{\ov x\\y}\\		\bmat{\delta \ov x\\\delta y}}}%
			{\bmat{\Rb\cdot\ov x+\ov\tb \\ \Rb\cdot\delta \ov x + \Rb\cdot\delta y}}
	}
	\algn{
		\Ab = \Tr\ab :&\fun{\Tr\E}{\Tr\E}%
			{\bmat{	\bmat{\ov x\\y}\\		\bmat{\delta \ov x\\\delta y}}}%
			{\bmat{	\bmat{\ov\ab(\ov x)\\\ab^\vr(\ov x)\cdot y+\tb^\vr(\ov x)}\\
					\bmat{\ov{\ab}'(\ov x)\cdot\delta \ov x\\
							\left ({\ab^\vr}'(\ov x)\cdot y + {\tb^\vr}'(\ov x)\right )\cdot\delta \ov x%
								+ \ab^\vr(\ov x)\cdot \delta y}}}\\
		\interproj\circ\Ab :& \fun{\Tr\E}{\Tr\EE}%
			{\bmat{	\bmat{\ov x\\y}\\		\bmat{\delta \ov x\\\delta y}}}%
			{\bmat{	\ov\ab(\ov x)\\
					\left (\ov{\ab}'(\ov x) + {\ab^\vr}'(\ov x)\cdot y + {\tb^\vr}'(\ov x)\right )\cdot\delta \ov x%
						+ \ab^\vr(\ov x)\cdot \delta y}}\\
	}
	
	Where primes denote differentiation with respect to the unique argument. Identification of monomials then successively yields:
	\begin{enumerate}
		\it $\ab^\vr(\ov x) = \Rb$ hence ${\ab^\vr}'(\ov x) = 0$
		\it $\ov\ab(\ov x) = \Rb\cdot\ov x + \ov \tb$ hence ${\ov\ab}'(\ov x) = \Rb$
		\it ${\ov\ab}'(\ov x) + {\tb^\vr}'(\ov x) = \Rb$ hence ${\tb^\vr}'(\ov x) = 0$ (using last point)
	\end{enumerate}
	This means $\Ab$ takes the stated form. Reciprocally, if $\Ab$ has the stated form then $\interproj\circ\Ab=\Tr\bmat{\Rb\\\ov\tb}\circ\interproj$ and $\Ab \in \Tr\Galf$.
\end{proof}

This result shows that a generalised Galilean transformation is, in a way, the composition of two "coupled" classical Galilean transformation, one macroscopic and one microscopic. 

\subsection{\texorpdfstring{$\Tr\Galf$}{The generalised Galilean group} as a stabiliser}

In the classical case, the Galilean group $\Gal$ can be described as the group of affine transformations preserving the Euclidean metric $\gb$. That is, a stabiliser of $\gb$ for the pull-back action. One may therefore wonder how the generalised Galilean group $\Tr\Galf$ compares to the set of physically acceptable transformations preserving the pseudo-metric.\\

Let $\Ab \in \lpa{\Tr\E}{\Tr\E}$ be over $\ab : \E \longto \E$. Preserving $\gf$ implies, in particular, preserving $\gf_\vr^\vr$ and $\gf_\hr^\vr$. Using \cref{lemm:Gf} one therefore respectively has:

\algn{
	\Ab^*\gf = \gf
		&\implies& && 					{\Ab_\vr^\vr}^\star \cdot \gf_\vr^\vr \cdot \Ab_\vr^\vr &= \gf_\vr^\vr\\
		&\implies& &\forall x \in \E,& 	\at{\Ab_\vr^\vr}_x^\star \cdot {\varthetab_{\ab(x)}^{-1}}^\star \cdot \gb_{\ov \ab(\ov x)} \cdot \varthetab_{\ab(x)}^{-1} \cdot \at{\Ab_\vr^\vr}_x &= {\varthetab_{x}^{-1}}^\star \cdot \gb_{\ov x} \cdot \varthetab_{x}^{-1}\\
		&\implies& &\forall x \in \E,& 	\varthetab_{x}^\star \cdot \at{\Ab_\vr^\vr}_x^\star \cdot {\varthetab_{\ab(x)}^{-1}}^\star \cdot \gb_{\ov \ab(\ov x)} \cdot \varthetab_{\ab(x)}^{-1} \cdot \at{\Ab_\vr^\vr}_x \cdot \varthetab_{x} &= \gb_{\ov x}\\
		&\implies& &\forall x \in \E,& \varthetab_{\ab(x)}^{-1} \cdot \at{\Ab_\vr^\vr}_x \cdot \varthetab_{x} &=: \Ob_x \in \O\left (\gb_{\ov x},~ \gb_{\ov \ab(\ov x)}\right )\\
\\
	\Ab^*\gf = \gf
		&\implies& && {\Ab}^* \gf_\hr^\vr &= \gf_\hr^\vr\\
		&\implies& && 	{\Ab_\vr^\vr}^\star \cdot \gf_\hr^\vr \cdot {\Ab_\hr^\hr}
			+ {\Ab_\vr^\vr}^\star \cdot \gf_\vr^\vr \cdot {\Ab_\hr^\vr} &= \gf_\hr^\vr\\
		&\implies& &\forall x \in \E,&
			\at{\Ab_\vr^\vr}_x^\star \cdot {\varthetab_{\ab(x)}^{-1}}^\star \cdot \gb_{\ov \ab(\ov x)} \cdot \Tr_{\ab(x)}\pi_\E \cdot \at{\Ab_\hr^\hr}_x\qquad&\\
			&&&&+~{\at{\Ab_\vr^\vr}}_x^\star \cdot {\varthetab_{\ab(x)}^{-1}}^\star \cdot \gb_{\ov \ab (\ov x)} \cdot \varthetab_{\ab(x)}^{-1}\cdot\at{\Ab_\hr^\vr}_x &= {\varthetab_{x}^{-1}}^\star \cdot \gb_{\ov x} \cdot \Tr_x\pi_\E\\
		&\implies& &\forall x \in \E,& 
		{\varthetab_{x}}^\star\cdot \at{\Ab_\vr^\vr}_x^\star \cdot {\varthetab_{\ab(x)}^{-1}}^\star \cdot \gb_{\ov \ab(\ov x)} \cdot \Tr_{\ab(x)}\pi_\E \cdot \at{\Ab_\hr^\hr}_x\qquad&\\
			&&&&+ {\varthetab_{x}}^\star \cdot {\at{\Ab_\vr^\vr}}_x^\star \cdot {\varthetab_{\ab(x)}^{-1}}^\star \cdot \gb_{\ov \ab (\ov x)} \cdot \varthetab_{\ab(x)}^{-1}\cdot\at{\Ab_\hr^\vr}_x &= \gb_{\ov x} \cdot \Tr_x\pi_\E\\
		&\implies& &\forall x \in \E,& 	\Ob_x^\star \cdot \gb_{\ov \ab(\ov x)} \cdot \Tr_{\ab(x)}\pi_\E \cdot \at{\Ab_\hr^\hr}_x
			+ \Ob_x^\star \cdot \gb_{\ov \ab (\ov x)} \cdot \varthetab_{\ab(x)}^{-1}\cdot\at{\Ab_\hr^\vr}_x &= \gb_{\ov x} \cdot \Tr_x\pi_\E\\
		&\implies& &\forall x \in \E,& 	\gb_{\ov x} \cdot \Ob_x^{-1} \cdot \Tr_{\ab(x)}\pi_\E \cdot \at{\Ab_\hr^\hr}_x
			+ \gb_{\ov x} \cdot \Ob_x^{-1} \cdot \varthetab_{\ab(x)}^{-1}\cdot\at{\Ab_\hr^\vr}_x &= \gb_{\ov x} \cdot \Tr_x\pi_\E\\
		&\implies& &\forall x \in \E,& 	\Ob_x^{-1} \cdot \Tr_{\ab(x)}\pi_\E \cdot \at{\Ab_\hr^\hr}_x
			+ \Ob_x^{-1} \cdot \varthetab_{\ab(x)}^{-1}\cdot\at{\Ab_\hr^\vr}_x &= \Tr_x\pi_\E\\
		&\implies& &\forall x \in \E,& 	\Ob_x^{-1} \cdot \Tr_{\ab(x)}\pi_\E \cdot \at{\Ab_\hr^\hr}_x \cdot \gammab_x
			+ \Ob_x^{-1} \cdot \varthetab_{\ab(x)}^{-1} \cdot \at{\Ab_\hr^\vr}_x \cdot \gammab_x &= \Id_x\\
		&\implies& &\forall x \in \E,& 	\Tr_{\ab(x)}\pi_\E \cdot \at{\Ab_\hr^\hr}_x \cdot \gammab_x
			+ \varthetab_{\ab(x)}^{-1} \cdot \at{\Ab_\hr^\vr}_x \cdot \gammab_x &= \Ob_x\\
		&\implies& &\forall x \in \E,& 	\ov\Ab_{\ov x}
			+ \varthetab_{\ab(x)}^{-1} \cdot \at{\Ab_\hr^\vr}_x \cdot \gammab_x &= \Ob_x\\
}

with the notation
\[
	\O\left (\gb_{\ov x},~ \gb_{\ov \ab(\ov x)}\right ) = \setof{\ov\Ab_\ov x : \Tr_\ov x\EE \longto \Tr_{\ov \ab(\ov x)}\EE}{{\ov\Ab_\ov x}^\star\cdot\gb_{\ov \ab(\ov x)}\cdot \ov\Ab_\ov x = \gb_x}
\]
which, using holonomic coordinates, can be identified with the group $\O(3)$ of orthogonal matrices. One then obtains two equalities that must be satisfied by any pseudo-metric-preserving map $\Ab$. By substituting those equalities back into \cref{lemm:Gf} and simplifying, one cycles back to $\Ab^*\gf=\gf$. This means that these two equations are also sufficient. One therefore has the following equivalence:\\

\begin{lemm}
\label{lem:pres_gf}
Let $\Ab \in \lpa{\Tr\E}{\Tr\E}$, then one has:
\algn{
	\Ab^*\gf = \gf &\iff& &\forall x \in \E,%
		&\\
		&&&&\ov\Ab_{\ov x} + \varthetab_{\ab(x)}^{-1} \cdot \at{\Ab_\hr^\vr}_x \cdot \gammab_x = \varthetab_{\ab(x)}^{-1} \cdot \at{\Ab_\vr^\vr}_x \cdot \varthetab_{x} \in \O\left (\gb_{\ov x},~ \gb_{\ov \ab(\ov x)}\right )&\\
		&\iff& &\forall x \in \E,~\exists \Ob_x \in \O\left (\gb_{\ov x},~ \gb_{\ov \ab(\ov x)}\right ),%
		\hspace{-61pt}&\\
			&&&&\Ab_x \equiv \bmat{%
				\gammab_{\ab(x)}\cdot\ov\Ab_\ov{x}\cdot\Tr_x\pi_\E & \bold 0\\%
				\vartheta_{\ab(x)}\cdot\left (\Ob_x - \ov\Ab_\ov{x}\right )\cdot\Tr_x\pi_\E & \vartheta_{\ab(x)}\cdot\Ob_x\cdot\varthetab_x^{-1}%
			}
}
\end{lemm}

In particular, this lemma shows that, while preserving $\gb$ is equivalent to being Galilean in the classical case, preserving $\gf$ is not enough in the generalised case. The missing step is the following lemma:

\begin{lemm}
\label{lem:pres_gamma}
Let $\Ab \in \lpa{\Tr\E}{\Tr\E}$ be injective over an injective $\ab : \E \longto \E$, then:
\algn{
	\Ab^*\gammab = \gammab &\iff& \Ab_\hr^\vr = \bold0
}
\end{lemm}

\begin{proof}
\algn{
	\Ab^*\gammab = \gammab
			&\iff& &\forall x \in \E,& \left (\Ab^*\gammab\right )_x &= \gammab_x\\
			&\iff& &\forall x \in \E,& \left (\Ab^*\gammab\right )_x\cdot\Tr\pi_\E &= \gammab_x\cdot\Tr\pi_\E\\
			&\iff& &\forall x \in \E,& {\Ab_x}^{-1}\cdot\gammab_{\ab(x)}\cdot\ov{\Ab}_{\ov x}\cdot\Tr\pi_\E &= \hr_\gammab\\
			&\iff& &\forall x \in \E,& {\Ab_x}^{-1}\cdot \gammab_{\ab(x)}\cdot\Tr\pi_\E \cdot\Ab_x  &= \hr_\gammab\\
			&\iff& &\forall x \in \E,& \gammab_{\ab(x)}\cdot\Tr\pi_\E\cdot\Ab_x &= \Ab_x\cdot\hr_\gammab\\
			&\iff& &\forall x \in \E,& \hr_\gammab\cdot\Ab_x &= \Ab_x\cdot\hr_\gammab\\
			&\iff& &&				\Ab^\hr &= \Ab_\hr\\
			&\iff& &&				\Ab^\hr_\hr &= \Ab_\hr^\hr + \Ab_\hr^\vr\\
			&\iff& &&				\bold0 &= \Ab_\hr^\vr
}

where $\Ab_\vr^\hr=0$ has been used, since $\Ab\in\lpa{\Tr\E}{\Tr\E}$.
\end{proof}

By using this lemma with $\Ab$ and $\Tr\ab$, one eventually has enough equations to retrieve $\Tr\Galf$. The result, contained in the following lemma, states that the group $\Tr\Galf$ is composed of those physically acceptable maps $\Ab$ which are not only a stabiliser for $\gf$ and $\gammab$ but whose shadow's tangent $\Tr\ab$ is a stabiliser of $\gammab$:\\

\begin{theo}{$\Tr\Galf$ as the intersection of stabilisers}{galf_as_invar}
	\label{lemm:galf_as_invar}
	The first-order Galilean group can be given by the following equations:
	\algn{
		\Tr\Galf &= \setof{\Ab \in \lpa{\Tr\E}{\Tr\E}}{\mat{%
		  				\Ab~\text{inversible},~~\ab~\text{inversible},\\
				  		\Ab^*\gf=\gf,~~\Ab^*\gammab = \gammab~~\text{and}~~\Tr\ab^*\gammab=\gammab
		  		}}
	}
	where $\ab : \E \longto \E$ implicitly refers to the punctual shadow of $\Ab$ in $\l{\Tr\E}{\Tr\E}$.
\end{theo}

\begin{proof}
	First, let $x \in \E$ be a fixed value. From \cref{lem:pres_gf} one has that, if $\Ab^*\gf = \gf$, there exists some $\Ob_x \in \O\left (\gb_{\ov x},~ \gb_{\ov \ab(\ov x)}\right )$ such that:
	\algn{
		\Ab_x &\equiv \bmat{\gammab_{\ab(x)}\cdot\ov\Ab_\ov{x}\cdot\Tr_x\pi_\E & \bold 0\\%
				\vartheta_{\ab(x)}\cdot\left (\Ob_x - \ov\Ab_\ov{x}\right )\cdot\Tr_x\pi_\E & \vartheta_{\ab(x)^{-1}}\cdot\Ob_x\cdot\varthetab_x}
	}
	Using \cref{lem:pres_gamma}, one then sees that $\Ab^*\gammab=\gammab$ further implies $\vartheta_{\ab(x)}\cdot\left (\Ob_x - \ov\Ab_\ov{x}\right )\cdot\Tr_x\pi_\E = \bold0$. That is, since $\Tr\pi_\E$ is surjective and $\varthetab$ is injective, $\Ob_x = \ov\Ab_\ov x$. In particular $\Ob_x$ only depends (smoothly) on $\ov x$.\\
	
	Furthermore, by definition, $\Ab \in \lpa{\Tr\E}{\Tr\E}$ implies $\ab \in \aff\E\E$ and $\Ab_\vr^\vr = \left (\Tr\ab\right )_\vr^\vr$. Where the later is the gradient of $\ab \in \aff\E\E$ in the (punctual) vertical coordinate. One therefore sees that those two conditions imply that, in the holonomic coordinates, $\Ob$ corresponds to both the macroscopic gradient $\Tr\ov\ab$ and the microscopic gradient $\left (\Tr\ab\right )_\vr^\vr$. That is, $\Ab$ takes the form:
	
	\algn{
		\Ab : \fun{\Tr\E}{\Tr\E}{\bmat{\bmat{\ov x\\ y} \\ \bmat{\delta \ov x\\\delta y}}}%
			{\bmat{\bmat{\ov{\ab}\left (\ov x\right ) \\ \Ob_\ov x \cdot y + \tb^\vr(\ov x)} \\%
				  \bmat{\Ob_\ov x \cdot \delta \ov x \\ \Ob_\ov x \cdot \delta y}}}
	}
	
	with ${\ov\ab}'(\ov x) = \Ob_\ov x$, where prime denote derivation with respect to the sole argument. By integrating, one has for $\left (\ov{x}_1, \ov{x}_2\right ) \in \EE^2$
\algn{
	\dr\left (\ov{x}_1, \ov{x}_2\right )
			&= \int_{\ab\left (\ov{x}_1\right )}^{\ab\left (\ov{x}_2\right )} \left \|\dr\ov x\right \|_\gb \dr\ov x\\
			&= \int_{\ov{x}_1}^{\ov{x}_2} \left \|\Ob_\ov x \cdot \dr\ov x\right \|_\gb \dr\ov x\\
			&= \int_{\ov{x}_1}^{\ov{x}_2} \left \|\dr\ov x\right \|_\gb \dr\ov x\\
			&= \dr\left (\ov{x}_1, \ov{x}_2\right )	
}
where $\dr$ is the euclidean distance. This means that $\ov\ab$ is an isometry of $\EE$ and, using standard results of the Euclidean geometry, $\ov\ab(\ov x) = \Rb\cdot \ov x + \ov\tb$ or, in other words, $\Ob_\ov x = \Rb$ constant.  Using \cref{lem:pres_gamma} with $\left (\Tr\ab\right )^*\gammab = \gammab$ one has $\left (\Tr\ab\right )_\hr^\vr = \bold0$. That is, $\ider{\ov x}{\Rb \cdot y + \tb^\vr(\ov x)} = 0$. This means $\ider{\ov x}{\tb^\vr(\ov x)} = 0$. $\Ab$ therefore takes the following form:
	\algn{
		\Ab : \fun{\Tr\E}{\Tr\E}{\bmat{\bmat{\ov x\\ y} \\ \bmat{\delta \ov x\\\delta y}}}%
			{\bmat{\bmat{\Rb\cdot\ov x + \ov\tb \\ \Rb \cdot y + \tb^\vr} \\%
				  \bmat{\Rb \cdot \delta \ov x \\ \Rb \cdot \delta y}}}
	}
According to \cref{lemm:form_TGalf}, this implies $\Ab \in \Tr\Galf$. Reciprocally, the form stated in \cref{lemm:form_TGalf} describes a fully invertible element of $\lpa{\Tr\E}{\Tr\E}$ which always verifies $\Ab^*\gammab = \left (\Tr\ab\right )^*\gammab = \gammab$ and $\Ab^*\gf=\gf$ thus proving the equality.
\end{proof}

\subsection{Orbital invariants of \texorpdfstring{$\Tr\Galf$}{the generalised Galilean group}}

Let $\bundle{\B}{\BB}$ be an affine bundle. The set $\Tr\Galf$ of generalised Galilean transformations is a group for the composition, acting on $\conf{\Tr\B}{\Tr\E}\subset\lpa{\Tr\B}{\Tr\E}$ by composition:
\algn{
	\fun{\Tr\Galf \times \conf{\Tr\B}{\Tr\E}}{\conf{\Tr\B}{\Tr\E}}{(\Of, \Fb)}{\Of\circ\Fb}
}

The aim of this section is to characterise the orbits\footnote{Let $\Fb \in \conf{\Tr\B}{\Tr\E}$. Then the orbit of $\Fb$ under the action of $\Tr\Galf$ is $\opp[\Tr\Galf]{Orb}\Fb = \setof{\Of\circ\Fb}{\Of \in \Tr\Galf}$} of this action $\Tr\Galf \acts \conf{\Tr\B}{\Tr\E}$ using a complete set of invariants. That is, functionals of $\Fb \in \conf{\Tr\B}{\Tr\E}$ which are constant on the orbits.\\

One could characterise the orbits of $\Fb$ using $\Fb^*\gf$, $\Fb^*\gammab$ and $\Tr\varphi^*\gammab$, as suggested in \cref{lemm:galf_as_invar}, however they are quite inhomogeneous in size. Indeed, while both connections are, roughly speaking, $n\times3$ matrices \partxt{with $n=\dim\left (\BB\right )$}, the pseudo-metric is significantly larger with $(n+3)\times(n+3)$ entries. Before continuing, splitting $\Gf$ into smaller elements may improve the clarity. This is done using the following equivalence of data:\\

\begin{theo}{Data equivalence}{data_equiv}
	Let $\Fb \in \lconf{\Tr\B}{\Tr\E}$ be unknown. Then, the data of $\Gf = \Fb^\star\cdot\gf\cdot\Fb$ is equivalent to the joint data of:
	\begin{enumerate}
		\it the micro-metric $\Gf_\vr^\vr : \Vr\B \longto \Vr^\star\B$
		\it the connection $\Gammab - \Thetab : \B\times_\BB\Tr\BB\longto \Tr\B$
	\end{enumerate}

	Furthermore:
	\[
		\Gb = \Thetab^\star \cdot \Gf_\vr^\vr \cdot \Thetab
	\]
	In particular, given that $\Thetab$ is known and $\dim\left (\F_\B\right ) = \dim\left (\BB\right )$, the data of $\Gf_\vr^\vr$ is equivalent to the data of $\Gb$.
\end{theo}

\begin{proof}
	First, $\Gf$ implies $\Gf_\vr^\vr$ by standard restriction from $\Tr\B$ to $\Vr\B$ and $\Gammab-\Thetab$ through $\ker\Gf = \Hr^{\Gammab-\Thetab}$. For the reverse implication one has:
	\byalgn{
		\vr_{\Gammab-\Thetab} &= \Id - \left (\Gammab-\Thetab\right )\cdot\Tr\pi_\B\\
							  &= \Id - \left (\Fb^{-1}\cdot\gammab\cdot\ov{\Fb}\cdot\Tr\pi_\B-\Fb^{-1}\cdot\varthetab\cdot\ov{\Fb}\cdot\Tr\pi_\B\right )\\
							  &= \Id - \left (\Fb^{-1}\cdot\gammab\cdot\Tr\pi_\E\cdot\Fb-\Fb^{-1}\cdot\varthetab\cdot\Tr\pi_\E\cdot\Fb\right )\\
							  &= \Id - \Fb^{-1}\cdot\left (\gammab-\varthetab\right )\cdot\Tr\pi_\E\cdot\Fb\\
							  &= \Fb^{-1}\cdot\vr_{\gammab-\varthetab}\cdot\Fb\\
	}{
		\vr_{\gammab-\varthetab} &= \Id - \left (\gammab - \varthetab\right )\cdot\Tr\pi_\E\\
								 &= \Id - \hr_\gammab + \varthetab\cdot\Tr\pi_\E\\
								 &= \vr_\gammab + \varthetab\cdot\Tr\pi_\E\\
		\varthetab^{-1}\cdot \vr_{\gammab-\varthetab}
								 &= \varthetab^{-1}\cdot\vr_\gammab + \Tr\pi_\E\\
								 &= \iota
	}

	\byalgn{
			\varthetab^{-1}\cdot\vr_\gammab\cdot\Fb \cdot \vr_{\Gammab-\Thetab}%
				\hspace{-34pt}&\\
				&= \varthetab^{-1}\cdot\vr_\gammab\cdot\Fb \cdot \Fb^{-1}\cdot\vr_{\gammab-\varthetab}\cdot\Fb\\
				&= \varthetab^{-1}\cdot\vr_\gammab\cdot\vr_{\gammab-\varthetab}\cdot\Fb\\
				&= \varthetab^{-1}\cdot\vr_{\gammab-\varthetab}\cdot\Fb\\
				&= \interproj\cdot\Fb
		}{
		\vr_{\Gammab-\Thetab}^\star\cdot\Gf_\vr^\vr\cdot\vr_{\Gammab-\Thetab}%
				\hspace{-35pt}&\\
					&= \left [\cdots\right ]^\star \cdot \gb \cdot \left [\varthetab^{-1}\cdot\vr_\gammab\cdot\Fb\cdot \vr_{\Gammab-\Thetab}\right ]\\
					&= \Fb^\star\cdot\interproj^\star \cdot \gb \cdot \interproj\cdot\Fb\\
					&= \Fb^\star \cdot \gf \cdot \Fb\\
					&= \Gf
	}
	And one sees that the term $\vr_{\Gammab-\Thetab}^\star\cdot\Gf_\vr^\vr\cdot\vr_{\Gammab-\Thetab}$ only involves $\Gf_\vr^\vr$, $\Gammab-\Thetab$ and $\Tr\pi_\B$.\\
	
	\hrulefill\\	

	For the second part, one has:
		\algn{
			\Thetab^\star\cdot\Gf_\vr^\vr \cdot\Thetab
					&= \left [\ov{\Fb}^\star\cdot\varthetab^\star\cdot{{\Fb_\vr^\vr}^{-1}}^\star\right ]\cdot\left [{\Fb_\vr^\vr}^\star\cdot\gf_\vr^\vr\cdot\Fb_\vr^\vr\right ]\cdot\left [{\Fb_\vr^\vr}^{-1}\cdot\varthetab\cdot\ov{\Fb}\right ]\\
					&= \ov{\Fb}^\star\cdot\varthetab^\star\cdot\gf_\vr^\vr\cdot\varthetab\cdot\ov{\Fb}\\
					&= \left [\ov{\Fb}^\star\cdot\varthetab^\star\right ]\cdot\left [\vr_\gammab^\star\cdot{\varthetab^{-1}}^\star\right ]\cdot\gb\cdot\left [\varthetab^{-1}\cdot\vr_\gammab\right ]\cdot\left [\varthetab\cdot\ov{\Fb}\right ]\\
					&= \ov{\Fb}^\star\cdot\gb\cdot\ov{\Fb}\\
					&= \Gb
		}
	Meaning that $\Thetab \in \operatorname{Iso}\left (\Gb, \Gf_\vr^\vr\right )$ is an isometry. And allows to acquire $\Gb$ from $\Gf_\vr^\vr$, and vis-versa if $\Thetab$ is invertible \partxt{which is the case iff $\dim\left (\F_\B\right ) = \dim\left (\BB\right )$}.
\end{proof}

Let $\Fb \in \conf{\Tr\B}{\Tr\E}$ and $\Fb' \in \conf{\Tr\B}{\Tr\E}$ be two, possibly distinct, placement maps (beware that primes are \textbf{not} derivatives). Let $\widehat{\Fb} := \Fb'\cdot\Fb^{-1} : \Imr\left (\Fb\right ) \longto \Tr\E$ and let $\varphi$, $\varphi'$ and $\widehat{\varphi}$ be their respective punctual shadows. Then, $\widehat{\Fb} \in \lpa{\Tr\E}{\Tr\E}$. Furthermore, one has:

\algn{
	\Orb{\Fb} = \Orb{\Fb'} &\iff& &\exists \Ab \in \Tr\Galf,& ~~~\Fb' &= \Ab\cdot\Fb\\
						   &\iff& &\exists \Ab \in \Tr\Galf,& \Fb'\cdot\Fb^{-1} &= \Ab_{\Imr\left (\Fb\right )}\\
						   &\iff& && \widehat{\Fb} &\in \at{\Tr\Galf}_{\Imr\left (\Fb\right )}
}

\renewcommand{\tmp}{\mathbf{\kappa}}

where $\at{\Tr\Galf}_\S := \setof{\at{\Ab}_\S}{\Ab \in \Tr\Galf}$. Using the fact that the pull-back is a group action $-$ that is, $\left (\Ab_1\circ\Ab_2\right )^*\kappa = \Ab_1^*\left (\Ab_2^*\kappa\right )$ $-$ one obtains that pulling back by $\widehat\Fb$ preserves a tensor if and only if pulling back by $\Fb$ and $\Fb'$ gives the same new tensor.\\

These two observations, along with earlier results, are the ingredients allowing us to prove the main result of this paper:\\

\begin{theo}{Orbital invariants of $\Tr\Galf \acts \conf{\Tr\B}{\Tr\E}$}{inv_orb}
	Let $\bundle{\B}{\BB}$ be an affine bundle, $\Tr\Galf$ be the generalised first-order Galilean group on $\E$ acting on the left of $\conf{\Tr\B}{\Tr\E}$ and $\S_\vr^{++}\left (\B\right )$ be the set of micro-metrics\footnote{The data of $\Mf \in \S_\vr^{++}\left (\B\right )$ is the smooth data of a metric $\displaystyle\Mf_{\ov X} \in \S^{++}\left (\B_\ov X\right )$ on $\displaystyle\B_\ov X$ for each $\ov X \in \BB$. In particular, locally, $\S_\vr^{++}\left (\B\right )$ looks like $\smooth{\BB}{\S^{++}\left (\F_\B\right )}$.} on $\B$.\\
	
	Let $\II$ be the following application:\footnote{Recall the notation $\Fb^*\gf_\vr^\vr = {\Fb_\vr^\vr}^*\gf_\vr^\vr = {\Fb_\vr^\vr}^\star\cdot\gf_\vr^\vr\cdot\Fb_\vr^\vr = \Gf_\vr^\vr$.}
	\algn{
		\II &: \fun%
		{\conf{\Tr\B}{\Tr\E}}%
		{\S^{++}_\vr\left (\B\right )\times\opp{Con}{\B}\times\opp{Con}{\B}\times\opp{Sold}{\B}}%
		{\Fb}%
		{\raisebox{-0.75em}{$\begin{matrix*}[r]%
			\II\left (\Fb\right )%
				&=& \left (\vphantom{\Fb^*}\right .\hspace{-1em}&\Fb^*\gf_\vr^\vr,&~\Fb^*\varthetab,&~\Fb^*\gammab,&\left .~\Tr\varphi^*\gammab\right )\\%
				&=& \left (\vphantom{\Fb^*}\right .\hspace{-1em}&\Tr\varphi^*\gf_\vr^\vr,&~\Tr\varphi^*\varthetab,&~\Fb^*\gammab,&\left .~\Tr\varphi^*\gammab\right )%
			\end{matrix*}$}%
		}
	}
	The application $\II$ is \underline{constant on the orbits} of the action $\Tr\Galf \acts \conf{\Tr\B}{\Tr\E}$ and provides, using the fundamental theorem on homomorphisms, the following bijection:
	\algn{
		\faktor{\conf{\Tr\B}{\Tr\E}}{\Tr\Galf} \quad &\xhooktwoheadrightarrow{\quad\II\quad}& \Imr\left (\II\right )
	}
	
	In other words, \underline{$\II$ is complete}, meaning it characterises the orbits:
	\algn{
		&\forall \left (\Fb_1, \Fb_2\right ) \in  \conf{\Tr\B}{\Tr\E}^2,& &\Orb{\Fb_1}=\Orb{\Fb_2}& &\iff& &\II\left (\Fb_1\right ) = \II\left (\Fb_2\right )&
	}
	
	\hrulefill\\
	
	The four elements $\Fb^*\gf_\vr^\vr$, $\Fb^*\gammab$, $\Tr\varphi^*\gammab$ and $\Fb^*\varthetab$ will be called the \underline{principal tensorial invariants}. $\II$ will be called principal tensorial invariants quadruplet or, by metonymy, simply principal tensorial invariants.
\end{theo}

\begin{proof}
First, one has:
	\byalgn{
		\Fb^*\gf_\vr^\vr &= {\Fb_\vr^\vr}^*\gf_\vr^\vr\\
						 &= {\left (\Tr\varphi\right )_\vr^\vr}^*\gf_\vr^\vr\\
						 &= \Tr\varphi^*\gf_\vr^\vr
	}{
		\Fb^*\varthetab  &= {\Fb_\vr^\vr}^*\varthetab\\
						 &= {\left (\Tr\varphi\right )_\vr^\vr}^*\varthetab\\
						 &= \Tr\varphi^*\varthetab
	}

Secondly, the first result of the theorem is implied by the last one through the fundamental theorem on homomorphisms. Therefore, only the later needs to be proven. Let $\Fb$ and $\Fb'$ be two elements of $\conf{\Tr\B}{\Tr\E}$ and $\widehat\Fb = \Fb'\cdot\Fb^{-1} : \Imr\left (\Fb\right )\longto \Imr\left (\Fb'\right )$. Let $\tmp : \ifun{\E\times_\EE\Tr\EE}{\Tr\E}{(x,\ub)}{\tmp_x\cdot\ub}$ be smooth, then:\\

\algn{
	\widehat{\Fb}^*\tmp = \at{\tmp}_{\Imr\left (\Fb\right )}%
									 &\iff& &\forall x \in \Imr\left (\Fb\right ),& 	\left (\widehat{\Fb}_x\right )^{-1}\cdot\tmp_{\widehat{\varphi}\left (x\right )}\cdot\ov{\widehat{\Fb}}_{\ov x} &= \tmp_{x}\\
									 &\iff& &\forall X \in \B,& 	\left ({\widehat{\Fb}_{\varphi(X)}}\right )^{-1}\cdot\tmp_{\widehat{\varphi}\left (\varphi(X)\right )}\cdot\ov{\widehat{\Fb}}_{\ov{\varphi}(\ov X)} &= \tmp_{\varphi(X)}\\
							 		 &\iff& &\forall X \in \B,&		\Fb_X\cdot{\Fb'_X}^{-1}\cdot\tmp_{\varphi'(X)}\cdot{\ov{\Fb}}_{\ov X}'\cdot{\ov{\Fb}_{\ov X}}^{-1} &= \tmp_{\varphi(X)}\\
									 &\iff& &\forall X \in \B,&		{\Fb'_X}^{-1}\cdot\tmp_{\varphi'(X)}\cdot{\ov{\Fb}}_{\ov X}' &= {\Fb_X}^{-1}\cdot\tmp_{\varphi(X)}\cdot\ov{\Fb}_{\ov X}\\
									 &\iff& &\forall X \in \B,&		{\Fb'}^*\tmp &= \Fb^*\tmp
}

By setting $\tmp = \gammab$ then $\tmp = \varthetab$ and by noticing that the formula remains true in the particular $\Fb=\Tr\varphi$ and $\Fb'=\Tr\varphi'$ cases \partxt{and accordingly $\widehat{\Fb}=\Tr\widehat{\varphi}$}, one gets:

\byalgn{
	\widehat{\Fb}^*\gammab &= \gammab &\iff& &\Fb^*\gammab &= {\Fb'}^*\gammab\\
	\widehat{\Fb}^*\varthetab &= \varthetab &\iff& &\Fb^*\varthetab &= {\Fb'}^*\varthetab
}{
	\Tr\widehat{\varphi}^*\gammab &= \gammab &\iff& &{\Tr\varphi}^*\gammab &= {\Tr\varphi'}^*\gammab\\
	\Tr\widehat{\varphi}^*\varthetab &= \varthetab &\iff& &{\Tr\varphi}^*\varthetab &= {\Tr\varphi'}^*\varthetab
}

Furthermore, the pull-back operation is also a group action on metrics:

\algn{
	\widehat{\Fb}^*\gf_\vr^\vr = \gf_\vr^\vr%
									 &\iff& &\forall x \in \E,\qquad& \at{\widehat{\Fb}_\vr^\vr}_{x}^\star\cdot\at{\gb_\vr^\vr}_{\widehat{\varphi}\left (x\right )}\cdot\at{\widehat{\Fb}_\vr^\vr}_{x} &= \at{\gb_\vr^\vr}_x\\
									 &\iff& &\forall X \in \BB,& 	\at{\widehat{\Fb}_\vr^\vr}_{\varphi(X)}^\star\cdot\at{\gb_\vr^\vr}_{\widehat{\varphi}\left (\varphi(X)\right )}\cdot\at{\widehat{\Fb}_\vr^\vr}_{\varphi(X)} &= \at{\gb_\vr^\vr}_{\varphi(X)}\\
									 &\iff& &\forall X \in \BB,& 	\left ({\Fb_\vr^\vr}^{-1}\right )^\star\cdot{{\Fb'}_\vr^\vr}^\star\cdot\gb_\vr^\vr\cdot{\Fb'}_\vr^\vr\cdot{\Fb_\vr^\vr}^{-1} &= \at{\gb_\vr^\vr}_{\varphi(X)}\\
									 &\iff& &\forall X \in \BB,& 	\at{{\Fb'}_\vr^\vr}_{\varphi(X)}^\star\cdot\at{\gb_\vr^\vr}_{\varphi'\left (X\right )}\cdot\at{{\Fb'}_\vr^\vr}_{\varphi(X)} &= \at{\Fb_\vr^\vr}_{\varphi(X)}^\star\cdot\at{\gb_\vr^\vr}_{\varphi(X)}\cdot\at{\Fb_\vr^\vr}_{\varphi(X)}\\
									 &\iff& &\forall X \in \BB,&		{{\Fb'}_\vr^\vr}^*\gf_\vr^\vr &= {\Fb_\vr^\vr}^*\gf_\vr^\vr
}

Noticing that $\widehat{\Fb} \in \lpa{\Tr\E}{\Tr\E}$ and that $\widehat{\Fb}$ and $\widehat\varphi$ are invertible, one obtains the following chain of equivalences:

\algn{
						&&		\Orb{\Fb} &= \Orb{\Fb'}\\
						&\iff&  \widehat{\Fb} &\in \at{\Tr\Galf}_{\Imr\left (\Fb\right )}\\
						&\iff&	\left (\widehat{\Fb}^*\gf,~ \widehat{\Fb}^*\gammab,~ \Tr\widehat{\varphi}^*\gammab\right ) &=%
						  		\left (\gf,~ \gammab,~ \gammab\right )%
						   			&\text{\partxt{\cref{lemm:galf_as_invar}}}\\
						&\iff& 	\left (\left (\widehat{\Fb}^*\gf_\vr^\vr,~ \widehat{\Fb}^*\gammab - \widehat{\Fb}^*\varthetab\right ),~ \widehat{\Fb}^*\gammab,~ \Tr\widehat{\varphi}^*\gammab\right ) &=%
						   		\left (\left (\gf_\vr^\vr,~ \gammab-\varthetab\right ),~ \gammab,~ \gammab\right )%
						   			&\text{\partxt{\cref{th:data_equiv}}}\\
						&\iff& 	\left (\widehat{\Fb}^*\gf_\vr^\vr,~ \widehat{\Fb}^*\varthetab,~ \widehat{\Fb}^*\gammab,~ \Tr\widehat{\varphi}^*\gammab\right ) &=%
						   		\left (\gf_\vr^\vr,~ \varthetab,~ \gammab,~ \gammab\right )\\
						%
%						&\iff& 	\left (\ov{\widehat{\Fb}}^*\gb,~ \widehat{\Fb}^*\varthetab,~ \widehat{\Fb}^*\gammab,~ \Tr\widehat{\varphi}^*\gammab\right ) &=%
%						   		\left (\gb,~ \varthetab,~ \gammab,~ \gammab\right )%
%						   			\qquad &\partxt{using part $2$ of \cref{th:data_equiv}}\\
						%
						&\iff& 	\left (\Fb^*\gf_\vr^\vr,~ \Fb^*\varthetab,~ \Fb^*\gammab,~ \Tr\varphi^*\gammab\right ) &=%
						   		\left ({\Fb'}^*\gf_\vr^\vr,~ {\Fb'}^*\varthetab,~ {\Fb'}^*\gammab,~ \Tr{\varphi'}^*\gammab\right )\\
						&\iff&	\II\left (\Fb\right ) &= \II\left (\Fb'\right )
}

\end{proof}

Technically, the term "tensorial" is abusive as $\Fb^*\gammab$ and $\Tr\varphi^*\gammab$ are connections and hence not tensors. However, the difference of two connections is a tensor. One can therefore replace $\II$ by $\II - \II_\mathrm{ref}$, where $\II_\mathrm{ref}$ is the value of the invariants in a reference placement map $\Fb_\mathrm{ref}$, and obtain tensors. This means that, in practice, a reference placement map is required. However, in order to keep the discussion clear, this reference placement map will be omitted and $\II$ used instead of $\II-\II_\mathrm{ref}$. For lack of a better word, the term "tensorial" shall be used, albeit abusively.\\

\begin{rema}[unbreakable]
	From the result of \cref{th:data_equiv}, one has that the data of the Cauchy-Green $\Gb : \Tr\BB \longto \Tr^\star\BB$ is contained in the data of $\Gf_\vr^\vr : \Vr\B \longto \Vr^\star\B$ \partxt{and hence also in $\Gf$}. It is therefore obtainable from the principal invariants of \cref{th:inv_orb}.
\end{rema}

\section{Discussions}
\label{sect:disc}
\subsection{On the values taken by \texorpdfstring{$\II$}{the principal tensorial invariants}}

As stated in \cref{th:inv_orb}, the principal invariants quadruplet $\II$ is complete, meaning it characterises the orbits. However, another important property for a set of invariants is whether or not the given choice is optimal, in the sense that no invariant is redundant. When such a property is verified, one says that the set is minimal. In the given case, the quadruplet is actually strongly minimal. That is, removing an invariant adds ambiguity not only to (at least) two orbits, but to all of them:

\begin{lemmatend}
\label{lemm_min}%
\renewcommand{\tmp}[2]{\II_{\sigma\left (#1\right )}\left (#2\right )}%
Let $\II \equiv \left (\II_1, \II_2, \II_3, \II_4\right )$ be the principal tensorial invariants quadruplet. Then \underline{$\II$ is strongly minimal}. That is, a given triplet of principal invariants characterises no orbit:\\

$\forall \sigma \in \Sf_4,~ \forall \Fb \in \conf{\Tr\B}{\Tr\E},~ \exists \Fb' \in \conf{\Tr\B}{\Tr\E},$
\algn{
	&&\left (\tmp1\Fb,~\tmp2\Fb,~\tmp3\Fb\right ) &= \left (\tmp1{\Fb'},~\tmp2{\Fb'},~\tmp3{\Fb'}\right )\\
	&\text{and}& \tmp4\Fb &\neq \tmp4{\Fb'}
}

\startproof

\renewcommand{\tmp}{{\Fb'}}
	In each case, $\Fb'$ shall be expressed as $\Ab\circ\Fb$ with $\Ab \in \lpa{\Tr\E}{\Tr\E}$ a differential embedding (see \cref{def:conf}) with punctual shadow $\ab: \E \longto \E$ a differential embedding too. This means that $\Fb'=\Ab\circ\Fb$ will automatically be an element of $\conf{\Tr\B}{\Tr\E}$. There are four distinct cases, depending on the choice of $\sigma(4)$:
	\begin{enumerate}
		\it if $\II_{\sigma(4)} = {\ov\Fb}^*\gf_\vr^\vr$ then a uniform total dilation by $\lambda \in \RR^*_+$ suffices:
		\algn{
			\Ab &: \fun{\Tr\E}{\Tr\E}%
				{\bmat{\bmat{\ov x\\ y}\\\bmat{\delta \ov x\\\delta y}}}%
				{\bmat{\bmat{\lambda\cdot\ov x\\ \lambda\cdot y}\\\bmat{\lambda\cdot\delta \ov x\\\lambda\cdot\delta y}}}
		}
		Indeed, one then has:
		\byalgn{
			{\tmp}^*\gf_\vr^\vr 	&= \left (\lambda\cdot\Fb_\vr^\vr\right )^\star\cdot\gf_\vr^\vr\cdot\left (\lambda \cdot \Fb_\vr^\vr\right )\\
									&= \lambda^2\cdot{\Fb_\vr^\vr}^\star\cdot\gf_\vr^\vr\cdot{\Fb_\vr^\vr}\\
							   		&= \lambda^2\cdot{\Fb}^*\gf_\vr^\vr\\
			\tmp^*\varthetab	 	&= \left (\lambda\cdot\Fb\right )^{-1}\cdot\varthetab\cdot\left (\lambda \cdot \ov\Fb\right )\\
									&= \Fb^{-1}\cdot\varthetab\cdot\ov\Fb\\
									&= \Fb^*\varthetab
		}{
			\tmp^*\gammab 			&= \left (\lambda\cdot\Fb\right )^{-1}\cdot\gammab\cdot\left (\lambda \cdot \ov\Fb\right )\\
									&= \Fb^{-1}\cdot\gammab\cdot\ov\Fb\\
									&= \Fb^*\gammab\\
			{\Tr\varphi'}^*\gammab 	&= \left (\lambda\cdot\Tr\varphi\right )^{-1}\cdot\gammab\cdot\left (\lambda \cdot \Tr\ov\varphi\right )\\
									&= {\Tr\varphi}^{-1}\cdot\gammab\cdot\Tr\ov\varphi\\
									&= {\Tr\varphi}^*\gammab
		}

		\it if $\II_{\sigma(4)} = \Fb^*\varthetab$ then a purely macroscopic non-trivial rotation $\Rb \in \O(3)$ suffices:
		\algn{
			\Ab &: \fun{\Tr\E}{\Tr\E}%
				{\bmat{\bmat{\ov x\\ y}\\\bmat{\delta \ov x\\\delta y}}}%
				{\bmat{\bmat{\Rb\cdot\ov x\\ y}\\\bmat{\Rb\cdot\delta \ov x\\\delta y}}}
		}
		Indeed, one then has:
		\byalgn{
			{\tmp}^*\gf_\vr^\vr 	&= {\tmp_\vr^\vr}^\star\cdot\gf_\vr^\vr\cdot\tmp_\vr^\vr\\
									&= {\Fb_\vr^\vr}^\star\cdot\gf_\vr^\vr\cdot{\Fb_\vr^\vr}\\
									&= {\Fb}^*\gf_\vr^\vr\\
			\tmp^*\varthetab	 	&= {\Fb_\vr^\vr}^{-1}\cdot\varthetab\cdot\Rb\cdot\ov\Fb\\
									&\neq \Fb^{-1}\cdot\varthetab\cdot\ov\Fb\\
									&\neq \Fb^*\varthetab 
		}{
			\tmp^*\gammab 			&= \left [\Id-{\tmp_\vr^\vr}^{-1}\cdot\tmp_\hr^\vr\right ]\cdot\Gammabref\\
									&= \left [\Id-{\Fb_\vr^\vr}^{-1}\cdot\Fb_\hr^\vr\right ]\cdot\Gammabref\\
									&= \Fb^*\gammab\\
			{\Tr\varphi'}^*\gammab 	&= \left [\Id-{\left (\Tr\varphi'\right )_\vr^\vr}^{-1}\cdot{\Tr\varphi'}_\hr^\vr\right ]\cdot\Gammabref\\
									&= \left [\Id-{\left (\Tr\varphi\right )_\vr^\vr}^{-1}\cdot{\Tr\varphi}_\hr^\vr\right ]\cdot\Gammabref\\
									&= {\Tr\varphi}^*\gammab
		}

		\it if $\II_{\sigma(4)} = \Fb^*\gammab$ then one can choose $\tmp_\hr^\vr$ freely as it does not appear in $\ov\tmp^*\gb$, $\tmp^*\varthetab$ nor ${\Tr\varphi'}^*\gammab$. For example, setting $\tmp_\hr^\vr=\Fb_\vr^\vr\cdot\Rb$ for $\Rb \in \O(3)$ and keeping everything else the same suffices. Indeed, it yields $\tmp^*\gammab = \left (\Id-\Rb\right )\cdot\Gammabref$ which cannot be equal to $\Fb^*\gammab$ for all $\Rb$. One can verify that the $\tmp$ defined this way is still in $\conf{\Tr\B}{\Tr\E}$.

		\it if $\II_{\sigma(4)} = {\Tr\varphi'}^*\gammab$ then a total transformation obtained from a non-uniform (differential of a) microscopic translation $\tb^\vr$ suffices:
		\algn{
			\Ab &: \fun{\Tr\E}{\Tr\E}%
				{\bmat{\bmat{\ov x\\ y}\\\bmat{\delta \ov x\\\delta y}}}%
				{\bmat{\bmat{\ov x\\ y + \tb^\vr(\ov x)}\\\bmat{\delta \ov x\\\ider{\ov x}{\tb^\vr(\ov x)}\cdot\delta \ov x + \delta y}}}
		}
		
		Indeed, only ${\Tr\varphi'}_\hr^\vr$ is change and, similarly to the case of $\Fb^*\gammab$, this term does not appear in the formulae of $\ov\tmp^*\gb$, $\tmp^*\varthetab$ and $\tmp^*\gammab$. For ${\Tr\varphi'}_\hr^\vr$ one has:
		\algn{
			{\Tr\varphi'}^*\gammab 	&= \left [\Id-{{\Tr\varphi'}_\vr^\vr}^{-1}\cdot{\Tr\varphi'}_\hr^\vr\right ]\cdot\Gammabref\\
									&= \left [\Id-{{\Tr\varphi}_\vr^\vr}^{-1}\cdot{{\Tr\ab}_\vr^\vr}^{-1}\cdot\left [{\Tr\ab}_\hr^\vr\cdot{\Tr\varphi}_\hr^\hr +{\Tr\ab}_\vr^\vr\cdot{\Tr\varphi}_\hr^\vr\right ]\right ]\cdot\Gammabref\\
									&= \left [\Id%
									- {{\Tr\varphi}_\vr^\vr}^{-1}\cdot\ider{\ov x}{\tb^\vr(\ov x)}\cdot{\Tr\varphi}_\hr^\hr%
									- {{\Tr\varphi}_\vr^\vr}^{-1}\cdot{\Tr\varphi}_\hr^\vr\right ]\cdot\Gammabref\\
									&= {\Tr\varphi}^*\gammab - {{\Tr\varphi}_\vr^\vr}^{-1}\cdot\ider{\ov x}{\tb^\vr(\ov x)}\cdot{\Tr\varphi}_\hr^\hr\cdot\Gammabref
		}
		which differs from $\Tr\varphi^*\gammab$ as soon as $\tb^\vr$ is non-uniform.
	\end{enumerate}
\end{lemmatend}

Although this result may seem to close the subject of minimality, it does not as it treats tensors as indivisible blocks. Those tensors however may be split into smaller invariants, for example scalar ones. This results says nothing about the minimality of such a set of smaller invariants. That is, even though the principal tensorial invariants are not redundant, they may "overlap" and have common sub-invariants. In particular, choosing the value of some of the principal tensorial invariants may restrict the set of possible values of the others.\\

This discussion relates to the fact that $\II$ is not surjective. In fact, determining $\Imr\left (\II\right )$ is a hard problem. The first major issue lies not in the linear part, but in its relation with the punctual transformation. Indeed, in \cref{def:lpa} one requires $\ov\Fb=\Tr\ov\varphi$ and $\Fb_\vr^\vr = \left (\Tr\varphi\right )_\vr^\vr$. These equations yield several scalar integrability conditions of the form $\op{curl}=0$ \partxt{one per line of the matrices $\ov\Fb$ and $\Fb_\vr^\vr$} along with Schwarz-like compatibility conditions, which make determining the set $\Imr\left (\II\right )$ a tedious problem. Lastly, the second major issue comes from the additional requirements in \cref{def:conf} requiring $\Fb$ and $\varphi$ to be differential embeddings. The interested reader may refer to the similar work of \fullciteauthor{le1996determination}, where the curvature-free case for the $\Fb=\Fb^\er\cdot\Fb^\pr$ model $-$ where the invariants are the elastic metric and the torsion of the connection $-$ has been done in the settings of large transformations.\\

\renewcommand{\tmp}[1]{\nu\left ({#1}\right )}
A heuristic approach using degrees of freedom may yield useful insight on the problem. Setting $n = \dim(\BB)$ and fixing $X_0 \in \B$ and $x_0 = \varphi(X_0) \in \E$, the degrees of freedom of $\II(\Fb)$ and $\Orb{\Fb}$ verify (due to $\II$ being injective):
\algn{
	&\tmp{\Orb{\Fb}}& &=& &\tmp{\II(\Fb)}& &\le& &\tmp{\Gf_\vr^\vr} + \tmp{\Gammab} + \tmp{\Tr\varphi^*\gammab} + \tmp{\Thetab}&
}
\algn{
	\tmp{\Orb{\Fb}}
			  &= \tmp{\Fb} - \tmp{\Tr_x\Galf}\\
			  &=  \tmp{\Fb_\hr^\hr} + \tmp{\Fb_\hr^\vr} + \tmp{\Fb_\vr^\vr}
			  + \tmp{\Tr\varphi_\hr^\vr} - \tmp{\O(3)}\\
			  &= 3\times n + 3\times n + 3\times3 + 3\times n - 3 &\text{for $X_0$ and $x_0$ fixed}\\
			  &= 9n + 6\\
\\
	\tmp{\Gf_\vr^\vr} + \tmp{\Gammab} + \tmp{\Tr\varphi^*\gammab} + \tmp{\Thetab}\hspace{-30pt}&\\
			  &= \frac{3\times(3+1)}2 + 3\times n + 3\times n + 3\times n &\text{for $X_0$ fixed}\\
			  &= 6 + 9n
}
The difference being zero, one heuristically has that, when $X_0$ and $x_0$ are fixed, the invariant can be considered independent. However, when considering the degrees of freedom for a varying $X \in \B$ and $x=\varphi(X)\in\E$, compatibility conditions arise as stated above. This means that the block of which $\Fb$ is composed are not independent fields, meaning $\tmp{\Orb{\Fb}}$ is lower that this heuristic suggests. Since $\II$ is a bijection, this means the different invariants must also have compatibility conditions, in turn reducing $\tmp{\II(\Fb)}$. Beware that these heuristics only give an initial idea in order to help visualise the issue. Nevertheless, all this discussion seems to converge to the statement that frame invariance is probably expressible in term of smaller, less interdependent, invariants. A recent work of \citeauthor{tamarasselvame2011form} \cite[313]{tamarasselvame2011form} goes in that direction. By carefully using the related notion of form-invariance, they were able to determine that a form-invariant Lagrangian depending on a metric and a (compatible) connection must be expressible as a functional of the metric, torsion and curvature alone.

\VOID{
Even thought, in general, the problem is not yet solved, in the unidimensional case most integrability conditions are trivial and one has the following result:

\begin{lemm}
\TODO[cas ${\dim(\BB)=1}$]{
	Dans le cas $\dim(\BB)=1$, c'est à dire la poutre, la seule restriction pour être dans l'image est $\opp{curl}{\Thetab_x\cdot\delta\Sb_\ov x} = 0$ où $\delta\Sb_\ov x \in \Tr_\ov x \BB$ est quelconque et $\Thetab_x\cdot\delta\Sb_\ov x$ est vu comme un champ vectoriel sur $\B_\ov x$ \partxt{à valeur dans $\Vr_{\B_\ov x}\B = \Tr\B_\ov x$}. L'intégrabilité de $\ov\Fb$ et $\Fb_\vr^\vr\cdot\left (\hr_\Gammabref-\hr_{\Gammab_\mathrm{holo}}\right )$ est quand à elle directe du fait de la dimension $1$.
}
\end{lemm}

\begin{proof}
	
\end{proof}
}

\subsection{On the dependencies on the invariants}

As stated in \cref{frame_inv_intro}, we are interested in frame invariant functions of the placement map, which will be the properties of the material:

\renewcommand{\tmp}{{\omega}}
\begin{defi}{Frame invariant function}{fram_inv_fun}
	Let $\S$ be a set. A function $\tmp : \conf{\Tr\B}{\Tr\E} \longto \S$ is said to be \underline{frame invariant} if and only if:
	\quant{\forall \Fb \in \conf{\Tr\B}{\Tr\E},~\forall \Of \in \Tr\Galf,}{\tmp\left (\Fb\right ) &= \tmp\left (\Of \circ \Fb\right )}
\end{defi}

The set $\S$ is voluntarily generic and may take any mathematical form. In practice, however, it will most often be the set of reals \partxt{\eg{} $\tmp$ gives the total energy, the mass, the volume, etc.}, scalar fields \partxt{\eg{} $\tmp$ gives the energy field, the local dilatation, etc.} or tensor fields \partxt{\eg{} $\tmp$ gives the curvature or torsion fields}. Notice that, by construction, $\II$ is frame invariant with $\S \supset \Imr\left (\II\right )$. In fact, as a direct corollary of \cref{th:inv_orb}, one has the following result:\\

\begin{lemm}
	\label{lemm_fact_fram_inv}
	Let $\S$ be a set and $\tmp : \conf{\Tr\B}{\Tr\E} \longto \S$. Then, $\tmp$ is frame invariant if and only if there exists a (unique) $\widehat{\tmp} : \Imr\left (\II\right ) \longto \S$ such that:
		\algn{
			\tmp &= \widehat{\tmp}\circ\II
		}
\end{lemm}

Well-defined material properties are, as discussed in \cref{frame_inv_intro}, exactly those frame invariant functions. The last results then states that being frame invariant is equivalent to depending on $\II$. However, a property rarely depends on all four invariants. One can then wonders if the dependency on a given subset of $\II$ can be reformulated in term of a certain invariance. This question relates to the meaning of each of the four principal invariants.\\

Two of the four components of $\II$ are connections: $\Fb^*\gammab$ and $\Tr\varphi^*\gammab$. One may therefore wonder how exactly those two differ. This can be done by analysing what it means for a property to depend explicitly one one and not the other. First, having a dependency on $\Fb^*\gammab$ which can be simplified out, relates to the notion of holonomic dependency:

\begin{defi}{Holonomic dependency}{hol_dep}
	Let $\S$ be a set. A function $\tmp : \conf{\Tr\B}{\Tr\E} \longto \S$ is said to have a \underline{purely holonomic dependency} if and only if:
	\quant{\forall \Fb \in \conf{\Tr\B}{\Tr\E},}{\tmp\left (\Fb\right ) &= \tmp\left (\Tr\varphi\right )}
	where $\varphi : \B \longto \E$ implicitly refers to the shadow of $\Fb$ for the punctual projections.\\
\end{defi}
\begin{lemm}
	A function $\tmp : \conf{\Tr\B}{\Tr\E} \longto \S$ has a purely holonomic dependency if and only if there exists a (unique) $\widehat{\tmp}_\mathrm{holo}$ such that:
		\quant{\forall \Fb \in \conf{\Tr\B}{\Tr\E},}{
			\tmp(\Fb) &= \widehat{\tmp}_\mathrm{holo}\left(\Tr\varphi^*\gf_\vr^\vr,~\Tr\varphi^*\varthetab,~\Tr\varphi^*\gammab\right )
		}
\end{lemm}

One can interpret this results as stating that the dependency of $\tmp$ in $\Fb^*\gammab$ can be simplified iff $\tmp$ is blind towards the non-holonomy of $\Fb$. Or, in other words, that $\Fb^*\gammab$ captures this non-holonomy. One the other hand, one has that having a dependency in $\Tr\varphi^*\gammab$ which can be simplified out, relates to the notion of vectorial dependency:\\

\begin{defi}{$\op{Mat}$ operator and vectoriality}{mat}
	Let $\Fb^\mathrm{ref} \in \conf{\Tr\B}{\Tr\E}$ be a chosen placement with punctual shadow $\varphi_\mathrm{ref} : \B \longto \E$. Let $\Fb \in \conf{\Tr\B}{\Tr\E}$. The \underline{matrix of $\Fb$ in the reference $\Fb^\mathrm{ref}$} at $X \in \B$, denoted $\at{\opp[\Fb^\mathrm{ref}]{Mat}\Fb}_X$, is the matrix of $\Fb_X\cdot\left (\Fb^\mathrm{ref}\right )^{-1}_{\varphi^\mathrm{ref}(X)}$ in the holonomic frames on $\Tr_{\varphi^\mathrm{ref}(X)}\E$ and $\Tr_{\varphi(X)}\E$. Or, in a more concise way:
	\algn{
		\op[\Fb^\mathrm{ref}]{Mat} : \fun{\conf{\Tr\B}{\Tr\E}}{\smooth{\B}{\GL_{\dim(\B)}(\RR)}}{\Fb}%
		{%
			X \mapsto \opp[\mathrm{holo}]{Mat}{\Fb_X\cdot\left (\Fb^\mathrm{ref}\right )^{-1}_{\varphi^\mathrm{ref}(X)}}%
		}
	}
	
\hrulefill\\

	Let $\S$ be a set. An application $\tmp : \conf{\Tr\B}{\Tr\E} \longto \S$ is said to have a \underline{purely vectorial dependency} if and only if there exist $\tmp_\mathrm{vec} : \smooth{\B}{\GL_{\dim(\B)}(\RR)} \longto \S$ and $\Fb^\mathrm{ref} \in \conf{\Tr\B}{\Tr\E}$ such that $\tmp = \tmp_\mathrm{vec}\circ\op[\Fb^\mathrm{ref}]{Mat}$.
\end{defi}

\begin{lemmatend}
	\label{lemm_purely_vec}
	Let $\S$ be a set. An application $\tmp : \conf{\Tr\B}{\Tr\E} \longto \S$ has a purely vectorial dependency if and only if there exists a (unique) $\widehat{\tmp}_\mathrm{vec}$ such that:
			\quant{\forall \Fb \in \conf{\Tr\B}{\Tr\E},}{
				\tmp(\Fb) &=& \widehat{\tmp}_\mathrm{vec}\left (\Fb^*\gf_\vr^\vr,~\Fb^*\gammab,~\Fb^*\varthetab\right )
			}
\startproof
\renewcommand{\tmp}{{\omega}}

	The first result comes directly from \cref{th:inv_orb} and the fundamental theorem on homomorphisms. For the second result, fixing $X \in \B$ and committing it for brevity, let
		\algn{
			\opp[\Fb^\mathrm{ref}]{Mat}{\Fb} &:= \bmat{\opp[\Fb^\mathrm{ref}]{Mat}{\Fb_\hr^\hr}&\bold 0\\\opp[\Fb^\mathrm{ref}]{Mat}{\Fb_\hr^\vr}&\opp[\Fb^\mathrm{ref}]{Mat}{\Fb_\vr^\vr}}
		}
	then
	\algn{
		\opp[\Fb^\mathrm{ref}]{Mat}{\Fb^*\gf_\vr^\vr} &= \opp[\Fb^\mathrm{ref}]{Mat}{\Fb_\vr^\vr}^\mathrm{t}\cdot\opp[\Fb^\mathrm{ref}]{Mat}{\Fb_\vr^\vr}
	}
	is invertible. Let $\Qb\cdot\Ub$ be the polar decomposition of $\opp[\Fb^\mathrm{ref}]{Mat}{\Fb_\vr^\vr}$, where $\Qb \in \O_3(\RR)$ and $\Ub \in \Sr\nr^{++}(\RR)$. Then
		\byalgn{
			\Ub &:= \sqrt{\opp[\Fb^\mathrm{ref}]{Mat}{\Fb^*\gf_\vr^\vr}}\\
				&= \Qb^{-1}\cdot \opp[\Fb^\mathrm{ref}]{Mat}{\Fb_\vr^\vr}\\
				&= \opp[\Fb^\mathrm{ref}]{Mat}{\left (\Ab\cdot\Fb\right )_\vr^\vr}
		}{
			\Ab : \fun{\Tr\E}{\Tr\E}%
				{\bmat{	\bmat{\ov x\\ y}\\	\bmat{\delta\ov x\\\delta y}}}%
				{\bmat{	\bmat{\Qb^{-1}\cdot\ov x + \ov \tb\\ \Qb^{-1}\cdot y+\tb^\vr}\\	\bmat{\Qb^{-1}\cdot\delta\ov x\\\Qb^{-1}\cdot\delta y}}}
		}
		where $\Ab \in \Tr\Galf$ and the values of $\ov\tb$ and $\tb^\vr$ are not important. One can then set:\\
		
		\algn{
			\opp[\Fb^\mathrm{ref}]{Mat}{\ov\Ab\ov\Fb} 
					&= \opp[\Fb^\mathrm{ref}]{Mat}{\ov\Ab\cdot\varthetab^{-1}\cdot\varthetab\cdot\ov\Fb}\\
					&= \opp[\Fb^\mathrm{ref}]{Mat}{\varthetab^{-1}\cdot\Ab\cdot\varthetab\cdot\ov\Fb}\\
					&= \opp[\Fb^\mathrm{ref}]{Mat}{\varthetab^{-1}\cdot\Ab\cdot\Fb_\vr^\vr\cdot\Fb^*\varthetab}\\
					&= \opp[\mathrm{holo}]{Mat}{\varthetab^{-1}}\cdot%
							\opp[\mathrm{holo}]{Mat}{\Ab\cdot\Fb_\vr^\vr}\cdot%
							\opp[\Fb^\mathrm{ref}]{Mat}{\Fb^*\varthetab}\\
					&= \opp[\mathrm{holo}]{Mat}{\varthetab^{-1}}\cdot%
							\opp[\Fb^\mathrm{ref}]{Mat}{\Ab\cdot\Fb_\vr^\vr}\cdot%
							\opp[\Fb^\mathrm{ref}]{Mat}{\Fb^\mathrm{ref}\cdot\Fb^*\varthetab}\\
			\\
			\opp[\Fb^\mathrm{ref}]{Mat}{\left (\Ab\cdot\Fb\right )_\hr^\vr}
					&= \opp[\Fb^\mathrm{ref}]{Mat}{\Ab\cdot\Fb_\vr^\vr\cdot{\Fb_\vr^\vr}^{-1}\cdot\Fb_\hr^\vr}\\
					&= \opp[\Fb^\mathrm{ref}]{Mat}{\Ab\cdot\Fb_\vr^\vr\cdot\left (\hr_\Gammabref-\hr_{\Fb^*\gammab}\right )}\\
					&= \opp[\mathrm{holo}]{Mat}{\Ab\cdot\Fb_\vr^\vr}\cdot
							\opp[\Fb^\mathrm{ref}]{Mat}{\hr_\Gammabref-\hr_{\Fb^*\gammab}}\\
					&= \opp[\Fb^\mathrm{ref}]{Mat}{\Ab\cdot\Fb_\vr^\vr}\cdot
							\opp[\Fb^\mathrm{ref}]{Mat}{\Fb^\mathrm{ref}\cdot\left (\hr_\Gammabref-\hr_{\Fb^*\gammab}\right )}
		}
		
		Where $\varthetab^{-1}_x\cdot\Ab = \ov\Ab\cdot\vartheta^{-1}_x$ $-$ true for any $\Ab \in \Tr\Galf$ $-$ has been used. This defines uniquely $\opp[\Fb^\mathrm{ref}]{Mat}{\Ab\cdot\Fb}$ in term of $\Fb^\mathrm{ref}$, $\varthetab$, $\Gammabref$, $\Fb^*\gf_\vr^\vr$, $\Fb^*\gammab$ and $\Fb^*\varthetab$. Using the fact that the first three are given, one has an application
		\[
			\upsilon : \left (\Fb^*\gf_\vr^\vr, \Fb^*\gammab, \Fb^*\varthetab\right ) \mapsto \opp[\Fb^\mathrm{ref}]{Mat}{\Ab\cdot\Fb}
		\]
		One then sets $\widehat{\tmp}_\mathrm{Mat} = \tmp_\mathrm{Mat}\circ\upsilon$, which is well defined since $\Ab \in \Tr\Galf$ and $\alpha$ is supposed to be invariant under $\Tr\Galf \acts \conf{\Tr\B}{\Tr\E}$.\\
		
		Reciprocally, if such a map $\tmp_\mathrm{Mat}$ exists then one needs to construct
		\[
			\eta : \opp[\Fb^\mathrm{ref}]{Mat}{\Fb} \mapsto \left (\Fb^*\gf_\vr^\vr, \Fb^*\gammab, \Fb^*\varthetab\right )
		\]
		First, let:
		\algn{
			\opp[\mathrm{holo}]{Mat}{{\gf_\vr^\vr}^\mathrm{ref}}
				&:= \opp[\Fb^\mathrm{ref}]{Mat}{\Fb_\vr^\vr}^{\mathrm{t}} \cdot \opp[\mathrm{holo}]{Mat}{\gf_\vr^\vr} \cdot \opp[\Fb^\mathrm{ref}]{Mat}{\Fb_\vr^\vr}\\
				&:= \opp[\mathrm{holo}]{Mat}{{{{\Fb_\vr^\vr}^\mathrm{ref}}^{-1}}^\star\cdot{{\Fb_\vr^\vr}^{-1}}^\star\cdot\gf_\vr^\vr\cdot\Fb_\vr^\vr\cdot{{\Fb_\vr^\vr}^\mathrm{ref}}^{-1}}\\
			\opp[\mathrm{holo}]{Mat}{\gammab^\mathrm{ref}}
				&:= \opp[\Fb^\mathrm{ref}]{Mat}{\Fb}^{-1} \cdot \opp[\mathrm{holo}]{Mat}{\gammab} \cdot \opp[\Fb^\mathrm{ref}]{Mat}{\Fb}\\
				&:= \opp[\mathrm{holo}]{Mat}{\Fb^\mathrm{ref}\cdot\Fb^{-1}\cdot\gammab\cdot\ov\Fb\cdot\ov\Fb^\mathrm{ref}}
		}
		Since $\Fb^\mathrm{ref}$ is fully known, this defines uniquely $\at{{\gf_\vr^\vr}^\mathrm{ref}}_{\varphi^\mathrm{ref}(X)}$ and $\at{\gammab^\mathrm{ref}}_{\varphi^\mathrm{ref}(X)}$. One then has ${\Gf_\vr^\vr}_X = {{\Fb_\vr^\vr}^\mathrm{ref}}^\star\cdot{\gb_\vr^\vr}^\mathrm{ref}\cdot{\Fb_\vr^\vr}^\mathrm{ref}$ and $\Gammab_X = {{\ov\Fb}^\mathrm{ref}}^{-1}\cdot\gammab^\mathrm{ref}\cdot\ov\Fb^\mathrm{ref}$. Using the same process with $\varthetab$ instead of $\gammab$ one obtains $\Thetab_X$. This gives $\eta$. One then concludes by setting $\tmp_\mathrm{Mat} = \widehat{\tmp}_\mathrm{Mat}\circ\eta$. Notice that $\Tr\varphi^*\gammab$ could not be obtained, since the matrix $\opp[\Fb^\mathrm{ref}]{Mat}{\left (\Tr\varphi\right )_\hr^\vr}$ is missing.
\end{lemmatend}

The proof of \cref{lemm_purely_vec} is relatively tedious but mostly relies on the polar decomposition and its inverse. Notice that the reference placement map $\Fb_\mathrm{ref}$ does not appear explicitly. One can interpret \cref{lemm_purely_vec} as saying that $\Tr\varphi^*\gammab$ encodes the data of the holonomic coupling $\left (\Tr\varphi\right )_\hr^\vr$ not present in the vectorial part of $\Fb$.\\

Regarding the role of the last invariants $\Gf_\vr^\vr=\Fb^*\gf_\vr^\vr$ and $\Thetab=\Fb^*\varthetab$, the proof of \cref{lemm_min} can be extended to prove that:
\begin{enumerate}
	\it the dependency in $\Gf_\vr^\vr$ can be simplified out, iff the functional is invariant under arbitrary synchronous uniform transformations $\bmat{\Ab&\bold 0\\\bold0 & \Ab} \in \Tr\aff{\E}{\E}$ of $\Tr\E$ \partxt{with $\Ab \in \GL(3)$ uniform on $\EE$}.
	\it the dependency in $\Thetab$ can be simplified out, iff the functional is invariant under arbitrary asynchronous uniform micro-metric-preserving transformations $\bmat{\Ab&\bold 0\\\bold0 & \Rb} \in \Tr\aff{\E}{\E}$ of $\Tr\E$ \partxt{with $\Ab \in \GL(3)$ and $\Rb \in \O(3)$ uniform.}.
\end{enumerate}

The energy of a system is a well-defined property of the later; hence, it is frame invariant. Choosing the nature of a material (gaz, sand, fluid, solid, micromorphic media, etc) relates to choosing a subset of the set of principal invariants, or of their sub-invariants. Then, choosing a constitutive equation (rubber, metal, stone, etc.) relates to choosing the dependency of the energy on those invariants. This, along with the earlier discussion of this section, is to be compared with \textcite{zheng1994theory}, in which a correspondence is established between group of symmetries of a (classical) system and the set of invariants on which the energy can depend.

\subsection{On the micro-linear case}

\label{sect:micro-lin}

Of particular interest is the special case of micro-linear materials, which are materials where the micro-spaces have a linear structure. That is, micro-linear materials are vector bundles. In such a case, one requires the physically acceptable maps to preserve that linear structure, leading to the notion of micro-linear maps, as detailed in the following definition:

\begin{defi}{Micro-linear maps}{micro_lin}
	Let $\bundle\A\AA$ and $\bundle\D\DD$ be two \textbf{vector} bundles. A map $\Lb \in \lpa{\Tr\A}{\Tr\D}$ is said to be \textbf{micro-linear} if and only if:
	\begin{enumerate}
		\it the punctual shadow $\ell : \A \longto \D$ of $\Lb$ is \textbf{linear}.
			$$\ell \in \l\A\D$$
		\it for all $\ub \in \Tr_a\A$ and every trivialising affine frame on $\Tr\A$ and $\Tr\D$, the vectorial coordinate of $\Lb_a\cdot\ub \in \Tr_{\ell(a)}\D$ is \textbf{linear} in the vertical coordinate of $a \in \A$
	\end{enumerate}
	
	The set of micro-linear elements of $\lpa{\Tr\A}{\Tr\D}$ are denoted:
	\algn{
		\llpa{\Tr\A}{\Tr\D} \subset \lpa{\Tr\A}{\Tr\D}
	}
	
\hrulefill

	Le $\bundle\B\BB$ be a material bundle. One has the following notations:
	\algn{
		\conf[\text{holo}]{\Tr\B}{\Tr\E} :=& \conf{\Tr\B}{\Tr\E} \cap \Tr\aff{\B}{\E}\\
		\lconf{\Tr\B}{\Tr\E} :=& \conf{\Tr\B}{\Tr\E} \cap \llpa{\Tr\B}{\Tr\E}
	}
	Where the last notation is only defined when $\bundle\B\BB$ is a vector bundle. That is, when the material has a linear micro-structure.
\end{defi}

Physically, the micro-linearity can be interpreted as stating that each microscopic space has a point (\eg{} its center of mass) which is preserved by the placement. This assumption has been used in the fundamental works of \fullciteauthors{toupin1964theories,mindlin1964micro,eringen1964nonlinear}{eringen1998microcontinuum,toupin1962elastic}. In (linear trivialising) coordinates, being micro-linear means that, setting $\bmat{\bmat{\ov x \\ y}\\\bmat{\delta\ov x\\\delta y}} = \Fb\cdot\bmat{\bmat{\ov X \\ Y}\\\bmat{\delta\ov X\\\delta Y}}$, one has that $y$ and $\delta y$ are linear in $Y$ (instead of affine) while $\ov x$ and $\delta \ov x$ are still blind towards it.\\

The micro-linearity of the placement map $\Fb \in \llpa{\Tr\B}{\Tr\E}$ has a noticeable impact on the invariants. First, one has that the connection $\Gammab := \Fb^*\gammab := \Fb^{-1}\cdot\gammab_\varphi\cdot\ov\Fb$ is linear. Furthermore, the solder form $\Thetab := \Fb^*\varthetab := {\Fb_\vr^\vr}^{-1}\cdot\varthetab\cdot\ov\Fb$ is constant in the vertical punctual coordinate \partxt{this is true for any $\Fb \in \conf{\Tr\B}{\Tr\E}$}. Indeed, coordinate-wise one has:\\

\byalgn[.\hspace{-36.4pt}]{
		\Gammab &: \fun{\B\times_\BB\Tr\BB}{\Tr\B}%
					   {\bmat{\bmat{\ov X\\ Y}\\\bmat{\ov X\\ \delta \ov X}}}%
					   {\bmat{\bmat{\ov X\\ Y}\\\bmat{\delta \ov X\\ \at{{\Fb_\vr^\vr}^{-1}}_{\ov X}\cdot\at{\Fb_\hr^\vr}_{\bmat{\ov X\\ Y}}\cdot\delta \ov X}}}
	}{
		\Thetab &: \fun{\B\times_\BB\Tr\BB}{\Tr\B}%
					   {\bmat{\bmat{\ov X\\ Y}\\\bmat{\ov X\\ \delta \ov X}}}%
					   {\bmat{\bmat{\ov X\\ Y}\\\bmat{\bold0\\ \at{{\Fb_\vr^\vr}^{-1}}_{\ov X}\cdot\at{\Fb_\hr^\hr}_{\ov X}\cdot\delta \ov X}}}
	}

where, since $\Fb \in \lpa{\Tr\B}{\Tr\E}$, $\Fb_\hr^\hr$ and $\Fb_\vr^\vr$ are blind towards $Y$ and, by assumption, $\Fb_\hr^\vr$ is linear in $Y$. This means that $\Gammab$ and $\Thetab$ are respectively the linear and constant part of the affine connection $\Gammab-\Thetab$ and can, in particular, canonically be obtained from it. From \cref{th:data_equiv} one has that $\Gf$ is equivalent to $\left (\Gf_\vr^\vr, \Gammab - \Thetab\right )$. In the case of micro-linear placement maps, one therefore has the following corollary:

\begin{coro}{Data equivalence (micro-linear case)}{data_equiv}
	Let $\Fb \in \lconf{\Tr\B}{\Tr\E}$ be unknown. Then, the data of $\Gf = \Fb^\star\cdot\gf\cdot\Fb$ is equivalent to the joint data of:
	\begin{enumerate}
		\it the micro-metric $\Gf_\vr^\vr : \Vr\B \longto \Vr^\star\B$
		\it the connection $\Gammab : \B\times_\BB\Tr\BB\longto \Tr\B$
		\it the solder form $\Thetab: \B\times_\BB\Tr\BB\longto \Vr\B$
	\end{enumerate}
\end{coro}

The second major impact concerns the material holonomic connection $\Tr\varphi^*\gammab$. Indeed, looking at the proof of \cref{th:galf_as_invar}, one realises that the condition $\Tr\ab^*\gammab=\gammab$ was only used to ensure the homogeneity of the punctual microscopic translation $\tb^\vr$. However, by definition, this very translation is zero when the map is micro-linear. As a direct consequence, one has the following corollary:

\begin{coro}{$\Tr\Galf \cap \llpa{\Tr\E}{\Tr\E}$ as the intersection of stabilisers}{galf_as_invar}
	The first-order micro-linear Galilean group can be given by the following equations:
	\algn{
		\Tr\Galf \cap \llpa{\Tr\E}{\Tr\E} &= \setof{\Ab \in \llpa{\Tr\E}{\Tr\E}}{\mat{%
		  				\Ab~\text{inversible},~~\ab~\text{inversible},\\
				  		\Ab^*\gf=\gf
		  		}}
	}
	where $\ab : \E \longto \E$ implicitly refers to the punctual shadow of $\Ab$ in $\l{\Tr\E}{\Tr\E}$.
\end{coro}

\begin{proof}
	From \cref{th:galf_as_invar}, one has $\Ab^*\gf=\gf$, $\Ab^*\gammab=\gammab$ and $\Tr\ab^*\gammab=\gammab$. From the discussion above, one has that, for $\Ab \in \llpa{\Tr\E}{\Tr\E}$, the later is immediate and, since $\Ab^*\gammab$ is canonically obtainable from $\Ab^*\gf$, $\Ab^*\gf=\gf$ implies $\Ab^*\gammab=\gammab$.
\end{proof}

By combining those two corollaries, one obtains the following crucial corollary:

\renewcommand{\tmp}{\lconf{\Tr\B}{\Tr\E}}
\begin{coro}{Tensorial invariants (micro-linear case)}{inv_orb}
	Let $\bundle\B\BB$ be a \textbf{vector} bundle. Let $\S^+\left (\B\right )$ be the set of pseudo-metrics on $\B$ and let $\Gf$ be the following application:
	\algn{
		\Gf &: \fun%
		{\tmp}%
		{\S^+\left (\B\right )}%
		{\Fb}%
		{\raisebox{-0.75em}{$\begin{matrix*}[l]%
			\Gf\left (\Fb\right )%
				&=& \Fb^*\gf\\%
				&=& \Tr\varphi^*\gf%
			\end{matrix*}$}%
		}
	}
	The application $\Gf$ is \underline{constant on the orbits} of the action $\Tr\Galf \acts \tmp$ and provides, using the fundamental theorem on homomorphisms, the following bijection:
	\algn{
		\faktor{\tmp}{\Tr\Galf} \quad &\xhooktwoheadrightarrow{\quad\Gf\quad}& \Imr\left (\Gf\right )
	}
	
	In other words, \underline{$\Gf$ is complete}, meaning it characterises the orbits:
	\algn{
		&\forall \left (\Fb_1, \Fb_2\right ) \in  {\tmp}^2,& &\Orb{\Fb_1}=\Orb{\Fb_2}& &\iff& &\Gf\left (\Fb_1\right ) = \Gf\left (\Fb_2\right )&
	}
\end{coro}

\begin{proof}
	First, one must notice that $\Tr\Galf$ does not send $\tmp$ into $\tmp$ but rather into $\conf{\Tr\B}{\Tr\E}$. The notation $\faktor{\tmp}{\Tr\Galf}$ therefore needs to be specified as it does not correspond to a group action any-more. As before, one defines it as
	$$\faktor{\tmp}{\Tr\Galf} = \setof{\Orb{\Fb}}{\Fb\in\tmp}$$
	The orbit can then be generalised in two ways:
	\byalgn[.\hspace{-9.5pt}\text{or}\hspace{-9.5pt}]{
		\Orb{\Fb}_{\vr1} = \setof{\Of\cdot\Fb}{\Of\in\Tr\Galf}
	}{
		\Orb{\Fb}_{\vr2} = \Orb{\Fb}_{\vr1} \cap \tmp
	}
	This leads to two equivalent definition of the quotient, in the sense that there is a natural bijection between the two. Furthermore, if $\Fb_1 = \Of\cdot\Fb_2$ with $\Of \in \Tr\Galf$ and $\left (\Fb_1, \Fb_2\right ) \in \tmp$ then $\Of = \Fb_1\cdot\Fb_2^{-1} \in \llpa{\Tr\E}{\Tr\E}$. This implies that one has:
	$$\faktor{\tmp}{\Tr\Galf} = \faktor{\tmp}{\Tr\Galf\cap\llpa{\Tr\E}{\Tr\E}}$$
	The proof then follows the same sketch as the proof of \cref{th:inv_orb}:
	\algn{
		&&			\Orb{\Fb_1} &= \Orb{\Fb_2}\\
		&\iff& 		\opp[\Tr\Galf\cap\llpa{\Tr\B}{\Tr\E}]{Orb}{\Fb_1} &= \opp[\Tr\Galf\cap\llpa{\Tr\B}{\Tr\E}]{Orb}{\Fb_2}\\
		&\iff& 		\widehat{\Fb} := \Fb_1\cdot\Fb_2^{-1} &\in \Tr\Galf\cap\llpa{\Tr\B}{\Tr\E}\\
		&\iff& 		\widehat{\Fb}^*\gf &= \gf\\
		&\iff& 		\Fb_1^*\gf &= \Fb_2^*\gf
	}
\end{proof}

This means that any frame-invariant functional and, in particular, any energy formulated in a micro-linear model can be expressed as a functional of the pseudo-metric $\Gf : \Tr\B \longto \Tr^\star\B$. Equivalently (\cref{coro:data_equiv}), it can be expressed as a functional of the three invariants of the micro-linear case $\left (\Gf_\vr^\vr, \Thetab, \Gammab\right )$. This result should be compared with its classical analogue stating that all classical energies are functionals of the Cauchy-Green tensor $\Gb : \Tr\BB \longto \Tr^\star\BB$ \parencite[275,283]{GON}.\\

Using the vocabulary of the previous sub-section, one notices that the result implies that every frame-invariant functional of the micro-linear placement map has a purely vectorial dependency. Note that, contrary to the holonomic case, the converse is false as having a purely vectorial dependency does not mean one is expressible as a functional of the micro-linear part. This is because the translational part of the coupling $\Fb_\hr^\vr$ may have some contribution. Beware that the fact that the holonomic connection $\Tr\varphi^*\gammab$ is not a principal invariant in the micro-linear case does not mean it does not contribute nor that it is equal to $\Gammab = \Fb^*\gammab$. That being said, it does mean that $\Tr\varphi^*\gammab$ will be a functional of the other three invariants $\left (\Gf_\vr^\vr, \Thetab, \Gammab\right )$.\\

As mentioned above, the micro-linear assumption is built-in the models of \fullciteauthors{toupin1964theories,mindlin1964micro,eringen1964nonlinear}{eringen1998microcontinuum,toupin1962elastic}. As an example, the principal invariants $\left (\Gf_\vr^\vr, \Thetab, \Gammab\right )$ can be compared to the \say{set of strain measures} obtained in \cite[p. 15]{eringen1998microcontinuum}. By carefully matching the definitions, one can see that the three strain measures are as follows:
	\begin{enumerate}
		\it the \say{deformation tensor}, which \cite{eringen1998microcontinuum} denoted $\Cf$, is
			\[
				\Thetab\cdot\Tr\pi_\B : \Tr\B \longto \Vr\B
			\]
		\it the \say{micro-deformation tensor}, which \cite{eringen1998microcontinuum} denoted $\C$, is
			\[
				\Gf_\vr^\vr : \Vr\B \longto \Vr^\star\B
			\]
		\it the \say{wryness tensor}, which \cite{eringen1998microcontinuum} denoted $\Gammab$, is
			\[
				\hr_\Gammabref - \hr_\Gammab = {\Fb_\vr^\vr}^{-1}\cdot\Fb_\hr^\vr : \Tr\B \longto \Vr\B
			\]
	\end{enumerate}

One notices that theses \say{strain measures} are therefore equivalent to the three principal tensorial invariants of the micro-linear case. However, in \cite{eringen1998microcontinuum}, the placement map is also holonomic, \ie{} $\Fb = \Tr\varphi$. As a direct consequence, $\Fb^*\gammab=\Tr\varphi^*\gammab$. When generalising the results to the entire set $\conf{\Tr\B}{\Tr\E}$, the two connections start to differ and the dependency in $\left (\Fb^*\gammab, \Tr\varphi^*\gammab\right )$ may take a lot of different forms. Based on the previous section, one sees that, among all the possible dependencies on $\II$, the generalisation may in particular be: purely vectorial \partxt{\ie{} using $\Fb^*\gammab$ as written above}, purely holonomic \partxt{\ie{} using $\Tr\varphi^*\gammab$ as the connection} or have a generic mixed dependency. Depending on the case, the material will be different. In particular, the second case will always render the functional blind to any curvature of the placement map \partxt{\ie{} disclinations of the material}.\\

As described in the introduction of this paper, several other models exist, to which this paper can also be partially compared, albeit in a less direct manner. Among those, some kept $\BB$ as the base space but weakened the regularity of $\ov\varphi$. As an example \cite{rakotomanana1998contribution} used a weakly-continuous map, allowing some discontinuities, while \cite{Kleinert:2008} used a multi-valued map, an alternative interpretation of the non-holonomy\footnote{One can obtain such a multi-valued map using the parallel transport of $\Gammab$. When $\Fb=\Tr\varphi$ this yields $\varphi$, but is multivalued otherwise.} of $\Fb$. Another approach is to make extensive use of germ theory and group theory, as in \cite{epstein2007material}.

\section{Conclusion}

In this paper, a generalised notion of continuum media has been introduced along with a notion of placement map. On the ambient space $\bundle\E\EE$, a generalisation $\Tr\Galf$ of the classical Galilean group $\Tr\Gal$ has been constructed. From the invariance of this group and some physical assumption on the material, a set of tensorial invariants has been computed. The result is a model which is:
\begin{enumerate}
	\it entirely governed by its first-order placement map $\Fb$, with no other hidden degree of freedom.
	\it non-holonomic \partxt{\ie{} $\Fb\neq\Tr\varphi$ in general}, hence potentially displaying curvature.
	\it geometrically exact, to be understand as in \cite{neff2007geometrically, meier2018geometrically}. This means that no approximation have been done once the mathematical definitions are given.
	\it valid in small and large transformations, as a direct consequence of the geometrical exactness.
	\it valid for solids, fluids and gases as $\B$ and $\Fb$ are left as generic as possible.
	\it valid for any macroscopic dimension and multiple structures. In particular, beams and rings can be model by setting $\dim\left (\BB\right )=1$; while plates, membranes and shells would have $\dim\left (\BB\right ) = 2$. 
\end{enumerate}
 
Furthermore, the nomenclature defined in this papers allows for a rigorous numerical implementation which is:
\begin{enumerate}
	\it strongly typed. Points belong to a given space, functions maps a given space to another, composition of function is only valid if the spaces are the same, etc.
	\it straightforward. Once a given coordinate system is fixed, everything can be implemented in term of field of matrices over a real domain and, most of the time, computations can be made at a fixed point; hence, no advanced system is required as would be the case for group theory, germ theory or even multi-valued functions.
	\it fast. Since most of the computation are linear algebra, computations are usually as fast as would their classical, purely macroscopic, counterpart be.
\end{enumerate}

Having exhaustively described our model, future work may focus, among other things, on specifying the energy and computing the statical equilibrium states for some particular materials; establishing more (partial) correspondences with other models or adding time to the model and studying its dynamics.

\ifredboxes{
	%\sep[not written after this line]
}
\VOID{\section{TODO}}

\subsection*{Acknowledgements}

This work is partially funded by Centre Henri Lebesgue, program ANR-11-LABX-0020-0.\\

\noindent We wish to express our sincere gratitude towards Lalaonirina Rakotomanana-Ravelonarivo for his continuous support and guidance throughout the elaboration of this work.\\

\noindent We also wish to thank all attendees of the GDR CNRS n°$2043$ \say{Géométrie Différentielle et Mécanique} which, through numerous exchanges, helped us to clear out our vision of the different topics related to this work.

\tcbstoprecording
\end{document}